\documentclass[11pt,reqno,UTF8,twoside]{amsart}
\usepackage{amssymb}
\usepackage[dvips]{epsfig}
\usepackage{booktabs}
\usepackage{multirow}
\usepackage{bbm}
\usepackage{mathrsfs}
\usepackage{xltabular}
\usepackage{tabularx}
\usepackage{makecell} 
\usepackage{ragged2e} 
\usepackage{graphicx}
\usepackage{color}
\usepackage{amsmath}
\allowdisplaybreaks
\usepackage{amssymb}
\usepackage{amsfonts}
\usepackage{geometry}
\usepackage{xstring}
\usepackage{amsthm}
\usepackage[normalem]{ulem}    
\usepackage{bm}
\usepackage{cite}
\usepackage{tikz}
\makeatletter
\@namedef{subjclassname@2020}{\textup{2020} Mathematics Subject 
	Classification}
\makeatother

\newcommand{\Z}{\mathbb{Z}}
\newcommand{\C}{\mathbb{C}}

\newcommand{\R}{\mathbb{R}}
\newcommand{\T}{\mathbb{T}}
\newcommand{\E}{\mathbb{E}}

\newcommand{\spec}{{\rm Spec}}

\newcommand{\hop}{\rm hop}

\newcommand{\Dhop}{D_{\rm hop}}
\newcommand{\Dbig}{D_{\rm big}}

\newcommand{\La}{\Lambda}

\renewcommand{\P}{\mathbb{P}}

\newcommand{\baromega}{\overline{\omega}}

\newcommand{\hatf}{\hat{f}}
\newcommand{\hatxi}{\hat{\xi}}
\newcommand{\baralpha}{\bar{\alpha}}

\newcommand{\mfN}{\mathfrak{N}}
\newcommand{\mfS}{\mathfrak{S}}

\newcommand{\mcB}{\mathcal{B}}

\newcommand{\mcC}{\mathcal{C}}
\newcommand{\mcE}{\mathcal{E}}
\newcommand{\mcF}{\mathcal{F}}

\newcommand{\mcS}{\mathcal{S}}

\newcommand{\mcK}{\mathcal{K}}
\newcommand{\mcT}{\mathcal{T}}

\newcommand{\mcO}{\mathcal{O}}
\newcommand{\mcG}{\mathcal{G}}
\newcommand{\mcW}{\mathcal{W}}

\newcommand{\mcN}{\mathcal{N}}

\newcommand{\wta}{\widetilde{a}}
\newcommand{\wtE}{\widetilde{E}}
\newcommand{\wha}{\widehat{a}}
\newcommand{\whL}{\widehat{L}}
\newcommand{\whD}{\widehat{D}}

\newcommand{\wtV}{\widetilde{V}}
\newcommand{\wtlambda}{\widetilde{\lambda}}

\newcommand{\wtS}{\widetilde{S}}
\newcommand{\wtB}{\widetilde{B}}

\renewcommand{\Im}{{\rm Im}}
\renewcommand{\Re}{{\rm Re}}
\newcommand{\dist}{{\rm dist}}

\newcommand{\supp}{{\rm supp}}

\newtheorem{thm}{Theorem}[section]
\newtheorem{lem}[thm]{Lemma}

\newtheorem{Def}{Definition}[section]

\newtheorem{rmk}{\bf Remark}[section]

\theoremstyle{definition}

\newtheorem{ctex}[thm]{Counterexample}

\newtheorem{claim}[thm]{\bf Claim}

\numberwithin{equation}{section}

\usepackage[colorlinks=true,pdfstartview=FitV,linkcolor=magenta,citecolor=cyan]{hyperref}


\usepackage{caption}

\begin{document}
\title[1D Anderson-Bernoulli model with long-range hopping]{Localization and unique continuation for the Anderson-Bernoulli model with long-range hopping on $\Z$}

\author[Liu]{Shihe Liu}
\address[S. Liu] {School of Mathematical Sciences,
Peking University,
Beijing 100871,
China}
\email{2301110021@stu.pku.edu.cn}

\author[Shi]{Yunfeng Shi}
\address[Y. Shi] {School of Mathematics,
Sichuan University,
Chengdu 610064,
China}
\email{yunfengshi@scu.edu.cn}

\author[Zhang]{Zhifei Zhang}
\address[Z. Zhang] {School of Mathematical Sciences,
Peking University,
Beijing 100871,
China}
\email{zfzhang@math.pku.edu.cn}

\begin{abstract}
  In this paper, we study Anderson localization near the spectral edge for the Anderson–Bernoulli model on $\Z$ with long-range hopping. When the hopping has a rational Laurent symbol, a quantitative version of the unique continuation principle can be proved, and localization occurs. For the unique continuation in the general case, we give some counterexamples and prove a weaker result for hopping that decays faster than exponential rate. To the best of our knowledge, this is the first localization result for the long-range Anderson model with pure Bernoulli potentials.
\end{abstract}

\maketitle

\tableofcontents

\section{Introduction}

\subsection{Background}

In 1958, Anderson \cite{And58} introduced the random tight-binding model
\begin{equation}\label{standard Anderson model}
	H=\Delta_{\rm free}+\lambda V,
\end{equation}
where $\Delta_{\rm free}$ is the free Laplacian on $\Z^d$ and $V$ is a random potential. An important feature of this model is  the presence of Anderson localization (AL), meaning that the spectrum is pure point in some energy region and the corresponding eigenfunctions decay exponentially. When the coupling strength $\lambda$ is small, localization holds near the spectral edges; for large $\lambda$, it may occur throughout the spectrum.

For the case where the hopping is the free Laplacian (i.e., a short-range interaction), many localization results have been obtained when $V$ is given by i.i.d. random variables with a \textbf{continuous} distribution (e.g., H\"older continuous or absolutely continuous). The important approaches in this setting include the transfer matrix method in one dimension, and in higher dimensions either the multi-scale analysis (MSA) developed by Fr\"ohlich and Spencer \cite{FS83}, or the fractional moment method (FFM) introduced by Aizenman and Molchanov \cite{AM93}.

Given the abundance of localization results for  \eqref{standard Anderson model},  another central question is whether localization persists when the free Laplacian is replaced by a long-range hopping operator $T$. This problem is of significant interest to both physicists and mathematicians, and has attracted considerable attention over the years. For instance, Yeung-Oono  \cite{YO87} performed numerical physical experiments for various long-range hopping operators  on the one-dimensional lattice $\Z$ (e.g., the  \textbf{exponential  decay $e^{-|n|/100}$, sub-exponential decay  $e^{-\sqrt{|n|}}$, and the power-law decay $|n|^{-3}$}) coupled with a random  potential $V$ having either a \textbf{uniform distribution} on intervals or \textbf{two-point Bernoulli} variables. Localization was observed numerically, suggesting that certain localization results should be mathematically  provable  for long-range Anderson type models.

In the long-range setting, when the random potential has a \textbf{continuous} distribution, there is a substantial body of rigorous results on localization; see, e.g., \cite{SS89, Wan91, AM93, Kle93, Gri94, JM99, Klopp02, Shi21, SWY25}. In contrast, for the \textbf{Bernoulli} potential case, much less is known, even for \eqref{standard Anderson model} with the free Laplacian. The main difficulty is that the lack of regularity in the potential distribution invalidates an a priori Wegner estimate, which plays an essential role in the proof of localization via the MSA method.

For the Anderson-Bernoulli model (ABM) with the free Laplacian, the progress so far can be listed as follows. In one dimension, the obstacle can be overcome by transfer matrix methods, such as the Furstenberg/LePage approach; see \cite{CKM87, GM89, KLS90 ,SVW98, DSS02, BDFGVWZ, JZ19, GK21}. In higher dimensions, however, transfer matrix methods fail. A breakthrough came with Bourgain's work \cite{Bou04}, which introduced the ``free site argument'' and Sperner's Lemma  to prove AL near the spectral edge for ABM on $\Z^d$ with some alloy-type potential, and later, together with Kenig \cite{BK05}, revealed a deep connection between localization and the unique continuation principle (for ABM on $\R^d$). In particular, \cite{BK05} showed that, \textbf{unlike the continuous case, when the potential contains Bernoulli components, the analytic properties of the differential operator play a crucial role}. Subsequently, Ding and Smart \cite{DS20}, using techniques from \cite{BLMS}, developed a probabilistic unique continuation principle for ABM on $\Z^2$ and established localization near the spectral edge. Recently, Li and Zhang \cite{LZ22} extended this result to $\Z^3$; the case $\Z^d$ for $d\ge 4$ remains open. The results of \cite{DS20} and \cite{LZ22} were later generalized to  the case of non-stationary distributions by  \cite{Hur26a,Hur26b}. For attempts in dimension $d\ge 4$, we mention the work of Imbrie \cite{Imb21}, which  proved AL for  Anderson type models with a discrete distribution taking $N$ values (for $N$ sufficiently large) at large disorder. Very recently,  Liu-Shi-Zhang   \cite{LSZ26}  established  AL for a hierarchical ABM in arbitrary dimensions.

All the works mentioned above concern ABM with the free Laplacian. To the best of our knowledge, no result  is  available for ABM with long-range hopping, even in the one-dimensional lattice $\Z$ (for the long-range hopping case, the transfer matrix methods are still not applicable). A related work is \cite{LSZ25}, in which we extended the method of \cite{Bou04} to the  long-range hopping case  in arbitrary dimensions, but the potential there is of alloy type rather than  the  Bernoulli one.  

In the present paper, we establish a (deterministic) unique continuation principle for a class of long-range Schr\"odinger operators on $\Z$, and provide counterexamples showing that such a principle cannot hold for all hopping. These counterexamples reveal that the unique continuation property is intimately connected to the singularities of the Laurent symbol of the hopping, indicating that the global structure of the hopping, rather than local information, plays a decisive role in this issue. The unique continuation result proved here, together with the MSA (which also contains some novelty in the present work), enables us to obtain a corresponding localization results  for the one-dimensional ABM with hopping  whose Laurent symbol is rational. This is sufficient to give an affirmative answer to the Anderson-Bernoulli localization conjecture in \cite{YO87} with a hopping decaying like  $e^{-|n|/100}$. Thus, in this paper, certain analytic features of the hopping play a key role, consistent with the insight of \cite{BK05}.

We also emphasize that, while the unique continuation result is obtained only in one dimension, the method in this paper is dimension-independent and can be extended to arbitrary dimensions. This may open up a promising avenue for investigating higher-dimensional long-range models.

Finally, we note that for hopping with a sub-exponential decay, power-law decay, or exponential decay with a more general Laurent symbol, the localization problem for the one-dimensional ABM remains open. Moreover, the extension of localization result  for long-range ABM to higher dimensions appears to be a challenging open problem.

\subsection{Main results}
Consider the following Schr\"odinger operator on $\Z$:
\begin{equation}\label{longrange schrodinger operator}
	H(\omega)=T+\lambda V(\omega), \ \lambda > 0,
\end{equation}
where $\omega$ lies in some probability space $(\Omega,\mathbb{P},\mathcal{F})$, and $V(n)(\omega)$, $n\in\Z$, are i.i.d. Bernoulli potentials on $\Z$ satisfying 
\begin{equation}\label{Bernoulli potential}
	\mathbb{P}(V(n)=0)=\mathbb{P}(V(n)=1)=\frac{1}{2}, \ \forall n\in\Z.
\end{equation}
For simplicity we take the probability space $\Omega=\{0,1\}^{\Z}$ and write $V(n)(\omega)=\omega_n$. The hopping operator $T$ is of convolution type with
\begin{equation}\label{hopping operator}
	T(m,n)=f(m-n),\ \forall m,n\in\Z
\end{equation}
for some complex-valued function $f:\Z\to\C$. The corresponding Fourier symbol and the Laurent symbol of $T$ are given respectively by
\begin{align}
	\label{Fourier symbol} \hat{f}(\theta)=\sum_{n\in \Z} f(n)  e^{2\pi i n\theta}, \ \theta\in \T; \\
	\label{Laurent symbol} F(z)=\sum_{n\in \Z} f(n) z^n,\ z\in\C_{\times}.
\end{align} 
Here $\mathbb{T}=\mathbb{R}/\mathbb{Z}$ is the one-dimensional torus, and $\mathbb{C}_{\times}=\mathbb{C}\setminus\{0\}$ is the punctured complex plane. In this paper, we mainly study those $T$ satisfying the following two assumptions:
\begin{itemize}
	\item[\textbf{(A1)}] $\hat{f}$ is real-valued;
	\item[\textbf{(A2)}] $\hat{f}$ is non-constant and real-analytic in $\theta\in\mathbb{T}$.
\end{itemize}
It is easy to check that assumption \textbf{(A1)} is equivalent to the self-adjointness of $T$, and assumption \textbf{(A2)} is equivalent to the statement that $T$ is bounded, not a scalar multiple of the identity, and the hopping decays at least exponentially fast, i.e., 
\begin{equation}\label{exp decay hopping}
	|f(n)|\leq C_{T}\, e^{-c_T|n|}
\end{equation}
for some constants $C_T,c_T>0$ depending on $T$. \\

It is well known (see, e.g., \cite[Lemma 2.1.2]{Kri08} and \cite[Theorem 3.9]{Kir08}) that, under the above setting, with probability one we have
\begin{equation}\label{as spectrum}
	\spec(H(\omega)) = \operatorname{range}(\hat{f}) + \{0,\lambda\}.
\end{equation}
Since the localization property of the spectrum is invariant under (real) affine transformations of the operator, without loss of generality we may assume that
\begin{itemize}
	\item[\textbf{(A3)}] $\sup_{\theta\in \mathbb{T}} \hat{f}(\theta) = 1$ and $\inf_{\theta\in \mathbb{T}} \hat{f}(\theta) = 0$.
\end{itemize}
Then \eqref{as spectrum} becomes $[0,1] \cup [\lambda, 1+\lambda]$.

Now denote by $\mathscr{R}$ the class of hopping operators with rational Laurent symbol:
\begin{equation}\label{rational hopping class}
	\mathscr{R}:= \left\{ T : T \text{ satisfies \textbf{(A1)}, \textbf{(A2)} and its Laurent symbol } F(z) \text{ is a rational function} \right\}.
\end{equation}
Our main result is stated as follows:
\begin{thm}\label{Localization near the edge band}
For any $T\in\mathscr{R}$ satisfying \textbf{\textup{(A3)}}, there exists $\delta>0$, depending on $T$ and $\lambda$, such that $H$ exhibits Anderson localization on $[0,\delta]$ almost surely.
\end{thm}

We make the following remarks:
\begin{rmk}
(1) Since the Fourier symbol is just the Laurent symbol restricted to the unit circle, i.e., $\hat{f}(\theta)=F(e^{2\pi i \theta})$, it is easy to check that $T\in\mathscr{R}$ is equivalent to the condition that 
\[\hat{f}(\theta)=\frac{p(\cos(2\pi\theta),\sin(2\pi\theta))}{q(\cos(2\pi\theta),\sin(2\pi\theta))}\]
is a rational trigonometric function, where $p$ and $q$ are real polynomials in two variables and $q(\cos(2\pi\theta),\sin(2\pi\theta))$ has no zero on $\mathbb{T}$. In particular, all self-adjoint short-range hopping operators (which include the standard Laplacian on $\mathbb{Z}$) are contained in $\mathscr{R}$.\\
(2) Although $\mathscr{R}$ contains only a portion of hopping operators that decay exponentially fast, if one takes the hopping as 
\begin{equation}\label{experiment hopping}
f(n)=a^{-|n|},\qquad n\in\mathbb{Z}
\end{equation}
with some $|a|>1$ (which is often taken in physical experiments), then the Laurent symbol becomes 
\[F(z)=1+\frac{a}{a-z}+\frac{az}{az-1},\]
and therefore \eqref{experiment hopping} lies in $\mathscr{R}$.\\
(3) The reason we prove the localization result only for $T\in\mathscr{R}$ is that our method relies heavily on a (deterministic) quantitative version of the unique continuation principle, which provides the transversality needed to handle the absence of a Wegner estimate in the Bernoulli potential case. However, such a unique continuation principle does not hold for every hopping that decays exponentially fast (some counterexamples will be given in Section \ref{QUC section}), and in this paper we can only establish it for $\mathscr{R}$. The localization result for more general hopping coupled with Bernoulli potential remains open.
\end{rmk}

\subsection{Organization of the paper}
The paper is organized as follows. In Section \ref{QUC section}, we introduce the unique continuation principle for the operator \eqref{longrange schrodinger operator}, give some counterexamples, and establish the affirmative version \textcolor{red}{\textbf{when the hopping lies in $\mathscr{R}$ or decays super-exponentially}}. In Section \ref{initial scale section}, we adapt the method developed in \cite[Section 2]{LSZ25} (which combines the arguments from \cite{Bou04} with some techniques for studying Lipschitz tails from \cite{Klopp98,Klopp02}) to obtain the initial-scale Green's function estimate; we also analyze the effect of multiple minima on the model. In Section \ref{large scale section}, we combine the free sites argument from \cite{BK05} and the eigenvalue movement estimate from \cite{DS20} to establish the Wegner estimate for our model, and then prove the Green's function estimate for large scales. In Section \ref{localization section}, we eliminate the dependence on $E$ from the previous estimates and establish localization near the spectral edge. Some proofs and useful lemmas are placed in the Appendix at the end of the paper for ease of reading. 
\subsection{Notations}
The  notations used in this paper can be collected as follows. 
\begin{itemize}
\item $C=C(a,b,\cdots)$ means the constant $C$ depends on parameters $a,b,\cdots$.
\item We adopt the Vinogradov symbol $f\lesssim g$ for two nonnegative quantities $f$ and $g$, if there is an absolute constant
$C > 0$ such that $f\leq Cg$. If we want to emphasize that $C$ depends on some parameters
$a,b,\cdots$ independent of $f,g$, then we write $f\lesssim_{a,b,\cdots}g$. We denote $f\sim g$ if  both $f\lesssim g$ and $g\lesssim f$. In some cases, we denote  $f\ll g$  if there is some small  enough $c>0$ independent of $f,g$ so that $f\leq c g.$ 
\item We adapt the Landau symbol $f=\mcO(g)$ to mean that $|f|\lesssim g$ and $f=\mcO_{a,b,\cdots}(g)$ to mean that $|f|\lesssim_{a,b,\cdots} g$. Moreover, if additionally $f$ is positive, we will write $f=\mcO_+(g)$ or $f=\mcO_{+;a,b.\cdots}(g)$. We also write $f=o(g)$ to mean $|f|\ll g$.
\item For any $a,b\in\R$, denote $a\wedge b=\min\{a,b\}$ and $a\vee b=\max\{a,b\}$. 
\item  Denote by $\#$ the cardinality of a set. In particular, when applied to a subset of the spectrum, for instance, $\#(\spec(H)\cap A)$, it means  the cardinality counting multiplicities. 
\item We use $|\cdot|$ to denote the $\ell^{\infty}$-norm and $|\cdot|_1$ the $\ell^1$-norm for vectors in $\Z^d$ or $\R^d$. Moreover, we write $\Lambda_L(n)=\{m\in \Z^d : |m-n|\le L\}$ and set $\Lambda_L = \Lambda_L(0)$. In particular, in the one-dimensional case $\Z$, the set $\Lambda_L(n)$ is simply the interval $[n-L,n+L]\cap \Z$, which we call an $L$-interval and still denote by $[n-L,n+L]$ if no misunderstanding.
  \item The notation $\| \cdot \|$, if has no misunderstanding, will denote the $L^2(\T)$ or $\ell^2(\Z^d)$ norm or the corresponding operator norm. $\langle \cdot, \cdot \rangle$ denotes  the standard inner product on each space.
  \item We denote the distance on the $d$-dimensional torus $\T^d$ by $\| \cdot \|_{\T^d} =\max_{1\leq i\leq d} \min_{k\in\Z} |\cdot_i-k|$. 
  \item Denote by $\Delta_{\rm free}:\ell^2(\Z^d)\rightarrow \ell^2(\Z^d)$ the free Laplacian on $\Z^d$, i.e. 
         \[\Delta_{\rm free} u (n)= \sum_{m:|m-n|_1=1} \left(u(n)-u(m)\right).\] 
 \item For a subset $\Lambda\subset \Z^d$, we define its boundaries to be  
\[\partial^{-}\Lambda =\{y\in \Lambda: \ \exists  x\notin \Lambda\ {\rm such\ that} \ x\sim y\},\]
\[\partial^{+}\Lambda =\{y \notin \Lambda: \ \exists  x\in \Lambda\ {\rm such\ that} \ x\sim y\}.\]
  \item For a subset $\Lambda \subset \mathbb{Z}^d$, let $R_{\Lambda}$ denote the restriction operator  on  $\ell^2(\Lambda)$. The restriction of $H$ to $\Lambda$ with Dirichlet boundary condition  is then given by  $H_{\Lambda} = R_{\Lambda} H R_{\Lambda}$. We also denote $V_{\Lambda}=R_{\Lambda}V R_{\Lambda}$.
\end{itemize}

\section{The quantitative unique continuation principle}\label{QUC section}
First, we recall the quantitative unique continuation principle for Schr\"odinger operator on $\R^d$:

\begin{thm}\label{UC in Rd}
\textup{(\cite[Lemma 3.10]{BK05})} Assume $\Delta u + V u = \gamma$ in $\mathbb{R}^d$ satisfies
\[u(0)=1,\quad \|u\|_{\infty} \le C,\quad \|V\|_{\infty} \le C.\]
Then the decay of $u$ obeys the following lower bound:
\[\sup_{|x-x_0|\le 1} |u(x)| + \|\gamma\|_{\infty} \;\gtrsim\; \exp\Bigl\{-c\, (\log |x_0|)\cdot |x_0|^{4/3}\Bigr\}, \ \forall\, |x_0|>1,\]
for some constant $c>0$.
\end{thm}

Analoguely, we ask whether the following generalization of Theorem \ref{UC in Rd} hold for the operator like \eqref{longrange schrodinger operator} on $\Z$:
\begin{Def}[QUC]
Assume that the hopping operator $T$ satisfies \textbf{\textup{(A1)}} and \textbf{\textup{(A2)}}. We say that $T$ has the \underline{quantitative unique continuation property (QUC)} if, for any bounded real potential $V$ with $\|V\|_{\infty}\le B$, the solution of
\begin{equation}\label{equation in QUC}
	 (T+V)u = 0 \quad \text{on } \mathbb{Z}, \quad u(0)=1,
\end{equation}
does not decay faster than an exponential on a set of \textbf{full dimension}. More precisely, there exist constants $C=C(B,T)>1$, $\epsilon=\epsilon(T)>0$, and $L_0=L_0(T)\ge 1$ such that for all $L\ge L_0$,
\begin{equation}\label{transversality on a one-dim set}
\#\Bigl\{ x\in[-L,L] : |u(x)| > C^{-L} \Bigr\} \;\ge\; \epsilon L.
\end{equation}
\end{Def}

Such a definition is reasonable, because when $Tu(n)=2u(n)-u(n-1)-u(n+1)$ is the free Laplacian on $\mathbb{Z}$, the so-called ``cone property'' (cf. \cite[Section 2.1]{LSZ26}) ensures that the free Laplacian has the QUC with parameters $C=3+B$, $\epsilon=1/2$, and $L_0=1$ (see \cite[Theorem 2.3]{LSZ26}). 

However, in higher dimensions, even the free Laplacian fails to satisfy the QUC, since there exists a famous counterexample in $\Z^2$ (see \cite[Theorem 2]{Jit07}) whose solution only exhibits slow decay on a one-dimensional subset (and not on a full-dimensional subset of $\mathbb{Z}^2$). But if one restrict to the case $V\equiv 0$, \eqref{transversality on a one-dim set} does hold (with $\epsilon L$ replaced by $\epsilon L^2$ and $[-L,L]$ replaced by $\Lambda_L$) for harmonic functions on $\Z^2$ by the results in \cite{BLMS}.

For the long-range hopping case on $\mathbb{Z}$, also, the QUC \textbf{does not} hold for every hopping operator $T$, even when $T$ is of elliptic type with exponential decay and even when $V\equiv 0$. This can be seen from the following counterexample, which also reveals that the QUC is indeed connected to the zeros of the hopping's Laurent symbol.

\begin{ctex}\label{oscillation counterexample}
The following example shows that when the Laurent symbol of $T$ has violent oscillation and many zeros, the QUC fails because there are too many linearly independent solutions of $Tu=0$.

Take $B(z)=\sin \frac{1}{z-1/2}$. Let $T$ have the Laurent symbol
\begin{equation}\label{oscillation counterexample Laurent symbol}
	F(z)=(2-z-z^{-1})\cdot B(z)\cdot B(z^{-1}).
\end{equation}
One can easily check that the function \eqref{oscillation counterexample Laurent symbol} has the following properties:
\begin{itemize}
	\item $F(z)$ is holomorphic in $\{\frac{1}{2}<|z|<2\}$; consequently, the hopping decays exponentially with $|f(n)|\lesssim (2-)^{-|n|}$;
    \item We have $f(n)\in\mathbb{R}$; i.e., all entries of $T$ are real. This is because each of the functions $\sin z$, $2-z-z^{-1}$, $\frac{1}{z-1/2}$, and $\frac{1}{z}$ has a real Laurent expansion, and the class of functions with real Laurent coefficients is closed under multiplication and composition;
	\item The zeros of $F(z)$ are as follows: $z=1$ with multiplicity $2$ (coming from $2-z-z^{-1}$); $z=\frac{1}{2}+\frac{1}{\pi n},\ n\in\mathbb{Z}$ with multiplicity $1$ (coming from $B(z)$); $z=(\frac{1}{2}+\frac{1}{\pi n})^{-1},\ n\in\mathbb{Z}$ with multiplicity $1$ (coming from $B(z^{-1})$);
	\item $z=2$ and $z=\frac{1}{2}$ are two essential singularities of $F(z)$;
	\item On the unit circle, the corresponding Fourier symbol is given by 
	  \[\hat{f}(\theta)=F(e^{2\pi i \theta})= (2-2\cos (2\pi\theta))\cdot |B(e^{2\pi i\theta})|^2\geq 0,\]
	  which is real and analytic on $\T$. The minimum is attained uniquely at $\theta=0$, and we have
	  \[\hat{f}(\theta)=4\pi^2 (\sin 2)^2 \cdot \theta^2+ \mcO(\theta^4) \quad {\rm near }\ \theta=0.\]
      This reveals that $T$ is just like the free Laplacian and is an elliptic operator from the perspective of the Fourier symbol.
\end{itemize}
Now we study equation \eqref{equation in QUC} with $V\equiv 0$, i.e., 
\begin{equation}\label{equ in oscillation counterexample}
Tu = 0 \quad \text{on } \mathbb{Z}, \quad u(0)=1.
\end{equation}
Denote by $\lambda_j = \bigl(\frac{1}{2}+\frac{1}{\pi j}\bigr)^{-1} > 1$, $j\ge 1$, the zeros of $F(z)$. For each $j\ge 1$, set 
\[u_j(n)=\lambda_j^n, \ \forall n\in \Z.\]
Then $u_j(0)=1$, and a simple computation shows that 
\begin{equation}\label{uj is a solution}
	Tu_j (m)=\sum_{k\in \Z}f(k)u_j(m-k)=\lambda_j^m \cdot \sum_{k\in \Z}f(k)\lambda_j^{-k}=\lambda_j^m \cdot F(\lambda_j^{-1}).
\end{equation}
Since $\lambda_j^{-1}$ is also a zero of $F(z)$, the right-hand side of \eqref{uj is a solution} vanishes; hence $u_j$ is a solution of \eqref{equ in oscillation counterexample} for each $j\ge 1$. Now, for any $L\geq 1$, let $c_1,c_2,\dots,c_{2L+1}$ be chosen such that 
\[u^{(L)}=\sum_{1\leq j\leq 2L+1}c_j u_j\]
satisfies
\begin{equation}\label{equation of cj}
	u^{(L)}(0)=1;\ u^{(L)}(n)=0,\ \forall 1\leq |n|\leq L. 
\end{equation}
Such coefficients $c_1,c_2,\dots,c_{2L+1}$ exist because the conditions \eqref{equation of cj} can be written as a linear system
\[M \mathbf{c}=\mathbf{b},\ \mathbf{c}=(c_1,\cdots,c_{2L+1})^{\top}, \ \mathbf{b}=(0,\cdots,0,1,0,\cdots,0)^{\top} \]
where the only nonzero entry of $\mathbf{b}$ is $b_{L+1}=1$. The coefficient matrix $M$ is given by 
\[
M = \begin{pmatrix}
\lambda_1^{-L} & \lambda_1^{-L+1} & \cdots & \lambda_1^{L} \\
\lambda_2^{-L} & \lambda_2^{-L+1} & \cdots & \lambda_2^{L} \\
\vdots & \vdots & \ddots & \vdots \\
\lambda_{2L+1}^{-L} & \lambda_{2L+1}^{-L+1} & \cdots & \lambda_{2L+1}^{L}
\end{pmatrix}.
\]
By the theory of Vandermonde determinants, we have
\[\det M=\prod_{1\leq j\leq 2L+1} \lambda_j^{-L}\cdot \prod_{1\leq i<j\leq 2L+1}(\lambda_j-\lambda_i)\neq 0,\]
hence the coefficients $c_1,\dots,c_{2L+1}$ can be solved as $\mathbf{c}=M^{-1}\mathbf{b}$. We have constructed a function $u^{(L)}$ satisfying \eqref{equ in oscillation counterexample} and \eqref{equation of cj} for each $L\geq 1$. This implies that
\[\#\Bigl\{ x\in[-L,L] : |u^{(L)}(x)| \neq 0 \Bigr\} = 1.\]
Therefore, \eqref{transversality on a one-dim set} cannot hold for any choice of $C$, $\epsilon$, and $L_0$, and hence $T$ fails to satisfy the QUC.

\end{ctex}

The counterexample \ref{oscillation counterexample} inspires us that, if one expects $T$ to have the QUC, it might be necessary to avoid the occurrence of essential singularities (including $\infty$) of the Laurent symbol. Thus it is natural to ask whether $T$ has the QUC when its Laurent symbol is a rational function, since a well-known result in complex analysis states that \textbf{\textcolor{red}{a meromorphic function on the Riemann sphere $\mathbb{C}\cup\{\infty\}$ must be rational}}.\footnote{We remark that there might be some other strange functions that have no essential singularities but possess a natural boundary; we do not study them in this paper either. In this sense, rational functions are exactly the whole class of functions we can investigate.} Fortunately, the QUC actually holds for rational functions, and we will discuss this in the following.


Recall the definition of $\mathscr{R}$ in \eqref{rational hopping class}. For each $T\in\mathscr{R}$, a simple algebraic argument shows that assumption \textbf{(A1)} allows us to \textbf{uniquely} represent the Laurent symbol of $T$ by
\begin{equation}\label{represent mathscrR}
	F(z) = \frac{P(z)}{Q(z)}, 
\end{equation}
where
\begin{equation*}
P(z) = \sum_{|k|\le d_1} p_k z^k,\quad 
Q(z) = \sum_{|k|\le d_2} q_k z^k
\end{equation*}
are Laurent polynomials satisfying
\begin{equation}\label{parameter relationship in the representation}
\begin{cases}
&\text{$P$ and $Q$ are coprime;}\\
&p_0,q_0\in\mathbb{R};\\
&p_{d_1}\neq 0,\ q_{d_2}\neq 0,\ 0\leq \arg(p_{d_1})<\pi;\\
&p_k = \overline{p_{-k}},\ \forall |k|\le d_1;\ q_k = \overline{q_{-k}}, \ \forall |k|\le d_2.
\end{cases}
\end{equation}
Therefore, all the parameters above essentially depend on $T$. Moreover, if we denote
\begin{equation}\label{max deg of P,Q}
	D_{\hop}(T)=d_1 \vee d_2,
\end{equation}
then the non-constant assumption \textbf{(A2)} will ensure that $D_{\hop}(T)\geq 1$.

Under the above representation, we can further classify the hoppings in $\mathscr{R}$ as
\[\mathscr{R}=\mathscr{R}_+\cup \mathscr{R}_0\cup \mathscr{R}_-,\]
where
\begin{align*}
	\mathscr{R}_+ &  := \left\{ T : T\in \mathscr{R} \ {\rm and } \ d_1>d_2 \right\};\\
		\mathscr{R}_0 &  := \left\{ T : T\in \mathscr{R} \ {\rm and } \ d_1=d_2 \right\};\\
			\mathscr{R}_- &  := \left\{ T : T\in \mathscr{R} \ {\rm and } \ d_1<d_2 \right\}.
\end{align*}
We have the following result:
\begin{thm}[QUC for $\mathscr{R}$]\label{QUC for rational class}
	For each $T\in \mathscr{R}$, under the representation \eqref{represent mathscrR} for $T$,
	\begin{itemize}
		\item[(1)] if $T\in \mathscr{R}_+$, then $T$ has the QUC;
		\item[(2)] if $T\in \mathscr{R}_0$ and we additionally assume that the potential $V$ satisfies the lower bound condition 
		            \begin{equation}\label{lower bound on potential, R0}
						\inf_{n\in \Z} \left|V(n)+\frac{p_{d_1}}{q_{d_1}} \right|\geq b,
					\end{equation}
					then $T$ has the QUC with \eqref{transversality on a one-dim set} replaced by 
					\begin{equation}\label{transversality on a one-dim set, scrR0}
                       \#\Bigl\{ x\in[-L,L] : |u(x)| > (b^{-1}C)^{-L} \Bigr\} \;\ge\; \epsilon L,\ \forall L\geq L_0;
                    \end{equation}
	    \item[(3)] if $T\in \mathscr{R}_-$ and we additionally assume that the potential $V$ satisfies the lower bound condition 
		            \begin{equation}\label{lower bound on potential, R-}
						\inf_{n\in \Z} \left|V(n) \right|\geq b,
					\end{equation}
					then $T$ has the QUC with \eqref{transversality on a one-dim set} replaced by 
					\begin{equation}\label{transversality on a one-dim set, scrR-}
                       \#\Bigl\{ x\in[-L,L] : |u(x)| > (b^{-1}C)^{-L} \Bigr\} \;\ge\; \epsilon L,\ \forall L\geq L_0.
                    \end{equation}

	\end{itemize}
\end{thm}

We emphasize that the lower bound conditions \eqref{lower bound on potential, R0} and \eqref{lower bound on potential, R-} are necessary, because otherwise we have the following counterexample:
\begin{ctex}
	Let
	\begin{equation}\label{laurent symbol of rational counterexample}
			P(z)=2-z-z^{-1}, \ Q(z)=1+\varepsilon \cdot \sum_{k=1}^{N}(z^k+z^{-k}), 
	\end{equation}
    where $N\ge 1$ is an integer and $0<\varepsilon<\frac{1}{4N}$. Then on $\{|z|=1\}$ we have
	\[|Q(z)|\geq 1-\varepsilon\cdot \sum_{k=1}^{N}(|z|^{k}+|z|^{-k})=1-2N\varepsilon\geq \frac{1}{2},\]
    so $Q(z)$ has no zero. Let $T$ have the Laurent symbol $F(z)=P(z)/Q(z)$. Take	
    \[
      u(k)=q_k=
        \begin{cases}
              1, & \text{if } k=0;\\
              \varepsilon, & \text{if } 1\le |k|\le N;\\
              0, & \text{otherwise}.
              \end{cases}
     \]
	Then $Tu(k)=\sum_{m\in\mathbb{Z}} T(k-m)u(m)$ is exactly the Laurent coefficient of $F(z)\cdot Q(z)=P(z)$, i.e.,
            \[
        Tu(k)=p_k=2\delta_{0,k}-\delta_{1,k}-\delta_{-1,k}, \ \forall k\in \Z,
         \]
      where $\delta_{m,n}$ is the Kronecker delta. Therefore, if we take the potential as 
	  \begin{equation}\label{potential of rational counterexample}
			   V(k)=-2\delta_{0,k}+\frac{1}{\varepsilon}\delta_{1,k}+\frac{1}{\varepsilon}\delta_{-1,k},\ \forall k\in \Z,
	  \end{equation}
       then we have $(T+V)u = 0$ on $\mathbb{Z}$, $u(0)=1$, and $\|V\|_{\infty} \le \frac{1}{\varepsilon}=B$. Hence $u$ is a solution of \eqref{equation in QUC}. However, for any sufficiently large $L \ge N$, we only have
	  \[
	                        \#\Bigl\{ x\in[-L,L] : |u(x)| \neq 0 \Bigr\} =2N+1. 
	  \]
	  Therefore \eqref{transversality on a one-dim set} fails for such $T$ and $V$. If one investigates our choices \eqref{laurent symbol of rational counterexample} and \eqref{potential of rational counterexample}, one finds the following:
	  \begin{itemize}
		\item If $N=2$, then $T\in \mathscr{R}_0$ with $d_1=d_2=2$ and $p_2/q_2 = -1/\varepsilon$. Consequently,
		         \[\inf_{n\in \Z} \left|V(n)-\frac{p_2}{q_2} \right|=\inf_{n\in \Z} \left|V(n)+\frac{1}{\varepsilon} \right|=0,\]
			which exactly violates \eqref{lower bound on potential, R0}.
		\item If $N>2$, then $T\in \mathscr{R}_-$. Thus
		     \[\inf_{n\in \Z}\left| V(n) \right|=0,\]
			which exactly violates \eqref{lower bound on potential, R-}.
	  \end{itemize}
	  Thus this counterexample shows that the lower bound conditions \eqref{lower bound on potential, R0} and \eqref{lower bound on potential, R-} are necessary.
\end{ctex}

\begin{proof}[Proof of Theorem \ref{QUC for rational class}]
	We take the representation as in \eqref{represent mathscrR}. Let $T\in\mathscr{R}$ have the Laurent symbol $F(z)=P(z)/Q(z)$. Clearly, $P(z)$ and $Q(z)$ are the Laurent symbols of the following short-range hopping operators $T_P$ and $T_Q$, respectively:
    \[T_P u(n)=\sum_{m\in\mathbb{Z}} p_{n-m} u(m),\quad T_Q u(n)=\sum_{m\in\mathbb{Z}} q_{n-m} u(m).\]
    Moreover, we have $T_Q  T = T_P$. Now if $u$ is a solution of \eqref{equation in QUC}, then
    \begin{equation}\label{convolution renormalization equation}
         T_Q (T+V)u = T_P u + T_Q(V u) = 0.
    \end{equation}
    Pointwise, equation \eqref{convolution renormalization equation} reads, for each $n\in\mathbb{Z}$,
    \begin{equation}\label{pointwise convolution renormalization equation}
            \sum_{|k|\leq D_{\hop}(T)} \bigl(p_k + q_k V(n-k)\bigr) u_{n-k} = 0,
    \end{equation}
	where $D_{\hop}(T)$ is defined in \eqref{max deg of P,Q}. For simplicity, in the rest of the proof we denote $R = D_{\mathrm{hop}}(T)$.\\
	
	\noindent (1)  When $T\in\mathscr{R}_+$, we have $R=d_1$ and, by \eqref{parameter relationship in the representation},
	\[p_R+q_{R}V(n-R)=p_{R}\neq 0.\]
    Therefore, \eqref{pointwise convolution renormalization equation} together with $\|V\|_{\infty}\le B$ yields
	\begin{align}\label{pigeonhole in R+}
	\notag	|u(n-R)|&= \left|-\frac{1}{p_R }   \left(  \sum_{-R\leq k< R}(p_k+q_k V(n-k)) u(n-k)  \right)   \right| \\
		 \notag       &\leq \frac{1}{|p_R|} \sum_{-R\leq k< R} \left(|p_k|+|q_k|\cdot | V(n-k)| \right) \cdot |u(n-k)| \\
				&  \leq \left( \frac{1}{|p_R|} \sum_{-R\leq  k< R} \left(|p_k|+|q_k|\cdot B \right) \right) \cdot \max_{-R\leq k< R}|u(n-k)|.
	\end{align}
	Define the constant depending only on $T$ and $B$ as
	\[ C=C(T,B)=\left( \frac{1}{|p_R|} \sum_{-R\leq  k< R} \left(|p_k|+|q_k|\cdot B \right) \right)\vee 1. \]
	Thus, \eqref{pigeonhole in R+} implies that
	\begin{equation}\label{right cone prop for R+}
		\max_{0<k\leq 2R}|u(n+k)|\geq C^{-1}\cdot |u(n)|
	\end{equation}
	holds for each $n\in \Z$. Since $u(0)=1$, \eqref{right cone prop for R+} ensures the existence of an increasing integer sequence $x_0=0 < x_1 < x_2 < \cdots < x_j < \cdots$ such that
	\[|x_{j+1}-x_j|\leq 2R,\ |u(x_{j+1})|\geq C^{-1}|u(x_j)|\geq C^{-j}|u(0)|\geq C^{-|x_j|}.\]
    Therefore, if we take $L_0 = 2R$, then for any $L \ge L_0$ we have $x_{\lfloor L/(2R)\rfloor} \le L$, and consequently
	\begin{equation}\label{right direction UC for R+}
		     \#\Bigl\{ x\in[0,L] : |u(x)| > C^{-L} \Bigr\} \;\ge\; \lfloor \frac{L}{2R}\rfloor +1\geq \frac{L}{2R}, \ \forall L\geq L_0.
	\end{equation}
	Since $[0,L] \subset [-L,L]$ and $R=D_{\hop}(T)$, \eqref{right direction UC for R+} proves that \eqref{transversality on a one-dim set} in the definition of QUC holds if we take $L_0 = L_0(T) = 2D_{\mathrm{hop}}(T)$ and $\epsilon = \epsilon(T) = \frac{1}{2D_{\mathrm{hop}}(T)}$.\\

	\noindent (2) When $T\in \mathscr{R}_0$, we have $R=d_1=d_2$ and, by \eqref{parameter relationship in the representation} and \eqref{lower bound on potential, R0},
	\[p_R+q_{R}V(n-R)=q_{R} \left(V(n-R)+\frac{p_R}{q_R} \right)\neq 0.\]
   The proof of this case is similar to (1). From \eqref{lower bound on potential, R0}, \eqref{pointwise convolution renormalization equation} and $\|V\|_{\infty} < B$, we obtain
		\begin{align}\label{pigeonhole in R0}
	\notag	|u(n-R)|&= \left|-\frac{1}{q_{R} \left(V(n-R)+\frac{p_R}{q_R} \right) }   \left(  \sum_{-R\leq k< R}(p_k+q_k V(n-k)) u(n-k)  \right)   \right| \\		
				&  \leq \left( \frac{1}{|q_R|} \sum_{-R\leq  k< R} \left(|p_k|+|q_k|\cdot B \right) \right) \cdot\frac{1}{b} \max_{-R\leq k< R}|u(n-k)|.
	\end{align}
	Define the constant depending only on $T$ and $B$ as
	\[ C=C(T,B)=\left( \frac{1}{|q_R|} \sum_{-R\leq  k< R} \left(|p_k|+|q_k|\cdot B \right) \right)\vee 1. \]
	Thus, \eqref{pigeonhole in R0} implies that
	\begin{equation}\label{right cone prop for R0 and R-}
		\max_{0<k\leq 2R}|u(n+k)|\geq C^{-1}b\cdot |u(n)|
	\end{equation}
	holds for each $n\in \Z$. Then, by an argument similar to the one used to prove \eqref{right direction UC for R+}, we see that \eqref{right cone prop for R0 and R-} yields
	\begin{equation}\label{right direction UC for R0 and R-}
		     \#\Bigl\{ x\in[0,L] : |u(x)| > (b^{-1}C)^{-L} \Bigr\} \;\ge\;  \epsilon L, \ \forall L\geq L_0.
	\end{equation}
    with $L_0 = L_0(T) = 2D_{\mathrm{hop}}(T)$ and $\epsilon = \epsilon(T) = \frac{1}{2D_{\mathrm{hop}}(T)}$, and hence \eqref{transversality on a one-dim set, scrR0} follows automatically from $[0,L]\subset[-L,L]$.\\

	\noindent (3) When $T\in \mathscr{R}_-$, we have $R=d_2$ and, by \eqref{parameter relationship in the representation} and \eqref{lower bound on potential, R-},
	\[p_R+q_{R}V(n-R)=q_{R} V(n-R) \neq 0.\]
  The rest of the proof is completely the same as in case (2), and one can still obtain \eqref{right cone prop for R0 and R-} and \eqref{right direction UC for R0 and R-} in this case.
\end{proof}

\begin{rmk}
(1) The pointwise equation \eqref{pointwise convolution renormalization equation} reveals that the hoppings in $\mathscr{R}$ are essentially short-range up to a convolution-type renormalization. It may be of some interest to ask whether, although we have already established localization in this paper, one could apply the transfer matrix method to this renormalized equation and obtain an alternative proof, though such an approach would likely be less general than the method presented here.\\
(2) Indeed, the cone properties \eqref{right cone prop for R+} and \eqref{right cone prop for R0 and R-} only describe the propagation in the right direction; therefore, the quantitative unique continuation can be established on the right interval $[0,L]\subset[-L,L]$ (see \eqref{right direction UC for R+} and \eqref{right direction UC for R0 and R-}). Analogously, if one considers the coefficient
       \[p_{-R}+q_{-R}V(n+R)=\overline{p_{R}+q_{R}V(n+R)}\]
	   in \eqref{pointwise convolution renormalization equation}, one also obtains the same estimates along the left direction, i.e. 
       \begin{equation}\label{left cone prop for R+}
       \max_{0<k\leq 2R}|u(n-k)|\geq C^{-1}\cdot |u(n)|;
       \end{equation}
       \begin{equation}\label{left direction UC for R+}
          \#\Bigl\{ x\in[-L,0] : |u(x)| > C^{-L} \Bigr\} \;\ge\; \epsilon L, \ \forall L\geq L_0.
        \end{equation}
       for $T\in \mathscr{R}_+$; and 
       \begin{equation}\label{left cone prop for R0 and R-}
       \max_{0<k\leq 2R}|u(n-k)|\geq C^{-1}b\cdot |u(n)|;
       \end{equation}
       \begin{equation}\label{left direction UC for R0 and R-}
       \#\Bigl\{ x\in[-L,0] : |u(x)| > (b^{-1}C)^{-L} \Bigr\} \;\ge\; \epsilon L, \ \forall L\geq L_0.
       \end{equation}
       for $T\in \mathscr{R}_0$ and $T\in \mathscr{R}_-$ with potentials satisfying \eqref{lower bound on potential, R0}, \eqref{lower bound on potential, R-} respectively. Here $R=D_{\hop}(T)$ actually.
\end{rmk}

However, to apply Theorem \ref{QUC for rational class} in the later proof of the Wegner estimate, we need to ensure that the transversality set matches the free sites argument. Thus we must refine Theorem \ref{QUC for rational class} as follows:

\begin{thm}[QUC for $\mathscr{R}$ on set of free sites]\label{QUC for rational class, refined to free sites}
	For each $T\in \mathscr{R}$, we adopt the representation \eqref{represent mathscrR} for $T$. For all $L\geq L_0=L_0(T)=2D_{\hop}(T)$, let $\Lambda=[-L,L]$ and consider $u$ to be a nonzero solution of the following Dirichlet boundary problem:
	\begin{equation}\label{Dirichlet QUC equation}
			(T_{\Lambda}+V_{\Lambda})u = 0, \quad u\in\ell^2(\Lambda).
	\end{equation}
	Assume that $\Lambda'\subset \Lambda$ is a $L'$-interval ($L'\leq L$), and $S\subset \La'$ is a union of disjoint intervals of length $2D_{\hop}(T)+1$. Let $\wtS$ denote the set of centers of those intervals, i.e.,
	\[
	S=\bigcup_{m \in \wtS}[m-D_{\hop}(T),\,m+D_{\hop}(T)].
	\]
	Assume that $\|V\|_{\infty}\leq B$ and that $u$ is a solution of \eqref{Dirichlet QUC equation}. Then
	\begin{itemize}
		\item[(1)] if $T\in \mathscr{R}_+$, then 
		            \begin{equation}\label{transversality on a one-dim set, scrR+, free sites type}
                       \#\Bigl\{ x\in S : |u(x)| > C^{-L'} \| u\|_{\ell^\infty(\La')} \Bigr\} \;\ge\; \# \wtS;
                    \end{equation}
		\item[(2)] if $T\in \mathscr{R}_0$ and we additionally assume that the potential $V$ satisfies \eqref{lower bound on potential, R0}, then 
					\begin{equation}\label{transversality on a one-dim set, scrR0, free sites type}
                       \#\Bigl\{ x\in S : |u(x)| > (b^{-1}C)^{-L'} \| u\|_{\ell^\infty(\La')} \Bigr\} \;\ge\; \# \wtS;
                    \end{equation}
	    \item[(3)] if $T\in \mathscr{R}_-$ and we additionally assume that the potential $V$ satisfies \eqref{lower bound on potential, R-}, then
					\begin{equation}\label{transversality on a one-dim set, scrR-, free sites type}
                       \#\Bigl\{ x\in S : |u(x)| > (b^{-1}C)^{-L'} \| u\|_{\ell^\infty(\La')} \Bigr\} \;\ge\; \# \wtS.
                    \end{equation}
	\end{itemize}
	Here $C=C(B,T)$ is a constant depending only on $B$ and $T$.
\end{thm}
\begin{proof}[Proof of Theorem \ref{QUC for rational class, refined to free sites}.]
We first show that the renormalized pointwise equation \eqref{pointwise convolution renormalization equation} in the non-boundary problem also holds for the Dirichlet boundary problem. This is because $T_Q$ and $T_P$ are short-range. Examining the entries of the relation $T_Q T = T_P$ gives 
\begin{equation}\label{entrie renormalization relationship}
	\sum_{k:|k-n|\leq d_2} q_{n-k} f(k-m) = p_{n-m}, \ \forall n,m\in \mathbb{Z}.
\end{equation}
Since $u$ is a solution of \eqref{Dirichlet QUC equation}, for every $n\in [-L,L]$ we have 
\[
(T_{\Lambda}+V_{\Lambda})u(n) = \sum_{m\in [-L,L]} f(n-m) u(m) + V(n)u(n) = 0.
\]
We restrict $-L+D_{\mathrm{hop}}(T) \leq n \leq L-D_{\mathrm{hop}}(T)$ and extend $(T_{\Lambda}+V_{\Lambda})u = 0$ to $\ell^2(\mathbb{Z})$ by zero. Then by \eqref{entrie renormalization relationship},
\begin{align*}
0 &= T_q (T_{\Lambda}+V_{\Lambda})u 
   = \sum_{k:|k-n|\leq d_2} q_{n-k} \bigl[(T_{\Lambda}+V_{\Lambda})u(k)\bigr] \\
  &= \sum_{k:|k-n|\leq d_2} q_{n-k} \left( \sum_{m\in [-L,L]} f(k-m)u(m) + V(k)u(k) \right) \\
  &= \sum_{m\in [-L,L]} \left( \sum_{k:|k-n|\leq d_2} q_{n-k} f(k-m) \right) u(m) 
     + \sum_{k:|k-n|\leq d_2} q_{n-k} V(k) u(k) \\
  &= \sum_{m:|m-n|\leq D_{\mathrm{hop}}(T)} \bigl( p_{n-m} + q_{n-m} V(m) \bigr) u(m).
\end{align*}
That is,
\begin{equation}\label{pointwise renormalization relationship, Dirichlet}
	\sum_{|k|\leq D_{\mathrm{hop}}(T)} \bigl( p_k + q_k V(n-k) \bigr) u(n-k) = 0
\end{equation}
holds for every $-L+D_{\mathrm{hop}}(T) \leq n \leq L-D_{\mathrm{hop}}(T)$. Thus we recover \eqref{pointwise convolution renormalization equation} for \eqref{Dirichlet QUC equation} successfully.

With \eqref{pointwise renormalization relationship, Dirichlet} in hand, we only consider the case $T\in\mathscr{R}_+$, since the other two cases follow by the same argument. Let $x_0\in\La'$ be a point where $u$ attains its maximum in $\La'$, i.e., $u(x_0)=\|u\|_{\ell^\infty(\La')}$. As in the proof of Theorem \ref{QUC for rational class}, we still have \eqref{right cone prop for R+} and \eqref{left cone prop for R+}. Thus there exists a sequence 
\[\cdots < x_s < x_{s+1} < \cdots < x_{-1} < x_0 < x_1 < \cdots < x_j < x_{j+1} < \cdots \]
such that 
\[|x_s - x_{s+1}| < 2R = 2D_{\hop}(T), \qquad |u(x_s)| \ge C^{-|s|} |u(0)| \ge C^{-|x_s|} \| u\|_{\ell^\infty(\La')}.\]
Therefore $\{x_s\}$ forms a $D_{\hop}(T)$-net of $\La'$, and hence each interval $[m-D_{\hop}(T), m+D_{\hop}(T)]$ with $m\in\wtS$ contains at least one $x_s$. This proves \eqref{transversality on a one-dim set, scrR+, free sites type}.
\end{proof}

Finally, one might still want to consider the natural question: \textbf{for more general hopping operators (with more general Laurent symbols), although the QUC may fail, can we establish some ``weaker'' results in the unique continuation scheme?} There are two possible ways: one is to shift the QUC from a deterministic type to a large deviation type; that is to say, when the potential $V$ indeed involves some randomness, one can prove that the QUC holds with large probability (just as \cite{DS20,Li22,LSZ26} did). Another is to weaken the lower bound on the dimension of the transversality set. In this direction, we have the following result for all hoppings with super-exponential decay, under the additional assumption that the solution is bounded:
\begin{thm}[weak QUC for hoppings with super-exponential decay]\label{QUC for super-exponential decay}
	Assume $T$ satisfies \textup{\textbf{(A1)}, \textbf{(A2)}} and 
	\[|f(n)|\lesssim \exp\{-|n|^{\alpha}\},\ \alpha>1.\]
	Let $u$ be a solution of \eqref{equation in QUC} satisfying $\|V\|_{\infty}\leq B$ and $\|u\|_{\infty}\leq \wtB$. Then there exist $L_0=L_0(T,B,\wtB)$, $C=C(T,B)$, and $\epsilon=\epsilon(T,B)$ such that for every $L\geq L_0$, we have 
	\begin{equation}\label{transversality for super-exponential decay}
		\#\Bigl\{ x\in[-L,L] : |u(x)| > C^{-L} \Bigr\} \;\ge\; \epsilon L^{1-\frac{1}{\alpha}}.
	\end{equation}
\end{thm}

We remark that Theorem \ref{QUC for super-exponential decay} does not hold without the boundedness condition $\| u\|_{\infty}<\wtB$. A counterexample is given by taking $f(k)=2^{-k^2}$, $k\in \mathbb{Z}$; then the Laurent symbol of $T$ becomes the Jacobi theta function 
\[
F(z)=\vartheta_3\!\left(\frac{\log z}{2i},\frac{1}{2}\right)=\prod_{m=1}^{\infty} (1-2^{-2m})(1+2^{-(2m-1)}z)(1+2^{-(2m-1)}z^{-1}),
\]
which has two essential singularities at $0$ and $\infty$, and zeros at $-2^{2m-1}$, $m\in \mathbb{Z}$. Then an argument similar to that in Counterexample \ref{oscillation counterexample} shows that \eqref{transversality for super-exponential decay} fails.

Since Theorem \ref{QUC for super-exponential decay} is independent of the proof of our main result Theorem \ref{Localization near the edge band}, we will defer its proof to Appendix \ref{appendix: proof of QUC for super-exponetial decay} for the convenience of the reader.

\section{The initial scale: elliptic analysis}\label{initial scale section}
In this section, we establish the Green's function estimate for \eqref{longrange schrodinger operator} near the edge of the spectrum at the initial scale. We adapt the method developed in \cite[Section 2]{LSZ25}, which combines the arguments from \cite{Bou04} with the periodic approximation techniques from \cite{Klopp98,Klopp02}. The argument in this section does not depend on the dimension, so we will carry it out in $\Z^d$ for comparison with the corresponding parts of previous works.

Let \eqref{longrange schrodinger operator} be defined on $\Z^d$; then the Fourier symbol of $T$ becomes 
\[ \hat{f}(\theta)=\sum_{n\in \Z^d} f(n) e^{2\pi i n\cdot \theta}, \ \theta\in \T^d. \]
We still assume that $\hat{f}$ satisfies \textbf{(A1)} and \textbf{(A2)} on $\T^d$, and, without loss of generality, we also assume that \textbf{(A3)} holds. Additionally, we make the following assumption:
\begin{itemize}
	\item[\textbf{(A4)}] $\hat{f}$ has finitely many minimum points $\theta_1,\theta_2,\dots,\theta_J$ and 
	    \begin{equation}\label{elliptic assumption}
			 \hat{f}(\theta) \geq D_T \cdot \left( \min_{1\leq j\leq J}\|\theta-\theta_j \|_{\T^d} \right)^{d_T}, \ \forall \theta\in \T^d
		\end{equation}
		for some constant $D_T>0$ and even positive integer $d_T$ depending only on $T$.
\end{itemize}

\begin{rmk}\label{remark on the assumption (A4)}
	(1) In the one-dimensional case, \textbf{\textup{(A4)}} automatically holds if \textbf{\textup{(A1)}} and \textbf{\textup{(A2)}} hold. This can be seen from the following argument: \textbf{\textup{(A2)}} ensures that $\hat{f}(\theta)$ can be holomorphically extended to $\hat{f}(z)$ in some band $\{z:|\Im z|<\rho \}$, and by the uniqueness theorem, $\hat{f}$ can have only finitely many zeros (and hence, by \textbf{\textup{(A3)}}, also finitely many minimum points) on $\{|\Im z|=0\}$, which we denote by $\theta_j$, $1\leq j\leq J$. Let the degree of $\theta_j$ be $d_j$ (which is an even positive number); then near each $\theta_j$ we have $\hat{f}(\theta)\gtrsim \|\theta-\theta_j \|_{\T}^{d_j}$. Set $d_T=\max_{1\leq j\leq J} d_j$, and then we obtain the global estimate \eqref{elliptic assumption} on $\T$.\\
    (2) However, when the dimension $d\geq 2$, \textbf{\textup{(A4)}} may fail even if \textbf{\textup{(A1)}} and \textbf{\textup{(A2)}} hold. For example, take $d=2$ and $\hat{f}(x_1,x_2)=2-2\cos(2\pi x_1)$ with $x=(x_1,x_2)\in \T^2$. Then $\hat{f}$ satisfies \textbf{\textup{(A1)}} and \textbf{\textup{(A2)}} and has minimum value zero, but it fails \textbf{\textup{(A4)}} because it attains its minimum on the whole line $\{(x_1,x_2): x_1=0\}$. Therefore, in higher dimensions we must additionally assume \textbf{\textup{(A4)}}, which is exactly what is done in \cite{LSZ25}.\\
    (3) The assumption \textbf{\textup{(A4)}} is not artificial. It is indeed an ``elliptic condition'' and ensures that $T$ acts like the Laplacian near the edge of the spectrum. This assumption can also be seen in \cite{Klopp98,Klopp02,GRM22,LSZ25}. However, unlike the Laplacian, whose Fourier symbol has only a unique minimum point, $T$ may have many minimum points. This poses some obstacles compared to the standard analysis for the Laplacian in \cite{Bou04,BK05,DS20,LZ22}, which we will discuss at the end of this section.
\end{rmk}

\subsection{LDT for the initial scales}
Let $0<\delta\ll 1$ and $N_0\gg 1$ be determined later. In the following, we only consider energies in the range $E\in [0,\delta]$. Write \eqref{longrange schrodinger operator} as 
\[
H(\omega)-E = (T-1) - \bigl(E-1-\lambda V(\omega)\bigr),
\]
and by Neumann series expansion we have 
\begin{align}\label{Neumann expansion}
   \notag  G_{N_0}(E;\omega) &= (H_{N_0}-E)^{-1} \\
   &= (\lambda V_{N_0}+1-E)^{-1} \sum_{s\geq 0} \left[ (1-T_{N_0})(\lambda V_{N_0}+1-E)^{-1} \right]^s.
\end{align}
Here, for simplicity we denote the restricted operators $H_{\Lambda_{N_0}}$ by $H_{N_0}$, and so on, and we hide the dependence on $\omega$.

By \textbf{(A3)}, the operator $1-T$ has Fourier symbol $1-\hat{f}\in [0,1]$, and therefore 
\begin{equation}\label{about 1-T}
	1-T \text{ is positive}; \quad \|1-T\|=\|1-\hat{f} \|_{\infty}=1. 
\end{equation}
Moreover, $E\in [0,\delta]$ and \eqref{Bernoulli potential} ensure that 
\begin{equation}\label{about reverse potential}
	\|(\lambda V+1-E)^{-1} \| \leq (1-\delta)^{-1}.
\end{equation}

We prove the following large deviation theorem (LDT) on the Green's function at the initial scale, which is the main part of this section: (The constant $C_{\rm in}$ will be chosen in Section~\ref{large scale section}.)
\begin{thm}\label{LDT for the initial scale, without free sites}
Assume that the hopping operator $T$ satisfies \textup{\textbf{(A1), (A2), (A3)}} and \textup{\textbf{(A4)}}. For any $C_{\rm in}\geq 2$, let $N_{\rm in}\gg_{\lambda,T,d, C_{\rm in}} 1$ be the initial scale, and let $\delta=(\log N_{\rm in})^{-6000d_T}$. Choose the scales 
\[
L=\lfloor \delta^{-\frac{1}{24}\cdot \frac{2}{d_T}}\rfloor \sim (\log N_{\rm in})^{500},\ L'=\lfloor\delta^{-\frac{1}{48}\cdot \frac{2}{d_T}} \rfloor \sim (\log N_{\rm in})^{250}.
\]
Then for any $N_{\rm in}\leq N_0 \leq N_{\rm in}^{C_{\rm in}}$ such that 
\begin{equation}\label{initial scale set Scale_in}
	N_0\in  {\rm Scale}_{\rm in}:=\left\{N\in\Z:\ N_{\rm in}\leq N\leq N_{\rm in}^{C_{\rm in}}, \ \frac{2N+1}{(2L'+1)(2L+1)}\in \Z_+  \right\}, 
\end{equation}
the following holds: for $\omega$ outside an event $\Xi_{N_0}\subset \{0,1\}^{\Lambda_{N_0}}$ with probability less than $e^{- (\log N_0)^{10}}$, we have 
\begin{equation}\label{L2 norm estimate on Green's function, initial scale}
	\|G_{N_0}(E;\omega) \| \leq (\log N_0)^{7000 d_{T}}
\end{equation}
and 
\begin{equation}\label{off diagonal decay estimate on Green's function, initial scale}
	|G_{N_0}(E;\omega)(n,n')| \leq  \exp \left\{  -\frac{|n-n'|}{(\log N_0)^{7000 d_T}} \right\}, \ \forall \ |n-n'|\geq \frac{N_0}{200}
\end{equation}
for every $E\in [0,\delta]$.
\end{thm}

\begin{proof}[Proof of Theorem \ref{LDT for the initial scale, without free sites}]
We divide the proof into six steps.

\noindent \textcolor{blue}{\large \textbf{(Step 1: construct the approximate eigenvector.)}} \\
Assume to the contrary that there exists some $E\in [0,\delta]$ such that 
\begin{equation}\label{reverse assumption in initial scale}
\left\| \left[  (1-T_{N_0})(\lambda V_{N_0}+1-E)^{-1} \right]^2 \right\| >1-\delta.		
\end{equation}
By \eqref{about reverse potential} and \eqref{reverse assumption in initial scale}, we obtain 
\begin{equation}\label{sandwich operator}
	\left\|  (1-T_{N_0})(\lambda V_{N_0}+1-E)^{-1}(1-T_{N_0})  \right\| > (1-\delta)^2. 
\end{equation}
Since $(1-T_{N_0})(\lambda V_{N_0}+1-E)^{-1}(1-T_{N_0})$ is self-adjoint on $\ell^2(\Lambda_{N_0})$, by \eqref{sandwich operator} we can find a unit vector $\xi$, supported in $\Lambda_{N_0}$, such that 
\begin{equation}\label{the inner product condition of vector xi}
	\langle \xi, (1-T_{N_0})(\lambda V_{N_0}+1-E)^{-1}(1-T_{N_0}) \xi  \rangle \geq (1-\delta)^2.
\end{equation}
Using \eqref{about reverse potential} and \eqref{the inner product condition of vector xi}, we have 
\[ \|(1-T)\xi \|^2 \geq  \| R_{\Lambda_{N_0}} (1-T)\xi\|^2 = \| (1-T_{N_0})\xi \|^2 \geq  (1-\delta)^3, \]
and by the Plancherel identity, this becomes
\begin{align*}
	(1-\delta)^3 &\leq \| (1-T) \xi \|^2 =\|(1-\hat{f})\hat{\xi} \|^2\\
	         &=\left( \int_{ \hatf< \eta}+\int_{\hatf\geq \eta} \right) | \hatxi |^2 (1-\hatf)^2 d\theta \\
			 &\leq \int_{\hatf<\eta} |\hatxi |^2 d\theta + (1-\eta)^2 \int_{\hatf \geq \eta}|\hatxi|^2 d\theta \\
             &=1-(1-(1-\eta)^2) \int_{\hatf\geq \eta}|\hatxi|^2 d\theta,
\end{align*}
where $\hatxi\in L^2(\T^d)$ is the Fourier transform of $\xi$ and we used $\| \hatxi \|=\| \xi \|=1$. The above estimate is equivalent to 
\begin{equation}\label{concentration of xi}
	\int_{\hatf\geq \eta} |\hatxi|^2 d\theta \leq \mcO_+\left(\frac{\delta}{\eta}\right),
\end{equation}
which means that $\hatxi$ is concentrated in $\{ \hatf<\eta\}$, i.e., near the minimum points of $\hatf$. With \eqref{concentration of xi} in hand, we can show that $\xi$ is a good approximate eigenvector of both $T$ and $V$:
\begin{itemize}
	\item (approximate eigenvector of $T$) We have 
	           \begin{align*}
				  \|T_{N_0}\xi\|^2 & \leq \| T\xi\|^2 = \left(  \int_{\hatf< \eta}+\int_{\hatf\geq \eta} \right) |\hatf|^2 |\hatxi|^2 d\theta \\
								&\leq \eta^2 \int_{\hatf< \eta} |\hatxi|^2 d\theta +\int_{\hatf\geq \eta} |\hatxi|^2 d\theta \\
								&\leq \eta^2+\mcO_+\left(\frac{\delta}{\eta}\right). 
			   \end{align*}
			   The optimal upper bound is attained at $\eta=\delta^{\frac{1}{3}}$, and we have 
	    \begin{equation}\label{approximate eigenvector of T}
			\|T_{N_0}\xi \| \leq \|T\xi \| \leq \mcO_+ \left(\delta^{\frac{1}{3}}\right).
		\end{equation}
    \item (approximate eigenvector of $V$) Substituting \eqref{approximate eigenvector of T} back into \eqref{the inner product condition of vector xi}, we obtain 
	        \[\langle \xi, (\lambda V_{N_0}+1-E)^{-1}\xi \rangle \geq (1-\delta)^2-\mcO_+\left(\delta^{\frac13}\right)=1-\mcO_+\left(\delta^{\frac{1}{3}}\right). \]
			Therefore, together with \eqref{about reverse potential}, we have 
			\begin{align*}
				\| \xi-(\lambda V_{N_0}+1-E)^{-1}\xi \|^2 & =\| \xi\|^2+\| (\lambda V_{N_0}+1-E)^{-1}\xi \|^2-2 \langle \xi, (\lambda V_{N_0}+1-E)^{-1}\xi \rangle \\
				                               &\leq 1+(1-\delta)^{-2}-2\left(1-\mcO_+\left(\delta^{\frac{1}{3}} \right)\right) \leq \mcO_+\left(\delta^{\frac{1}{3}}\right),
			\end{align*}
			and thus 
			\begin{align}\label{approximate eigenvector of V}
             \notag				\| (\lambda V_{N_0}-E)\xi\| &= \| \xi-(\lambda V_{N_0}+1-E)^{-1}\xi \| \\
				                    &\leq \| \lambda V_{N_0}+1-E\|\cdot \| \xi-(\lambda V_{N_0}+1-E)^{-1}\xi \| \leq \mcO_{+,\lambda} \left(\delta^{\frac{1}{6}}\right).
			\end{align}   
\end{itemize}

Combining \eqref{approximate eigenvector of T} and \eqref{approximate eigenvector of V}, we obtain 
\begin{equation}\label{approximate eigenvector of H}
	\| (H-E)\xi \|\leq \| T\xi\| +\| (\lambda V_{N_0}-E)\xi \| \leq \mcO_{+,\lambda}\left(\delta^{\frac{1}{6}} \right).
\end{equation}

\noindent \textcolor{blue}{\large \textbf{(Step 2: the periodic approximation and Floquet-Bloch decomposition.)}} \\
Next, we need to use the Floquet-Bloch theory and periodic approximation to extract the hidden structure of the potential in \eqref{approximate eigenvector of V} via the structure of the eigenvector $\xi$. For some preliminaries about the Floquet-Bloch theory, one can refer to \cite[Appendix B]{LSZ25}.

We use $(\cdot)_{N_0}$ to denote the restriction of an operator to $\Lambda_{N_0}$, and we will use $(\cdot)^{N_0}$ to denote its periodic extension (not the $N_0$-th power of the operator). Define 
\[
V^{N_0}=\sum_{n\in \Lambda_{N_0}} V(n) \left(\sum_{l\in [(2N_0+1)\Z]^d} \delta_{l+n,l+n}\right)
\]
as the periodic extension of the potential $V_{N_0}$, and denote 
\[
H^{N_0}=T+\lambda V^{N_0}.
\]
Since $\xi$ is supported in $\Lambda_{N_0}$, we have $H\xi=H^{N_0}\xi$, and therefore \eqref{approximate eigenvector of H} yields 
\[
\|(H^{N_0}-E)\xi \|\leq \mcO_{+,\lambda}\left(  \delta^{\frac{1}{6}}\right).
\]
This means
\begin{equation}\label{spec information on periodic extension from approximate on H}
	\operatorname{spec}(H^{N_0})\cap \left[ E-\mcO_{+,\lambda}\left(  \delta^{\frac{1}{6}}\right),\; E+\mcO_{+,\lambda}\left(  \delta^{\frac{1}{6}}\right)\right] \neq \varnothing.
\end{equation}
Since both $T$ and $\lambda V^{N_0}$ are positive operators, \eqref{spec information on periodic extension from approximate on H} together with $E\in [0,\delta]$ implies
\begin{equation}\label{spec information on periodic extension}
	\operatorname{spec}(H^{N_0})\cap \left[ 0,\; \mcO_{+,\lambda}\left(  \delta^{\frac{1}{6}}\right)\right] \neq \varnothing.
\end{equation}
Now denote $h(\theta)=\hatf(\theta)$. Via Floquet-Bloch theory, we can decompose $H^{N_0}$ into fiber matrices 
\[
H^{N_0}\cong_{\text{unitarily}} \bigoplus\nolimits_{y\in \left(\frac{1}{2N_0+1} \T\right)^d} M^{N_0}(y),
\]
where 
\[
M^{N_0}(y)=\left(h_{k-j}(y)\right)_{\Lambda_{N_0}\times \Lambda_{N_0}} +\lambda V_{N_0}, \quad
h_k(y)=\sum_{l\in [(2N_0+1)\Z]^d } f(k+l) e^{2\pi i l\cdot y}
\]
is the fiber matrix defined on $\ell^2(\Lambda_{N_0})$ with respect to the Floquet quasi-momentum $y\in \left(\frac{1}{2N_0+1} \T\right)^d$. Therefore,
\[
\spec(H^{N_0})=\bigcup\nolimits_{y\in \left(\frac{1}{2N_0+1} \T\right)^d} \spec(M^{N_0}(y)),
\]
and \eqref{spec information on periodic extension} implies that there exists some $x\in \left(\frac{1}{2N_0+1} \T\right)^d$ such that 
\begin{equation}\label{spec information on fiber matrix at x}
		\spec(M^{N_0}(x))\cap \left[ 0,\; \mcO_{+,\lambda}\left(  \delta^{\frac{1}{6}}\right)\right] \neq \varnothing.
\end{equation}
For simplicity, we denote $P= \left(h_{k-j}(x)\right)_{\Lambda_{N_0}\times \Lambda_{N_0}}$ and therefore $M^{N_0}(x)=P+ \lambda V_{N_0}$. Since $M^{N_0}(x)$ is self-adjoint, \eqref{spec information on fiber matrix at x} enables us to find a unit vector $a\in \ell^2(\La_{N_0})$ such that 
\begin{align}\label{approximate eigenvector of fiber matrix}
	\notag 0 &\leq \langle a,M^{N_0}(x) a\rangle_{\ell^2 (\La_{N_0})}  \\
	   &=   \langle a,P a\rangle_{\ell^2 (\La_{N_0})} + \langle a,\lambda V_{N_0} a\rangle_{\ell^2 (\La_{N_0}) } \leq \mcO_{+,\lambda}\left( \delta^{\frac{1}{6}}\right). 
\end{align}
Now, since $h=\hatf\geq 0$, Floquet-Bloch theory tells us that $P$ has the following positive Floquet eigenvalues with corresponding Floquet eigenbasis:
\begin{equation}\label{Floquet eigenvalue and Floquet eigenbasis}
	E_s(x) := h\left( x+\frac{s}{2N_0+1} \right) \ \Longleftrightarrow \ \beta_s(x) := \frac{1}{(2N_0+1)^{d/2}} \sum_{k\in \Lambda_{N_0}} e^{-2\pi i  k\cdot \left(x+ \frac{s}{2N_0+1}\right)} \delta_k,\; s\in \La_{N_0}.
\end{equation}
Here we denote $\delta_k$ as the standard orthonormal basis on $\ell^2(\Lambda_{N_0})$, which can conversely be represented via $\beta_s(x)$ by 
\begin{equation}\label{represent standard basis by Floquet basis}
	\delta_k= \frac{1}{(2N_0+1)^{d/2}} \sum_{s\in \Lambda_{N_0}} e^{2\pi i  k\cdot \left(x+ \frac{s}{2N_0+1}\right)} \beta_s(x).
\end{equation}
\eqref{Floquet eigenvalue and Floquet eigenbasis} implies that $P$ is a positive operator, and consequently \eqref{approximate eigenvector of fiber matrix} implies that the vector $a$, just like $\xi$, is a good approximate eigenvector of both $P$ and $V$:
\begin{equation}\label{approximate eigenvector of P}
	0 \leq  \langle a,P a\rangle_{\ell^2 (\La_{N_0})}  \leq \mcO_{+,\lambda}\left( \delta^{\frac{1}{6}}\right),
\end{equation}
\begin{equation}\label{approximate eigenvector of periodic V}
	0\leq  \langle a,\lambda V_{N_0} a\rangle_{\ell^2 (\La_{N_0}) } \leq \mcO_{+,\lambda}\left( \delta^{\frac{1}{6}}\right).
\end{equation}
Since the Floquet eigenbasis $\{\beta_s(x)\}_{s\in \La_{N_0}}$ diagonalizes $P$, we investigate \eqref{approximate eigenvector of P} by writing 
\[
a = \sum_{k\in \Lambda_{N_0}} a_k \cdot \beta_k(x) := (a_k)^{\rm FB}_{k\in \La_{N_0}}.
\]
Here we denote by $(\cdot)^{\rm FB}_{k\in \Lambda_{N_0}}$ the coefficients with respect to the basis $\beta_s(x)$, and by $(\cdot)_{k\in\Lambda_{N_0}}$ the coefficients with respect to the basis $\delta_k$. Moreover, we define the rotated standard orthonormal basis by 
\begin{equation}\label{rotated standard basis}
v_l = e^{-2\pi i l\cdot x} \delta_l = \left( \frac{1}{(2N_0+1)^{d/2}} e^{2\pi i l\cdot \frac{k}{2N_0+1}} \right)^{\rm FB}_{k\in \La_{N_0}}, \ l\in \La_{N_0}.
\end{equation}

\noindent \textcolor{blue}{\large \textbf{(Step 3: apply the quantitative uncertainty principle.)}} \\
From now on, for simplicity we still denote the inner product $\langle\cdot,\cdot\rangle_{\ell^2(\La_{N_0})}$ by $\langle\cdot,\cdot\rangle$, and $\| \cdot \|_{\ell^2(\La_{N_0})}$ by $\|\cdot \|$ when no confusion arises. Now let the following scales (positive integers) be determined later: 
\begin{equation}\label{scales in uncertainty principle}
	1\ll L'\ll L, \ 1\ll K\ll K',\ (2N_0+1)=(2L+1)(2K+1)=(2L'+1)(2K'+1).
\end{equation}
For each $1\leq j\leq J$, take $k_j\in \Z^d$ to be the lattice point closest to $(2N_0+1)\theta_j$, i.e., 
\[
|k_j - (2N_0+1)\theta_j| \leq \frac{1}{2}.
\]
Construct the cut-off of $a$ as follows: for each $1\leq j\leq J$,
\begin{equation}\label{the vector a^j}
	a^{(j)}=(a^{(j)}_{k})^{\rm FB}_{k\in \Lambda_{N_0}}, \qquad 
	a^{(j)}_k= \begin{cases}
		a_k, & {\rm if} \ |k-k_j|\leq K;\\
		0, & {\rm otherwise.}
	\end{cases}
\end{equation}
(Here, more explicitly, we mean that $k\in\Lambda_{N_0}\simeq \left(\Z/(2N_0+1)\Z\right)^d$ and $|\cdot|$ is viewed on the discrete torus.) We claim that the supports of $a^{(j)}$ for $1\leq j\leq J$ under the Floquet eigenbasis are mutually disjoint. This is because, by our choice of $k_j$, for $j\neq j'$ we have 
\begin{align*}
      |k_j-k_{j'}|& \geq |(2N_0+1)(\theta_j-\theta_{j'})|-|k_j-(2N_0+1)\theta_j|-|k_{j'}-(2N_0+1)\theta_{j'}|\\ 
          &\geq \min_{1\leq j\neq j'\leq J}|\theta_j-\theta_{j'}|\cdot (2N_0+1)-1 \gtrsim_T N_0\gg K. 
\end{align*}
Now we define the vector 
\begin{equation}\label{hat a}
    \widehat{a}= a-\sum_{1\leq j\leq J}a^{(j)},
\end{equation}
which, under the Floquet eigenbasis, is supported in 
\begin{align*}
      \left\{k:\ |k-k_j|>K,\ \forall 1\leq j\leq J \right\} &\subset  \left\{k:\ |k-(2N_0+1)\theta_j|>K-\tfrac{1}{2},\ \forall 1\leq j\leq J \right\}\\  
      &\subset \left\{k:\ \left\| \tfrac{k}{2N_0+1}-\theta_j\right\|_{\T^d} >\tfrac{K-\frac{1}{2}}{2N_0+1},\ \forall 1\leq j\leq J \right\} \\
      &\subset \left\{k:\ \left\|\tfrac{k}{2N_0+1}-\theta_j\right\|_{\T^d}>\tfrac{1}{3(2L+1)},\ \forall 1\leq j\leq J\right\} :=\mathcal K.
\end{align*}
Since $x\in \left(\frac{1}{2N_0+1} \T\right)^d= \left[-\frac{1}{2(2N_0+1)},\frac{1}{2(2N_0+1)} \right]^d$, for any $k\in \mcK$, we have 
\[\min_{1\leq j\leq J } \left\|x+\frac{k}{2N_0+1}-\theta_j \right\|_{\T^d}>\frac{1}{3(2L+1)}-\frac{1}{2(2N_0+1)} \geq \frac{1}{6(2L+1)},\]
and by \textbf{(A4)}, $h=\hatf$ and \eqref{Floquet eigenvalue and Floquet eigenbasis}, the corresponding Floquet eigenvalues on $\mcK$ satisfy
\begin{equation}\label{Floquet eigenvalue in mcK}
    E_k(x)=h \left(x+\frac{k}{2N_0+1}\right) \geq D_T \cdot \left(\frac{1}{6(L+1)}\right)^{d_T}, \forall k\in \mathcal K. 
\end{equation}
Therefore,
\begin{align*}
      \mcO_{+,\lambda}\left(\delta^{\frac16}\right)&\geq  \langle a, P a\rangle  \\
           & = \left\langle \widehat{a} ,P  \widehat{a}\right\rangle+  \left\langle \sum_{1\leq j\leq J}a^{(j)} ,P \sum_{1\leq j\leq J}a^{(j)} \right\rangle \\
           &\geq  \left\langle \widehat{a} ,P  \widehat{a}\right\rangle  \  \geq \ \min_{k\in \mathcal K} E_k(x)\cdot  \left\|  \widehat{a} \right\| ^2 \\
           &\geq D_T \cdot\left(\frac{1}{6(2L+1)}\right)^{d_T}  \left\|  \widehat{a} \right\| ^2,
\end{align*} 
which is equivalent to 
\begin{equation}\label{concentration of widehat a}
	\| \widehat{a} \| \leq \mcO_{+,\lambda,T} \left(\delta^{\frac{1}{12}} L^{\frac{d_T}{2}}\right).
\end{equation}
\eqref{concentration of widehat a} means that $a$ is concentrated in the union of the supports of $a^{(j)}$, which in some sense can be viewed as the neighbourhoods of $\theta_j$ on $\Lambda_{N_0}$. Substituting \eqref{concentration of widehat a} back into \eqref{approximate eigenvector of periodic V}, we obtain 
\begin{align}\label{the relation of a^j and potential}
	\notag  0& \leq \left\langle \sum_{1\leq j\leq J} a^{(j)},\lambda V_{N_0} \sum_{1\leq j\leq J} a^{(j)} \right\rangle = \left\langle a-\widehat{a},\lambda V_{N_0} (a-\wha)\right\rangle \\
	\notag    &= \left\langle a,\lambda V_{N_0} a\right\rangle +\left\langle \widehat{a},\lambda V_{N_0} \wha\right\rangle - 2\left\langle \wha,\lambda V_{N_0} a\right\rangle\\
	       &\leq \mcO_{+,\lambda} \left(\delta^{\frac{1}{6}} \right) + \mcO_{+,\lambda,T} \left(\delta^{\frac{1}{12}} L^{\frac{d_T}{2}}\right) = \mcO_{+,\lambda,T} \left(\delta^{\frac{1}{12}} L^{\frac{d_T}{2}}\right).
\end{align}
By our construction, $a^{(j)}$ is supported near $k_j$ under the Floquet eigenbasis, but the basis that diagonalizes $V_{N_0}$ is the (rotated) standard eigenbasis; therefore we must represent $a^{(j)}$ under the basis \eqref{rotated standard basis} as 
\begin{equation}\label{represent a^j by rotated standard basis}
	a^{(j)}=\sum_{l\in\La_{N_0}} \langle v_l,a^{(j)}\rangle v_l.
\end{equation}

In this representation, \eqref{the relation of a^j and potential} becomes
{\small

\begin{equation}\label{LDT event, without centering}
	0\leq \lambda\sum_{j=1}^{J} \sum_{l\in \Lambda_{N_0}}  V(l) \cdot |\langle v_l,a^{(j)}\rangle|^2 
                                      +2\lambda \ \Re \left( \sum_{1\leq j< j'\leq J} \sum_{l\in \Lambda_{N_0}}  V(l) \cdot \langle v_l,a^{(j)} \rangle \cdot \overline{\langle v_l,a^{(j')}\rangle}    \right) \leq \mcO_{+,\lambda,T} \left(\delta^{\frac{1}{12}} L^{\frac{d_T}{2}}\right) . 
\end{equation}
}
However, since \eqref{Bernoulli potential} gives $\E V(n)=1/2\neq 0$, we need to center the random potential in order to investigate hypercontractivity in the following steps. Let 
\begin{equation}\label{centering random potential}
	V=\frac{1-\wtV}{2}, \qquad \P(\wtV(n)=\pm 1)=\frac{1}{2} \quad \text{i.i.d.}
\end{equation}

Then $\E\wtV(n)=0$ and \eqref{the relation of a^j and potential} becomes 
\begin{align*}
	\left\langle \sum_{1\leq j\leq J} a^{(j)},\lambda \wtV_{N_0} \sum_{1\leq j\leq J} a^{(j)} \right\rangle & \geq \lambda \|a-\wha \|^2- \mcO_{+,\lambda,T} \left(\delta^{\frac{1}{12}} L^{\frac{d_T}{2}}\right) \\
	                &\geq \lambda(1-\| \wha\|^2)-\mcO_{+,\lambda,T} \left(\delta^{\frac{1}{12}} L^{\frac{d_T}{2}}\right)\\
					&\overset{\text{\eqref{concentration of widehat a}}}{\geq} \lambda -\mcO_{+,\lambda,T} \left(\delta^{\frac{1}{12}} L^{\frac{d_T}{2}}\right),
\end{align*}
which, again under the representation \eqref{represent a^j by rotated standard basis}, reads 
{\small
\begin{equation}\label{LDT event, centering}
	\lambda\sum_{j=1}^{J} \sum_{l\in \Lambda_{N_0}}  \wtV(l) \cdot |\langle v_l,a^{(j)}\rangle|^2 
                                      +2\lambda \ \Re \left( \sum_{1\leq j< j'\leq J} \sum_{l\in \Lambda_{N_0}}  \wtV(l) \cdot \langle v_l,a^{(j)} \rangle \cdot \overline{\langle v_l,a^{(j')}\rangle}    \right) \geq\lambda - \mcO_{+,\lambda,T} \left(\delta^{\frac{1}{12}} L^{\frac{d_T}{2}}\right) . 
\end{equation}
}

Since $\supp(a^{(j)})= \Lambda_K(k_j)$ (under the Floquet basis), we construct another $\widetilde{a}^{(j)}$ by shifting the center $k_j$ of the support to the origin, namely,
\begin{equation}\label{widetilde a^j}
    \widetilde{a}^{(j)}=\sum_{k\in \Lambda_K}a_{k+k_j} \cdot \beta_k(x) =(a^{(j)}_{k+k_j})^{\rm FB}_{k\in \La_{N_0}}.
\end{equation}
By \eqref{rotated standard basis} and \eqref{widetilde a^j}, simple computations show that 
\begin{equation}\label{rotation relationship between coefficients of a^j and widetilde a^j}
    \langle v_l,a^{(j)}\rangle = e^{-2\pi i l\cdot \frac{k_j}{2N_0+1} } \langle v_l,\widetilde{a}^{(j)}\rangle.
\end{equation}
Substituting \eqref{rotation relationship between coefficients of a^j and widetilde a^j} into \eqref{LDT event, centering} yields 
{\small
\begin{align}\label{LDT under coefficients of widetilde a^j, centering}
 \notag     \sum_{j=1}^{J} \sum_{l\in \Lambda_{N_0}} \wtV(l) \cdot |\langle v_l,\widetilde{a}^{(j)}\rangle|^2   +  2\Re \left( \sum_{1\leq j< j'\leq J} \sum_{l\in \Lambda_{N_0}} e^{-2\pi i  l\cdot\frac{k_j-k_{j'}}{2N_0+1}}  \ \wtV(l) \cdot \langle v_l,\widetilde{a}^{(j)}\rangle  \cdot \overline{\langle v_l,\wta^{(j')}\rangle}    \right) \\
         \hfill  \geq 1- \mcO_{+,\lambda,T} \left(\delta^{\frac{1}{12}} L^{\frac{d_T}{2}}\right) .
\end{align}
}

Next, we apply the uncertainty principle Lemma \ref{discrete uncertainty principle} to each $\wta^{(j)}$. From \eqref{rotated standard basis} and \eqref{discrete Fourier transform}, it is easy to see that 
\[
\langle v_l,\wta^{(j)}\rangle = (\mcF_{N_0} \wta^{(j)})_l
\]
is exactly the discrete Fourier transform of $\wta^{(j)}$ (under the coefficients of the Floquet eigenbasis) on $\La_{N_0}\simeq \Z^d_{2N_0+1}$.
Recall that we assumed \eqref{scales in uncertainty principle}. Since $\supp(\wta^{(j)}) \subset \Lambda_K$ for $1\leq j\leq J$, applying Lemma \ref{discrete uncertainty principle} yields some $b^{(j)}\in \ell^2(\Lambda_{N_0})$ such that
\begin{enumerate}
	\item $\|\wta^{(j)}\| = \|b^{(j)}\|$;
	\item $\|\wta^{(j)} - b^{(j)}\| \leq \mcO_+ (K/K')$;
	\item For $l' \in \Lambda_{L'}$ and $k' \in \Lambda_{K'}$, we have $\langle v_{l'+k'(2L'+1)}, b^{(j)}\rangle = \langle v_{k'(2L'+1)}, b^{(j)}\rangle$.
\end{enumerate}
Substituting $\wta^{(j)}$ by $b^{(j)}$ in \eqref{LDT under coefficients of widetilde a^j, centering} and using the Cauchy-Schwarz inequality yields (since the number of minimum points $J$ depends essentially on $T$)  
\begin{align}\label{LDT under coefficients of b^j, centering}
   \notag \sum_{j=1}^{J} \sum_{l\in \Lambda_{N_0}} \wtV(l) \cdot |\langle v_l,b^{(j)} \rangle |^2  & +2\Re \left( \sum_{1\leq j< j'\leq J} \sum_{l\in \Lambda_{N_0}} e^{-2\pi i  l\cdot\frac{k_j-k_{j'}}{2N_0+1}} \ \wtV(l) \cdot \langle v_l,b^{(j)} \rangle  \cdot \overline{\langle v_l,b^{(j')}\rangle}    \right) \\
      \notag     &\geq 1- \mcO_{+,\lambda,T} \left(\delta^{\frac{1}{12}} L^{\frac{d_T}{2}}\right)-3 J^2 \cdot \| \wtV\|_{\infty} \cdot \max_{1\leq j\leq J} \|\wta^{(j)} - b^{(j)}\| \\
       &\overset{\text{property (2) of $b^{(j)}$}}{\geq}  1- \mcO_{+,\lambda,T} \left(\delta^{\frac{1}{12}} L^{\frac{d_T}{2}}\right)- \mcO_{+,T} \left(K/K' \right).
\end{align}
Moreover, on the left-hand side of \eqref{LDT under coefficients of b^j, centering}, writing the summation index uniquely as 
\[\Lambda_{N_0}\ni l = l' + k'(2L'+1),\quad l'\in \Lambda_{L'},\; k'\in \Lambda_{K'}\]
and applying property (3) of $b^{(j)}$ yields 
{\small
\begin{align}\label{LDT under mcS, centering}
  \notag  \sum_{j=1}^{J} \sum_{k'\in \Lambda_{K'}} &  \mcS(j,j,k')  \cdot (2L'+1)^d |\langle v_{k'(2L'+1)},b^{(j)}\rangle |^2  \\
    \notag    +&\sum_{1\leq j< j'\leq J} \sum_{k'\in \Lambda_{K'}} 2\Re \left( \mcS(j,j',k')  \cdot (2L'+1)^d e^{-2\pi i k'\cdot \frac{k_j-k_{j'}}{2K'+1}} \langle v_{k'(2L'+1)},b^{(j)} \rangle \cdot \overline{\langle v_{k'(2L'+1)},b^{(j')}\rangle } \right)  \\
       &\geq  1- \mcO_{+,\lambda,T} \left(\delta^{\frac{1}{12}} L^{\frac{d_T}{2}}\right)- \mcO_{+,T} \left(K/K' \right),
  \end{align}
}
where 
\begin{equation}\label{mcS summation}
  \mcS(j,j',k')= \frac{1}{(2L'+1)^d}\sum_{l'\in \Lambda_{L'}} e^{-2\pi i l'\cdot \frac{k_j-k_{j'}}{2N_0+1}} \  \wtV(l'+k'(2L'+1)). 
\end{equation}
In addition, the left-hand side of \eqref{LDT under mcS, centering} can be controlled by: 
\begin{equation}\label{abs control of LDT under mcS}
  \text{LHS of } \eqref{LDT under mcS, centering} \leq \sum_{1\le j,j'\leq J} \sum_{k'\in \Lambda_{K'}}  |\mcS(j,j',k')|\cdot (2L'+1)^d \left|\langle v_{k'(2L'+1)},b^{(j)}\rangle \right| \cdot \left| \overline{\langle v_{k'(2L'+1)},b^{(j')}\rangle}\right|. 
\end{equation}
Now, recalling $|\frac{k_j}{2N_0+1}-\theta_j|\leq \frac{1}{2(2N_0+1)}$, we can define
\begin{equation}\label{mcN summation}
  \mcN(j,j',k')=\frac{1}{(2L'+1)^d}\sum_{l'\in \Lambda_{L'}} e^{-2\pi i  l' \cdot(\theta_j-\theta_{j'}) } \ \wtV(l'+k'(2L'+1)),
\end{equation}
and thus for each $j,j'$ and $k'$,
\begin{align}\label{difference between mcS and mcN}
   \notag   |\mcS(j,j',k')-\mcN(j,j',k')| & \lesssim_d \frac{1}{(2L'+1)^d} \sum_{l'\in \Lambda_{L'}}|l'|\cdot \left |\frac{k_j-k_{j'}}{2N_0+1}-(\theta_j-\theta_{j'})\right| \\
                         & = \mcO_{+,d}\left(\frac{L'}{2N_0+1}\right) = \mcO_{+,d} \left(\frac{1}{K'}\right).
\end{align}
So combining \eqref{LDT under mcS, centering}, \eqref{abs control of LDT under mcS} and \eqref{difference between mcS and mcN} shows  
\begin{align*}
 &\sum_{1\le j,j'\leq J} \sum_{k'\in \Lambda_{K'}}  |\mcN(j,j',k')|\cdot (2L'+1)^d \left|\langle v_{k'(2L'+1)},b^{(j)}\rangle \right| \cdot \left| \overline{\langle v_{k'(2L'+1)},b^{(j')}\rangle}\right|\\
   &\geq 1- \mcO_{+,\lambda,T} \left(\delta^{\frac{1}{12}} L^{\frac{d_T}{2}}\right)- \mcO_{+,T} \left(K/K' \right)\\
                        &\qquad -\mcO_{+,d}\left(\frac{1}{K'}\right) \cdot \sup_{1\leq j,j'\leq J} \sum_{k'\in \Lambda_{K'}}  (2L'+1)^d \left|\langle v_{k'(2L'+1)},b^{(j)}\rangle \right| \cdot \left| \overline{\langle v_{k'(2L'+1)},b^{(j')}\rangle}\right|.   
\end{align*}
It can be easily seen from properties (1) and (3) of $b^{(j)}$ that 
\begin{equation}\label{weight on b^j}
\sup_{1\leq j,j'\leq J} \sum_{k'\in \Lambda_{K'}}  (2L'+1)^d \left|\langle v_{k'(2L'+1)},b^{(j)}\rangle \right| \cdot \left| \overline{\langle v_{k'(2L'+1)},b^{(j')}\rangle}\right|   \leq \sup_{1\leq j,j'\leq J} \|\wta^{(j)}\| \cdot \|\wta^{(j')}\| \leq 1.
\end{equation}
Thus,
\begin{align}\label{LDT under mcN, centering}
 \notag \sum_{1\le j,j'\leq J} \sum_{k'\in \Lambda_{K'}} & |\mcN(j,j',k')|\cdot (2L'+1)^d \left|\langle v_{k'(2L'+1)},b^{(j)}\rangle \right| \cdot \left| \overline{\langle v_{k'(2L'+1)},b^{(j')}\rangle}\right|\\
   &\geq 1- \mcO_{+,\lambda,T} \left(\delta^{\frac{1}{12}} L^{\frac{d_T}{2}}\right)- \mcO_{+,T} \left(K/K' \right)  -\mcO_{+,d}\left(\frac{1}{K'}\right) .
\end{align}
Finally, by \eqref{weight on b^j} we have 
\begin{equation}
	 \notag \sum_{1\le j,j'\leq J} \sum_{k'\in \Lambda_{K'}} (2L'+1)^d \left|\langle v_{k'(2L'+1)},b^{(j)}\rangle \right| \cdot \left| \overline{\langle v_{k'(2L'+1)},b^{(j')}\rangle}\right| \leq J^2,
\end{equation}
and therefore by the pigeonhole principle applied to \eqref{LDT under mcN, centering} we finally deduce that 
\begin{equation}\label{final deduction on mcN}
  \sup_{1\leq j,j'\leq J \atop k'\in \Lambda_{K'}} |\mcN(j,j',k')|\geq \frac{1}{J^2} \left(    1- \mcO_{+,\lambda,T} \left(\delta^{\frac{1}{12}} L^{\frac{d_T}{2}}\right)- \mcO_{+,T} \left(K/K' \right)  -\mcO_{+,d}\left(\frac{1}{K'}\right)      \right)
\end{equation}
under the assumption \eqref{reverse assumption in initial scale}.

\noindent \textcolor{blue}{\large \textbf{(Step 4: determine the scales and parameters.)}} \\
In view of \eqref{final deduction on mcN}, we must choose the parameters $\delta,N_0$ and the scales in \eqref{scales in uncertainty principle} as follows: 
  \begin{enumerate}
    \item $2N_0+1=(2L+1)(2K+1)=(2L'+1)(2K'+1)$;
    \item $K/K'\ll_T 1$, $K'\gg_d 1$;
    \item $\delta^{\frac{1}{12}} L^{\frac{d_T}{2}}\ll_{\lambda,T} 1$. 
  \end{enumerate} 
 The above conditions can be fulfilled by taking the initial scale $N_{\rm in}\gg_{\lambda,T,d,C_{\rm in}} 1$ and by the following choice of parameters and scales:
  \begin{itemize} 
	\item $\delta=(\log N_{\rm in})^{-6000d_T}\ll_{\lambda,T,d} 1$;
    \item $L=\lfloor \delta^{-\frac{1}{24}\cdot \frac{2}{d_T}}\rfloor \sim (\log N_{\rm in})^{500} \ \Longrightarrow \delta^{\frac{1}{12}} L^{\frac{d_T}{2}} \sim  \delta^{\frac{1}{24}} =(\log N_{\rm in})^{-250 d_T}\ll_{\lambda,T} 1$;
    \item $L'=\lfloor\delta^{-\frac{1}{48}\cdot \frac{2}{d_T}} \rfloor \sim (\log N_{\rm in})^{250} \ \Longrightarrow \frac{K}{K'}\sim \frac{L'}{L}\sim \delta^{\frac{1}{48}\cdot \frac{2}{d_T}}\sim (\log N_{\rm in})^{-250} \ll_T 1$;
    \item $N_0\in {\rm Scale}_{\rm in}$ with 
          \[  {\rm Scale}_{\rm in}:=\left\{N\in\Z:\ N_{\rm in}\leq N\leq N_{\rm in}^{C_{\rm in}}, \ \frac{2N+1}{(2L'+1)(2L+1)}\in \Z_+  \right\}, \]
		  so that $2N_0+1$ satisfies the divisibility condition (1) with $(2L+1)$ and $(2L'+1)$;
	\item $K\sim\frac{N_0}{L} \sim_{C_{\rm in}} N_0 (\log N_0)^{-500}, \ K'\sim \frac{N_0}{L'}\sim_{C_{\rm in}} N_0 (\log N_0)^{-250}\gg_d 1$.
   \end{itemize}
  With the above chosen parameters and scales, \eqref{final deduction on mcN} becomes 
  \begin{equation}\label{final deduction on mcN, after choice of parameters}
  \sup_{1\leq j,j'\leq J \atop k'\in \Lambda_{K'}} |\mcN(j,j',k')|\geq \frac{1}{2J^2}.
 \end{equation}

\noindent \textcolor{blue}{\large \textbf{(Step 5: estimate the probability.)}} \\
Since we have already shown that \eqref{reverse assumption in initial scale} implies \eqref{final deduction on mcN, after choice of parameters}, we define 
\begin{equation}\label{final event Omega_N0 at initial scale}
  \Xi_{N_0}=\left\{ \omega :\  \sup_{1\leq j,j'\leq J \atop k'\in \Lambda_{K'}}|\mcN(j,j',k')|\geq \frac{1}{2J^2}   \right\},
\end{equation}
which depends only on the randomness in $\La_{N_0}$, and therefore
\begin{equation}\label{reverse assumption deduce Omega_N0}
  \left\{ \omega: \exists \ E \in [0,\delta] \ {\rm such \ that \ \eqref{reverse assumption in initial scale} \ holds} \right\} \subset \Xi_{N_0}. 
\end{equation}
Next we use some hypercontractivity inequality for sub-Gaussian distributions to estimate the probability of $\Xi_{N_0}$. For some preliminary knowledge, one can refer to \cite[Appendix D]{LSZ25}.

Now recall the centered random potential \eqref{centering random potential} and \eqref{mcN summation}. The standard Dudley $L^{\psi_2}$ estimate (\cite[Theorem D.2]{LSZ25}) yields that  
\begin{equation*}
	  \left\|  \sup_{1\leq j,j'\leq J \atop k'\in \Lambda_K'}|\mcN(j,j',k')|  \right\|_{\psi_2}  \lesssim \sqrt{\log(J^2 \cdot \#\Lambda_{K'})} \cdot \sup_{1\leq j,j'\leq J \atop k'\in \Lambda_K'} \|\mcN(j,j',k')\|_{\psi_2},
\end{equation*}
and the orthogonality of independent mean-zero sub-Gaussian random variables (\cite[Theorem D.3]{LSZ25}) yields that 
\begin{align*}
	\|\mcN(j,j',k')\|_{\psi_2} &= \frac{1}{(2L'+1)^d} \left\| \sum_{l'\in \Lambda_{L'}} e^{-2\pi i  l' \cdot(\theta_j-\theta_{j'}) } \ \wtV(l'+k'(2L'+1)) \right\|_{\psi_2} \\
	     &\lesssim (2L'+1)^{-d}  \left( \sum_{l'\in \La_{L'}} \|\wtV(l'+k'(2L'+1)) \|_{\psi_2}^2\right)^{\frac{1}{2}}\\
		 & \lesssim (2L'+1)^{-d/2}
\end{align*} 
for every $j,j'$ and $k'$. Thus,
\begin{equation}\label{Orlicz norm of sup of mcN}
	\left\|  \sup_{1\leq j,j'\leq J \atop k'\in \Lambda_K'}|\mcN(j,j',k')|  \right\|_{\psi_2} \lesssim_{T,d} \sqrt{\log K'} (2L'+1)^{-d/2},
\end{equation}
and the Hoeffding inequality (\cite[Theorem D.1]{LSZ25}), together with our choice of parameters, yields
\begin{equation}\label{probability estimate on Omega_N0}
	\P(\Xi_{N_0}) \leq 2\exp\left\{ -\frac{1/(2J^{2})}{\mcO_{+,T,d}\left(   \sqrt{\log K'} (2L'+1)^{-d/2}  \right)}\right\}\leq e^ { - (\log N_0)^{10} }.
\end{equation}

\noindent \textcolor{blue}{\large \textbf{(Step 6: return to the Neumann series expansion \eqref{Neumann expansion}.)}} \\
\eqref{reverse assumption deduce Omega_N0} reveals that for $\omega\notin \Xi_{N_0}$, 
\begin{equation}\label{the previous estimate to apply Neumann series expansion}
\left\| \left[  (1-T_{N_0})(\lambda V_{N_0}+1-E)^{-1} \right]^2 \right\| \leq 1-\delta		
\end{equation}
holds uniformly for all $E\in [0,\delta]$. Now, substituting \eqref{the previous estimate to apply Neumann series expansion} into the Neumann series expansion \eqref{Neumann expansion} and applying \eqref{about 1-T} and \eqref{about reverse potential} gives  
\begin{equation}\label{L2 norm from Neumann}
     \| G_{N_0} (E)\|   \leq \left( \frac{1}{1-\delta}+\frac{1}{(1-\delta)^2}\right)\cdot  \sum_{s\geq 0}(1-\delta)^s \lesssim \delta^{-1}\sim (\log N_0)^{6000d_T}.
\end{equation}
This proves \eqref{L2 norm estimate on Green's function, initial scale}.

Next, let $A$ be determined later and decompose the sum into the following two parts:
\begin{align*}
      |G_{N_0}(E)(n,n')| &\leq \sum_{s\geq 0} \left| (\lambda V_{N_0}+1-E)^{-1} \big((1-T_{N_0})(\lambda V_{N_0}+1-E)^{-1} \big)^s (n,n') \right|\\
           &= \left(\sum_{s< A}+\sum_{s \geq  A} \right)\cdots. 
\end{align*}
For the $s\geq A$ part, an argument similar to that used to deduce \eqref{L2 norm from Neumann} yields
\begin{equation}\label{estimate for s geq A}
	\sum_{s\geq A} \cdots \lesssim \sum_{s\geq A/2}(1-\delta)^s \lesssim \delta^{-1}(1-\delta)^{A/2} \leq \delta^{-1} e^{-\delta A/2}.
\end{equation}
For the $s< A$ part, assumption \textbf{(A2)} ensures \eqref{exp decay hopping}, and therefore  
\[
|(1-T_{N_0})(m,m')|\leq (C_T+1)e^{-c_T |m-m'|}.
\] 
Thus,
\begin{align*}
       &\ \ \  \left| (\lambda V_{N_0}+1-E)^{-1} \big((1-T_{N_0})(\lambda V_{N_0}+1-E)^{-1} \big)^s (n,n') \right| \\
          & \leq \left(\frac{1}{1-\delta} \right)^{s+1}  \sum_{n_1,n_2, \cdots,n_{s-1}\in \Lambda_{N_0}}  |(1-T_{N_0})(n,n_1)|\cdot |(1-T_{N_0})(n_1,n_2)|\cdots |(1-T_{N_0})(n_{s-1},n')| \\
          &\leq  \left(\frac{1}{1-\delta}\right)^{s+1} (C_T+1)^s  \sum_{n_1,n_2 \cdots,n_{s-1}\in \Lambda_{N_0}}  \exp\Bigl\{ -c_T\bigl(|n-n_1|+|n_1-n_2|+\cdots+|n_{s-1}-n'|\bigr) \Bigr\} \\
          &\lesssim \left(\frac{C_T+1}{1-\delta}\right)^s (\# \Lambda_{N_0})^{s-1} e^{-c_T|n-n'|}  \ \leq  \mcO_{+,T,d}\left(N_0^d\right)^s e^{-c_T |n-n'|},
\end{align*}
which implies
\begin{equation}\label{estimate for s < A}
      \sum_{s< A} \cdots  \leq \sum_{s < A} \mcO_{+,T,d}\left(N_0^d\right)^s  e^{-c_T|n-n'|} \leq  \mcO_{+,T,d}\left(N_0^d\right)^A e^{-c_T|n-n'|}.
\end{equation}
Combining \eqref{estimate for s geq A} and \eqref{estimate for s < A} and setting $A= \frac{N_0}{(\log N_0)^2}$ yields that for every $|n-n'|\geq \frac{N_0}{200}$, 
\begin{align}\label{off diagonal decay from Neumann}
    \notag  |G_{N_0}(E)(n,n')| &\lesssim \frac{1}{\delta} e^{-\delta A/2} + \mcO_{+,T,d}\left(N_0^d\right)^A e^{-c_T|n-n'|}\\
       \notag &= \exp\bigg\{ -\delta A/2 + 6000d_T \log\log N_0  \bigg\} + \exp \bigg\{ -c_T|n-n'| + A \cdot \mcO_{+,T,d}(\log N_0) \bigg\} \\
        & \leq \exp\bigg\{ -\frac{2N_0}{(\log N_0)^{7000 d_T}} \bigg\} \leq \exp\bigg\{ -\frac{|n-n'|}{(\log N_0)^{7000 d_T}} \bigg\},
\end{align}
where in the above inequality we used $N_0\geq N_{\rm in}\gg_{\lambda,T,d,C_{\rm in}} 1$ and $\frac{N_0}{200}\leq |n-n'|\leq 2N_0$. This proves \eqref{off diagonal decay estimate on Green's function, initial scale}.

\end{proof}

\subsection{Effect of multiple minima: comparison with previous works}\label{subsection: multiple minima}
We make some remarks below to compare the proof of Theorem \ref{LDT for the initial scale, without free sites} with the elliptic analysis at the initial scale in previous works.

It is easy to see from \textcolor{blue}{\textbf{(Step 6)}} that, in order to ensure \eqref{L2 norm estimate on Green's function, initial scale} and \eqref{off diagonal decay estimate on Green's function, initial scale}, it suffices to ensure \eqref{the previous estimate to apply Neumann series expansion}. Simultaneously, the proof also shows that the failure of \eqref{the previous estimate to apply Neumann series expansion} (i.e., the validity of \eqref{reverse assumption in initial scale}) implies \eqref{LDT event, without centering}.

Now, via an argument similar to that used to deduce \eqref{LDT under mcS, centering} from \eqref{LDT event, centering}, one may prove from \eqref{LDT event, without centering} that 
{\small
\begin{align}\label{LDT under mfS, without centering}
  \notag  0\leq \sum_{j=1}^{J}  &  \sum_{k'\in \Lambda_{K'}}   \mfS(j,j,k')  \cdot (2L'+1)^d |\langle v_{k'(2L'+1)},b^{(j)}\rangle |^2  \\
    \notag    +&\sum_{1\leq j< j'\leq J} \sum_{k'\in \Lambda_{K'}} 2\Re \left( \mfS(j,j',k')  \cdot (2L'+1)^d e^{-2\pi i k'\cdot \frac{k_j-k_{j'}}{2K'+1}} \langle v_{k'(2L'+1)},b^{(j)} \rangle \cdot \overline{\langle v_{k'(2L'+1)},b^{(j')}\rangle } \right)  \\
       &\leq  \mcO_{+,\lambda,T} \left(\delta^{\frac{1}{12}} L^{\frac{d_T}{2}}\right)+ \mcO_{+,T} \left(K/K' \right),
  \end{align}
} 
where 
\begin{equation}\label{mfS summation, without centering}
  \mathfrak{S}(j,j',k')= \frac{1}{(2L'+1)^d}\sum_{l'\in \Lambda_{L'}} e^{-2\pi i l'\cdot \frac{k_j-k_{j'}}{2N_0+1}} \  V(l'+k'(2L'+1)). 
\end{equation}
This is what we obtain \textcolor{red}{\textbf{without centering the random potential as \eqref{centering random potential}}}.

When the hopping operator $T$ is the simplest case, i.e., the free Laplacian on $\Z^d$ considered in \cite{BK05,DS20,LZ22}, an important feature is that its Fourier symbol has only a unique minimum point $\theta_1 = 0$, i.e., $J=1$ in \textbf{(A4)}. In this unique minimum case, \eqref{LDT under mfS, without centering} becomes 
\begin{equation}\label{LDT under mfS, without centering, unique minimum}
	0\leq \sum_{k'\in \Lambda_{K'}}   \mfS(1,1,k')  \cdot (2L'+1)^d |\langle v_{k'(2L'+1)},b^{(1)}\rangle |^2 \leq  \mcO_{+,\lambda,T} \left(\delta^{\frac{1}{12}} L^{\frac{d_T}{2}}\right)+ \mcO_{+,T} \left(K/K' \right),
\end{equation}
and \eqref{mfS summation, without centering} becomes
\[
\mfS(1,1,k')= \frac{1}{(2L'+1)^d}\sum_{l'\in \Lambda_{L'}}  V(l'+k'(2L'+1)). 
\]
Moreover, properties (1) and (3) of $b^{(1)}$, together with \eqref{concentration of widehat a}, yield
\begin{equation}\label{norm of b^1}
\sum_{k'\in \Lambda_{K'}}  (2L'+1)^d \left|\langle v_{k'(2L'+1)},b^{(1)}\rangle \right|^2 = \|\wta^{(1)}\|^2 = 1 - \mcO_{+,\lambda,T} \left(\delta^{\frac{1}{12}} L^{\frac{d_T}{2}}\right) \geq \frac{1}{2},
\end{equation}
Hence we can apply the pigeonhole principle to \eqref{LDT under mfS, without centering, unique minimum} together with \eqref{norm of b^1} to deduce 
\begin{equation}\label{final event without centering}
	\inf_{k'\in \La_{K'}} \mfS(1,1,k') \leq \mcO_{+,\lambda,T} \left(\delta^{\frac{1}{12}} L^{\frac{d_T}{2}}\right)+ \mcO_{+,T} \left(K/K' \right)\leq  (\log N_0)^{-200}
\end{equation}
by our choice of parameters and scales in \textcolor{blue}{\textbf{(Step 4)}}. In a nutshell, we have shown that in the unique minimum case, if \eqref{L2 norm estimate on Green's function, initial scale} or \eqref{off diagonal decay estimate on Green's function, initial scale} fails, then we obtain the event \eqref{final event without centering}. Hence, we have proved that 
\begin{equation}\label{average summation event in BK05, without centering}
	\inf_{k'\in \La_{K'}} \frac{1}{(2L'+1)^d}\sum_{n\in \Lambda_{L'}(k'(2L'+1))}  V(n) > (\log N_0)^{-200}
\end{equation}
implies the validity of \eqref{L2 norm estimate on Green's function, initial scale} and \eqref{off diagonal decay estimate on Green's function, initial scale}.

We compare \eqref{average summation event in BK05, without centering} with the condition to deduce the Green's function estimate at the initial scales in previous works: 
\begin{itemize}
	\item (Compared with \cite{BK05}) The event \eqref{average summation event in BK05, without centering} is exactly same with \cite[(4.13)]{BK05}.
	\item (Compared with \cite{DS20,LZ22}) We recall the following result on free Laplacian:
	  \begin{thm}\label{R net theorem from DS20}
			\textup{(\cite[Lemma 7.2]{DS20} and \cite[Corollary B.3]{LZ22})} Let $T$ be the free Laplacian on $\Z^d$. Assume that the set $\{V=1\}$ is an $R$-net in $\Lambda_{N_0}$. Then we have
			\begin{enumerate}
				\item The principal eigenvalue of $H_{N_0}$ satisfies 
				     \[E_{\rm prin} \left(H_{N_0}\right) \geq  L(\lambda,R,d),\]
					 where 
					 \begin{equation}\label{principle eigenvalue}
						 L(\lambda, R,d):=\begin{cases}
							 C_1(2) \lambda \cdot R^{-2}(\log R)^{-1}, & d=2;\\
							 C_1(d) \lambda \cdot R^{-d}, & d\geq 3
						 \end{cases}
					 \end{equation}
					 with some numerical constant $C_1(d)>0$ depending only on the dimension $d$.
				\item For any energy $0\leq E\leq L(\lambda,R,d)/2$, we have the Green's function estimate 
				    \begin{equation}\label{principle Green's function}
						\left| G_{N_0}(E) (x,y)\right|\leq \frac{2}{L(\lambda,R,d)}\exp \left\{ -c\,L(\lambda,R,d) |x-y| \right\} ,\ \forall x,y\in \La_{N_0}
					\end{equation}
					for some numerical constant $c>0$.
			\end{enumerate}
	  \end{thm}
         Now we show that, in some sense, \eqref{average summation event in BK05, without centering} is a weaker condition than the R-net condition in Theorem \ref{R net theorem from DS20}. Assume $\{V=1\}$ is an R-net in $\La_{N_0}$. Since by \eqref{Bernoulli potential} each $V(n)$ is either $1$ or $0$, we can deduce that  
          \begin{align*}
	         \inf_{k'\in \La_{K'}} \frac{1}{(2L'+1)^d}\sum_{n\in \Lambda_{L'}(k'(2L'+1))}  V(n) &= \inf_{k'\in \La_{K'}} \frac{\# \left(\Lambda_{L'}(k'(2L'+1))\cap \{V=1\}\right)}{\# \Lambda_{L'}(k'(2L'+1)) }\\
	           &\gtrsim_d R^{-d} >(\log N_0)^{-200}
         \end{align*}
            if we take $R=\lfloor (\log N_0)^{\frac{100}{d}}\rfloor$. Therefore, \eqref{average summation event in BK05, without centering} is fulfilled, and then the Green's function estimates \eqref{L2 norm estimate on Green's function, initial scale} and \eqref{off diagonal decay estimate on Green's function, initial scale} can be established.
\end{itemize}
However, the above discussion is valid only in the case of a unique minimum point; it breaks down when there are multiple minima. In the multiple minima case, \textcolor{red}{\textbf{resonances among the different minima give rise to complex-valued coupling coefficients}} $e^{-2\pi i k'\cdot \frac{k_j-k_{j'}}{2K'+1}}$ in \eqref{LDT under mfS, without centering}, which prevent us from applying the corresponding pigeonhole principle directly. This forces us to first center the random potential.\\

Similar issues arise in the R-net argument. Indeed, we prove (in Appendix \ref{appendix: proof of R net argument}) the following generalization of Theorem \ref{R net theorem from DS20} to all hopping operators $T$ (not only the free Laplacian) whose Fourier symbol has only a unique minimum.
\begin{thm}\label{R net theorem for general hopping with unique minimum}
			Let $T$ on $\Z^d$ satisfy \textup{\textbf{(A1), (A2)}}, and let its Fourier symbol $\hatf(\theta)\geq 0$ attain its minimum uniquely at $\hatf(0)=0$ with the elliptic condition
                      \[
                      \operatorname{Hess}(\hatf)(0)>0.
                      \]             
			Assume that the set $\{V=1\}$ is an $R$-net in $\Lambda_{N_0}$. Then the results in Theorem \ref{R net theorem from DS20} also hold.
\end{thm}
However, Theorem \ref{R net theorem for general hopping with unique minimum} also fails in the multiple minima case, because \textcolor{red}{\textbf{multiple minima also reduce the connectivity of the hopping acting on $\Z^d$}}. This can be seen from the following counterexample:

\begin{ctex}
Let
\[
Tu(n)=\sum_{j=1}^{d}\bigl(2u(n)-u(n+2\mathbf{e}_j)-u(n-2\mathbf{e}_j)\bigr),
\]
where $\mathbf{e}_j$, $1\leq j\leq d$, are the standard basis vectors of $\Z^d$. In this case, the Fourier symbol becomes
\[
\hatf(\theta)=\sum_{j=1}^{d}(2-2\cos(4\pi\theta_j)) \geq 0,
\]
which has $2^d$ distinct minima at $\{\theta:\theta_j=0 \text{ or } 1/2\}$. Indeed, the action of $T$ is disconnected on $\Z^d$: it has $2^d$ connected components, namely $n+(2\Z)^d\simeq \Z^d$ for each $n\in \{0,1\}^d$, and on each of them $T$ acts as the free Laplacian $\Delta_{\rm free}$ on $\Z^d$.

Now we take 
\[
V\equiv 1 \text{ on } \Z^d\setminus (2\Z)^d, \qquad V\equiv 0 \text{ on } (2\Z)^d.
\]
Then $\{V=1\}$ is a $1$-net in $\Z^d$. However, the operator $H=T+\lambda V$ acts as the free Laplacian on $(2\Z)^d$, and therefore for every $x,y\in \La_{N_0}\cap (2\Z)^d$, the Green's function has only the following decay rate:
\begin{equation*}
 |G_{N_0}(0)(x,y)|=\left|\Delta_{\mathrm{free},\lfloor N_0/2\rfloor}^{-1}(x,y) \right|  \lesssim_d \begin{cases}
	                          \log(N_0/|x-y|), & d=2; \\
							  |x-y|^{2-d}, & d\geq 3
 \end{cases}
\end{equation*}
by the standard results on the Dirichlet Laplacian. Thus, the exponential decay \eqref{principle Green's function} does not hold in this case, and Theorem \ref{R net theorem for general hopping with unique minimum} fails.

\end{ctex}
In light of the above discussion, it can be seen that, relatively speaking, our elliptic method and equation \eqref{LDT under mfS, without centering} are more intrinsic and robust in dealing with the multiple minima case.

\subsection{Free sites argument for initial scales in one dimension}
Now we return to the setting of Theorem \ref{Localization near the edge band}, namely, we consider the one-dimensional case $d=1$ and assume that $T\in \mathscr{R}$ satisfies \textbf{(A3)}. Then, by Remark \ref{remark on the assumption (A4)}(1), assumption \textbf{(A4)} is automatically fulfilled for $T$, and hence Theorem \ref{LDT for the initial scale, without free sites} (along with its proof) is applicable.

Recalling \eqref{max deg of P,Q}, we let $\Dbig(T)  =20 J^2 \cdot \Dhop(T) $. Set 
\begin{equation}\label{center of initial free sites}
	\wtS_{\rm in} := \Dbig(T)\Z \subset \Z, \qquad S_{\rm in}=\bigcup_{m\in \wtS_{\rm in}}[m-\Dhop(T),m+\Dhop(T)].
\end{equation}
We will show that the event $\Xi_{N_0}$ and the probability estimate in Theorem \ref{LDT for the initial scale, without free sites} can be made independent of the randomness on $S_{\rm in}$; hence we may take $S_{\rm in}$ as our initial set of free sites. Let $t\in [0,1]^{S_{\rm in}\cap \La_{N_0}}$ denote the variables on these sites, and let $\baromega \in \{0,1\}^{\La_{N_0}\setminus S_{\rm in}}$ denote the randomness outside $S_{\rm in}$. We will write $H_{N_0}(\baromega,t), G_{N_0}(\baromega,t)$, etc., for the corresponding operators with potential
\begin{equation}\label{potential under the initial free sites}
V(n)=
\begin{cases}
\baromega_n, & n\in \La_{N_0}\setminus S_{\rm in},\\
t_n, & n\in \La_{N_0}\cap S_{\rm in}.
\end{cases}
\end{equation}

\begin{thm}\label{LDT for the initial scale, free sites}
Assume $d=1$, $T\in \mathscr{R}$, and $T$ satisfies \textup{\textbf{(A3)}}. With the same choice of parameters and scales as in Theorem \ref{LDT for the initial scale, without free sites}, the following holds: for any $N_0\in  {\rm Scale}_{\rm in}$, and for $\baromega$ outside an event $\Xi'_{N_0}\subset \{0,1\}^{\La_{N_0}\setminus S_{\rm in}}$ that is independent of the randomness on $S_{\rm in}$ and has probability at most $e^{-(\log N_0)^{10}}$, we have 
\begin{equation}\label{L2 norm estimate on Green's function, initial scale, free site type}
	\|G_{N_0}(E;\baromega,t) \| \leq (\log N_0)^{7000 d_{T}}
\end{equation}
and 
\begin{equation}\label{off diagonal decay estimate on Green's function, initial scale, free site type}
	|G_{N_0}(E;\baromega,t)(n,n')| \leq  \exp \left\{  -\frac{|n-n'|}{(\log N_0)^{7000 d_T}} \right\}, \  \forall |n-n'|\geq \frac{N_0}{200}
\end{equation}
for every $E\in [0,\delta]$ and every $t\in  [0,1]^{S_{\rm in}\cap \La_{N_0}}$.
\end{thm}

\begin{proof}[Proof of Theorem \ref{LDT for the initial scale, free sites}]
	From \textcolor{blue}{\textbf{(Step 6)}} in the proof of Theorem \ref{LDT for the initial scale, without free sites}, in order to ensure \eqref{L2 norm estimate on Green's function, initial scale, free site type} and \eqref{off diagonal decay estimate on Green's function, initial scale, free site type}, it suffices to prove \eqref{the previous estimate to apply Neumann series expansion} under the configuration \eqref{potential under the initial free sites} for all $E\in [0,\delta]$ and $t\in  [0,1]^{S_{\rm in}\cap \La_{N_0}}$. Assume to the contrary that there exist some $E$ and $t$ such that 
    \begin{equation}\label{reverse assumption, free sites}
        \left\| \left[  (1-T_{N_0})(\lambda V_{N_0}+1-E)^{-1} \right]^2 \right\| >1-\delta		
    \end{equation}
	for the potential \eqref{potential under the initial free sites}. Then the proof proceeds exactly as in Theorem \ref{LDT for the initial scale, without free sites}, except that we will finally deduce from \eqref{reverse assumption, free sites} that (compared with \eqref{final deduction on mcN, after choice of parameters})	
     \begin{equation}\label{final deduction on mfN, free sites}
  \sup_{1\leq j,j'\leq J \atop k'\in \Lambda_{K'}} |\mfN(j,j',k')|\geq \frac{1}{2J^2}
 \end{equation} 
 with 
\begin{equation}\label{mfN summation, free sites}
  \mfN(j,j',k')=\frac{1}{2L'+1}\sum_{l'\in \Lambda_{L'}} e^{-2\pi i  l' \cdot(\theta_j-\theta_{j'}) } \ \wtV(l'+k'(2L'+1)),
\end{equation}
where in \eqref{mfN summation, free sites} the centered potential $\wtV$ is given by 
\begin{equation}\label{centering potential under the initial free sites}
\wtV(n)=
\begin{cases}
1-2\baromega_n, & n\in \La_{N_0}\setminus S_{\rm in},\\
1-2t_n, & n\in \La_{N_0}\cap S_{\rm in}.
\end{cases}
\end{equation}
Therefore, 
\[
\P(\wtV(n)=\pm 1) =\frac{1}{2} \quad \text{i.i.d. for } n\in  \La_{N_0}\setminus S_{\rm in}
\]
and $\| \wtV\|_{\ell^{\infty}(\La_{N_0}\cap S_{\rm in})} \leq 1$. Now for all $j,j'$ and $k'$, 
\begin{align*}
	&| \mfN(j,j',k') | = \left| \frac{1}{2L'+1} \left(\sum_{l'\in \Lambda_{L'}(k'(2L'+1)) \cap S_{\rm in}} +\sum_{l'\in \Lambda_{L'}(k'(2L'+1))\setminus S_{\rm in}} \right) e^{-2\pi i  l' \cdot(\theta_j-\theta_{j'}) } \ \wtV(l') \right| \\
	     \leq &\left| \frac{1}{2L'+1}  \sum_{l'\in \Lambda_{L'} (k'(2L'+1))\setminus S_{\rm in}}  e^{-2\pi i  l' \cdot(\theta_j-\theta_{j'}) } \ \wtV(l) \right|  +  \| \wtV\|_{\ell^{\infty}(\La_{N_0}\cap S_{\rm in})} \cdot \frac{\# (\La_{N_0}(k'(2L'+1))\cap S_{\rm in})}{\# \La_{N_0}(k'(2L'+1))}\\
		 	     \leq &\left| \frac{1}{2L'+1}  \sum_{l'\in \Lambda_{L'} (k'(2L'+1))\setminus S_{\rm in}}  e^{-2\pi i  l' \cdot(\theta_j-\theta_{j'}) } \ \wtV(l) \right|  +   5  \ \frac{\Dhop(T)}{\Dbig(T)} .
\end{align*}
Therefore, if we define 
\begin{equation}\label{wtmfN summation, free sites}
  \widetilde{\mfN} (j,j',k')=\frac{1}{2L'+1}  \sum_{l'\in \Lambda_{L'} (k'(2L'+1))\setminus S_{\rm in}}  e^{-2\pi i  l' \cdot(\theta_j-\theta_{j'}) } \ \wtV(l),
\end{equation}
then \eqref{final deduction on mfN, free sites} yields 
   \begin{equation}\label{final deduction on wtmfN, free sites}
  \sup_{1\leq j,j'\leq J \atop k'\in \Lambda_{K'}} |\widetilde{\mfN}(j,j',k')|\geq \frac{1}{2J^2} -  5\ \frac{\Dhop(T)}{\Dbig(T)}  \geq \frac{1}{4J^2}
\end{equation} 
since we choose $\Dbig (T)= 20J^2\cdot  \Dhop(T)$. Thus we take the event 
\begin{equation}\label{final event Omega'_N0, free sites}
  \Xi'_{N_0}=\left\{ \omega :\  \sup_{1\leq j,j'\leq J \atop k'\in \Lambda_{K'}}|\widetilde{\mfN} (j,j',k')|\geq \frac{1}{4J^2}   \right\} 
\end{equation} 
which is independent of the randomness on $S_{\rm in}$. By the same probability estimate as in \textcolor{blue}{\textbf{(Step 5)}} of the proof of Theorem \ref{LDT for the initial scale, without free sites}, we have 
\begin{equation}\label{Orlicz norm of sup of wtmfN}
	\left\|  \sup_{1\leq j,j'\leq J \atop k'\in \Lambda_K'}|\widetilde{\mfN}(j,j',k')|  \right\|_{\psi_2} \lesssim_{T} \sqrt{\log K'} (2L'+1)^{-1/2},
\end{equation}
and 
\begin{equation}\label{probability estimate on Omega'_N0}
	\P(\Xi'_{N_0}) \leq 2\exp\left\{ -\frac{1/(4J^{2})}{\mcO_{+,T}\left(   \sqrt{\log K'} (2L'+1)^{-1/2}  \right)}\right\}\leq e^ { - (\log N_0)^{10} }.
\end{equation}
This completes the proof.
\end{proof}

\section{The large scale: multi-scale analysis}\label{large scale section}
This section aims to establish the LDT for the Green's function at large scales via multi-scale analysis (MSA). Since our initial set of free sites $S_{\rm in}$ in \eqref{center of initial free sites} is more similar to the one in \cite{BK05} (of the form $Q\Z, Q\gg 1$) than to the one in \cite{DS20} (of the form $\Z\setminus Q\Z, Q\gg 1$), our iteration will proceed in a manner closer to that of \cite{BK05}.

Since $\lambda$ and $T$ are fixed, in the following we will omit the dependence of constants on $\lambda$ and $T$ for simplicity. 

Let $\varepsilon,\rho$ be numerical constants such that $0<12\rho\leq \varepsilon\le 10^{-10}\ll 1$, and let $M\in \mathbb Z_+$ be such that $\kappa =(1-\varepsilon)^M\le \varepsilon/10$. We set the constant
\begin{equation}\label{C_in}
C_{\rm in}=\frac{3}{2}(1-\varepsilon)^{-M}	
\end{equation}

to be used in Theorem \ref{LDT for the initial scale, without free sites} and Theorem \ref{LDT for the initial scale, free sites}. We prove the following results:

\begin{thm}\label{LDT for the large scales}
Assume that $T\in\mathscr{R}$ satisfies \textbf{\textup{(A3)}}. Let $\delta$ and $N_{\rm in}$ be as in Theorem \ref{LDT for the initial scale, without free sites}, and set $\gamma_0 = (\log N_{\rm in})^{-8000 d_T}$. Moreover, let $E\in [2^{-6000d_T}\delta,\delta]$ be an arbitrary \textbf{fixed} energy. For any scale $N\in {\rm Scale}_{\rm in}\cup (N_{\rm in}^{C_{\rm in}},\infty)$, there exists an event $ \Omega_N \subset \{0,1\}^{\Lambda}$, where $\Lambda = \Lambda_N$, such that:

\begin{enumerate}
    \item
    \begin{equation}\label{probability estimate for large scales}
        1-\P\left(\Omega_N \right) < N^{-1/3}.
    \end{equation}

    \item
    $\Omega_N$ is obtained as a disjoint union of ``cylinders'' of the form
    \begin{equation}\label{form of cylinder}
		        \mcC = \{(\omega_j)_{j\in \Lambda\setminus S}\} \times \{0,1\}^{S} \subset \{0,1\}^{\Lambda},
	\end{equation}
    where the index set $S$ (which we call the set of free sites) and the element $\baromega=(\omega_j)_{j\in \Lambda\setminus S}$ depend on $\mcC$. Moreover, $S$ is a union of disjoint intervals of length $2D_{\hop}(T)+1$, and the set $\wtS$ of the centers of those intervals satisfies the density assumption
    \begin{equation}\label{ample of free sites}
		   \# (\Lambda' \cap \wtS) \geq (N')^{1-\rho} \quad \text{for any } N'\text{-interval } \Lambda' \subset \Lambda, \ \forall N'\ge N^{1-\varepsilon/4}.
	\end{equation}

    \item
    For a cylinder $\mcC$ as above, the Green's function $G_\Lambda(E)=G_\Lambda(E;\bar{\omega}, t_j\,(j\in S))$, with arbitrary $t_j\in[0,1]$, satisfies
    \begin{align}
      \label{L2 estimate for large scales}  &\|G_\Lambda(E)\| < e^{N^{1-\varepsilon/4}}, \\
      \label{Off diagonal estimate for large scales}  &\|G_\Lambda(E)(x,x')\| < e^{-\gamma_N |x-x'|} \quad \text{for } x,x'\in\Lambda,\; |x-x'|\geq \frac{N}{200},
    \end{align}
    for some rate $\gamma_N\ge \gamma_0/2$. Moreover, the above estimate is stable under energy perturbations: for any $\wtE$ such that $|\wtE-E|\le e^{-N^{1-\varepsilon/2}}$, we still obtain 
	\begin{align}
      \label{L2 estimate for large scales, perturbation energy}  &\|G_\Lambda(\wtE)\| < e^{N^{1-\varepsilon/4}}, \\
      \label{Off diagonal estimate for large scales, perturbation energy}  &\|G_\Lambda(\wtE)(x,x')\| < e^{-\gamma_N |x-x'|} \quad \text{for } x,x'\in\Lambda,\; |x-x'|\geq \frac{N}{200}.
    \end{align}
\end{enumerate}
\end{thm}

\begin{rmk}\label{rmk on LDT for large scales}
(1) It might cause confusion that in \eqref{L2 estimate for large scales, perturbation energy} and \eqref{Off diagonal estimate for large scales, perturbation energy}, the perturbation range is $e^{-N^{1-\varepsilon/2}}$, which is larger than the order $e^{-N^{1-\varepsilon/4}}$ appearing in the upper bound \eqref{L2 estimate for large scales}. This seems counterintuitive, since the Neumann series argument
\begin{equation}\label{perturbation energy Neumann}
			G_{\La}(\wtE)= \sum_{s\geq 1}G_{\Lambda} (E) [(\wtE-E)G_{\La}(E)]^s 
\end{equation}
can only preserve the form of \eqref{L2 estimate for large scales, perturbation energy} from \eqref{L2 estimate for large scales} under perturbations of size at most $(1-)e^{-N^{1-\varepsilon/4}}$. This size is what \cite[Lemma 6.4]{DS20} admits. 
(In fact, if one wants to preserve \eqref{Off diagonal estimate for large scales, perturbation energy} from \eqref{Off diagonal estimate for large scales} via \eqref{perturbation energy Neumann}, the perturbation must be at most $e^{-N}$ in scale, due to the distance lower bound $N/200$.) 
The reason why the perturbation can be as large as $e^{-N^{1-\varepsilon/2}}$ is that our proof is inductive. Although $e^{-N^{1-\varepsilon/2}}$ is too large for the scale $N$ itself, it is sufficiently small relative to the previous (smaller) scales, and hence preserves the Green's function estimates at those scales, thereby maintaining the ``goodness'' of the intervals at previous scales. This in turn preserves the Green's function estimate at scale $N$.

This phenomenon can also be understood from two aspects in the following proofs. First, for the initial scales, the Green's function estimate is indeed established uniformly for all $E\in[0,\delta]$ in Theorem \ref{LDT for the initial scale, free sites}. Second, for the large scales, thanks to the hierarchical resonant structure, the Green's function estimate depends essentially on the resonances in some smaller intervals $Q_s$; a perturbation of size $e^{-N^{1-\varepsilon/2}}$ remains sufficiently small to preserve the Wegner estimate on each $Q_s$ (see \eqref{Qs inequality, perturbation wegner}).\\

(2) Although Theorem \ref{LDT for the initial scale, without free sites} and Theorem \ref{LDT for the initial scale, free sites} hold for all $E\in [0,\delta]$, in Theorem \ref{LDT for the large scales} we only consider energies $E\in [2^{-6000d_T}\delta,\delta]$. The additional lower bound away from zero is mainly imposed to match the lower bound conditions \eqref{lower bound on potential, R0} and \eqref{lower bound on potential, R-} in the QUC, as will be seen later in the proof (see \eqref{low bound quanti in R0} and \eqref{lower bound on potential, R-}).

	Indeed, to prove localization in the full region $[0,\delta]$, Theorem \ref{LDT for the large scales} is sufficient, since we can use the following trick: decompose 
	\[
	(0,\delta] = \bigcup_{j=0}^{\infty} [2^{-(j+1)\times 6000d_T}\delta,\, 2^{-j\times 6000d_T}\delta].
	\]
	Suppose Theorem \ref{LDT for the large scales} holds. For example, consider the interval $[2^{-2\times 6000d_T}\delta,2^{-6000d_T}\delta]$. We set 
	\[
	\delta' = 2^{-6000d_T}\delta,\qquad N'_{\rm in}=N_{\rm in}^2,
	\]
	then 
	\[
	\delta' = 2^{-6000d_T}(\log N_{\rm in})^{-6000d_T} = (\log N'_{\rm in})^{-6000d_T}.
	\]
	Therefore, Theorem \ref{LDT for the large scales} holds for energies in $[2^{-2\times 6000d_T}\delta,2^{-6000d_T}\delta]$ with $\delta$ replaced by $\delta'$ and $N_{\rm in}$ replaced by $N'_{\rm in}$. The same argument applies to the other intervals, and thus we essentially obtain the Green's function estimate in Theorem \ref{LDT for the large scales} for all energies $E\in (0,\delta]$ (with the caveat that the closer $E$ is to zero, the larger the initial scale $N_{\rm in}$ must be chosen).

	Concerning localization, Theorem \ref{LDT for the large scales} will imply Anderson localization in the spectral region $[2^{-6000d_T}\delta,\delta]$. Similarly, by rechoosing the parameters and initial scales as discussed above, we also obtain Anderson localization in each interval $[2^{-(j+1)\times 6000d_T}\delta,2^{-j\times 6000d_T}\delta]$, and hence establish localization on the whole interval $[0,\delta]$.\\

(3) Actually, the MSA argument in the proof of Theorem \ref{LDT for the large scales} is independent of dimension and can be generalized to arbitrary dimensions $\Z^d$. The only dependence on the dimension in our proof comes from the fact that the quantitative unique continuation is established in one dimension.

\end{rmk}

\subsection{Proof of LDT for large scales}
\begin{proof}[Proof of Theorem \ref{LDT for the large scales}] $\quad$\\

\noindent \textcolor{black}{\large \textbf{Initial step}} \\
For scales $N=N_0\in {\rm Scale}_{\rm in}$, Theorem \ref{LDT for the large scales} is a direct consequence of Theorem \ref{LDT for the initial scale, free sites}. For such scales we set $\Omega_{N_0} = (\Xi'_{N_0})^c$, and hence
\[
1-\P(\Omega_{N_0}) = \P(\Xi'_{N_0}) \le e^{-(\log N_0)^{10}} \ll N_0^{-1/3}.
\]
Moreover, since $\Omega_{N_0}$ is independent of the randomness on $S_{\rm in}\cap \Lambda_{N_0}$, we may write it as
\[
\Omega_{N_0} = \bigcup_{\substack{
   \bar{\omega} \in \operatorname{proj}_{\Lambda_{N_0}\setminus S_{\rm in}}(\Omega_{N_0})
}} 
 \{(\omega_j)_{j\in \Lambda_{N_0}\setminus S_{\rm in}}\} \times \{0,1\}^{S_{\rm in}\cap \Lambda_{N_0}}.
\]
For each cylinder, we take $S=S_{\rm in}\cap \La_{N_0}$, with $S_{\rm in}$ defined in \eqref{center of initial free sites}. (If some interval of $S_{\rm in}$ intersects the boundary of $\La_{N_0}$, we remove it from $S$; this does not affect the proof. Similar boundary issues also arise when constructing $S_{\rm remaining}$ and $\wtS_{\rm remaining}$ in \eqref{S remaining} and \eqref{wtS remaining}.) The set of centers of $S$ is $\wtS=\Lambda_{N_0}\cap \wtS_{\rm in}$, and therefore
\[
\#(\Lambda' \cap \wtS) = \#(\Lambda' \cap \wtS_{\rm in}) \ge \frac{2N'+1}{2\Dbig(T)} \gtrsim N'\geq (N')^{1-\rho}
\]
for any $N'$-interval $\Lambda' \subset \Lambda_{N_0}$ with $N'\ge N_0^{1-\varepsilon/4}$. Finally, taking
\[
\gamma_{N_0}\equiv \gamma_0 = (\log N_{\rm in})^{-8000d_T} \ll (\log N_0)^{-7000d_T},
\]
the estimates \eqref{L2 estimate for large scales} and \eqref{Off diagonal estimate for large scales} follow from \eqref{L2 norm estimate on Green's function, initial scale, free site type} and \eqref{off diagonal decay estimate on Green's function, initial scale, free site type}, respectively.
Moreover, since in Theorem \ref{LDT for the initial scale, free sites} the estimates \eqref{L2 norm estimate on Green's function, initial scale, free site type} and \eqref{off diagonal decay estimate on Green's function, initial scale, free site type} hold uniformly for all $E\in [0,\delta]$ (and indeed can be extended to the region $[0,2\delta]$, as is readily seen from the proof of Theorem \ref{LDT for the initial scale, without free sites}), the estimates \eqref{L2 estimate for large scales, perturbation energy} and \eqref{Off diagonal estimate for large scales, perturbation energy} also hold, since $\wtE \leq E+|E-\wtE|\leq \delta+e^{-N^{1-\varepsilon/2}}\leq 2\delta$.\\

\noindent \textcolor{black}{\large \textbf{Inductive step}} \\
For scales $N>N_{\rm in}^{C_{\rm in}}$, suppose that Theorem \ref{LDT for the large scales} has been established for all previous scales in ${\rm Scale}_{\rm in} \cup (N_{\rm in}^{C_{\rm in}}, N-1]$. 

Since our choice of $C_{\rm in}$ in \eqref{C_in} ensures that 
\begin{equation}\label{L_M larger than N_in}
N^{(1-\varepsilon)^{M}} > N_{\rm in}^{C_{\rm in} (1-\varepsilon)^M} \ge N_{\rm in}^{\frac{3}{2}} \gg N_{\rm in},
\end{equation}
we may define $L_0=N$ and $L_k$ for $1\le k\le M$ by
\begin{equation}\label{various iteration scales}
       L_k=\begin{cases}
		       \lfloor L_0^{(1-\varepsilon)^k} \rfloor, & \text{if } L_0^{(1-\varepsilon)^k} > N_{\rm in}^{C_{\rm in}},\\
			   \text{the scale in ${\rm Scale}_{\rm in}$ nearest to } L_0^{(1-\varepsilon)^k}, & \text{otherwise.}
	   \end{cases}
\end{equation}
Therefore, $L_k\in {\rm Scale}_{\rm in} \cup (N_{\rm in}^{C_{\rm in}}, N-1]$. From the definition \eqref{initial scale set Scale_in}, it is easy to see that any two consecutive elements of ${\rm Scale}_{\rm in}$ differ by $(2L'+1)(2L+1)$ (which is $\sim (\log N_{\rm in})^{750}$ by our choice of scales in Theorem \ref{LDT for the initial scale, without free sites}). Therefore, regardless of which case occurs in \eqref{various iteration scales}, we have 
\[
L_k - L_0^{(1-\varepsilon)^k} = \mcO\left((\log N_{\rm in})^{750}\right),
\]
and hence 
\begin{equation}\label{L_k approximate}
	L_k = \left(1 - \mcO\left(N_{\rm in}^{-3/2}(\log N_{\rm in})^{750}\right)\right) L_0^{(1-\varepsilon)^k}.
\end{equation}
Since $N_{\rm in}\gg 1$, \eqref{L_k approximate} shows that $L_k$ is equal to $L_0^{(1-\varepsilon)^k}$ up to a factor very close to $1$. It will be clear that our subsequent argument is robust with respect to such a small perturbation of the factor; thus, for simplicity of the proof, we will simply treat 
\[
L_k = L_0^{(1-\varepsilon)^k}, \quad 1\le k\le M.
\]

First, since $L_1 = L_0^{1-\varepsilon} \ll L_0 = N$, we can cover $\La = \Lambda_N$ by a family $\mcF^{(1)}$ of $L_1$-intervals (contained in $\Lambda$) with the following structure:
\begin{itemize}
	\item $\Lambda = \bigcup_{\La_{L_1}(r) \in \mcF^{(1)}} \La_{L_1}(r)$, and $\#\mcF^{(1)} \lesssim L_0/L_1 = L_0^{\varepsilon}$;
	\item For any distinct $\La_{L_1}(r), \La_{L_1}(r') \in \mcF^{(1)}$, we have $\Lambda_{L_1/10}(r')\cap \Lambda_{L_1}(r)=\varnothing$; 
	\item For every $x\in \Lambda$, there exists $\La_{L_1}(r) \in \mcF^{(1)}$ such that $\Lambda_{L_1/10}(x)\cap \La \subset \Lambda_{L_1}(r)$. 
\end{itemize}
Then, for each $\La_{L_1}(r)\in \mcF^{(1)}$, since Theorem \ref{LDT for the large scales} holds at scale $L_1$ and is shift-invariant with respect to the randomness, there exists an event $\Omega_{L_1}(r)\subset \{0,1\}^{\La_{N_1}(r)}$ with probability greater than $1-L_1^{-1/3}$, which can be decomposed into disjoint cylinders 
\begin{equation}\label{cylinder decomp of L_1 r_1}
	\Omega_{L_1}(r)= \bigcup_{\alpha \in \mathfrak{A}_r } \mcC_{r,\alpha},
\end{equation}
where each cylinder $\mcC_{r,\alpha}$ is of the form \eqref{form of cylinder}. For the randomness in each $\mcC_{r,\alpha}$, the estimates \eqref{L2 estimate for large scales} and \eqref{Off diagonal estimate for large scales} hold. If the randomness lies in $\Omega_{L_1}(r)$, we call the interval $\La_{L_1}(r)$ ``good''; otherwise, we call it ``bad''. The same terminology and notations apply to other scales.

Let $K_1$ be a large positive integer to be chosen later, and consider the probability
\begin{align}\label{event at least K_1 bad L_1 cube}
\P\left( \text{there are at least } K_1 \text{ bad } L_1\text{-intervals in } \mcF^{(1)} \right).
\end{align}
Since our construction ensures that any two distinct $L$-intervals in $\mcF^{(1)}$ are either adjacent or disjoint (see \cite{LSZ25}, Section 3.3, Figure 1 for an illustration), the probability in \eqref{event at least K_1 bad L_1 cube} can be bounded by
\begin{align}\label{prob on event at least K_1 bad L_1 cube}
	\eqref{event at least K_1 bad L_1 cube}
	& \leq \left(\#\mcF^{(1)} \right)^ {K_1} \cdot L_1^{-\frac{1}{3}\cdot \frac{1}{2}K_1}
	\leq L_1^{\left( \varepsilon+ \mcO((\log L_1)^{-1}) -\frac{1}{6}\right)K_1}
	\leq L_1^{-K_1/7}.
\end{align}
That is, we can ensure that there are at most $K_1$ bad $L_1$-intervals in $\mcF^{(1)}$, except on an event of probability less than $L_1^{-K_1/7}$.

Next, similarly to the construction of $\mcF^{(1)}$, we continue to pave each $L_1$-interval with smaller $L_2$-intervals. Specifically, for each $\La_{L_1}(r)\in \mcF^{(1)}$, we cover it by a family $\mcF^{(2)}_r$ of $L_2$-intervals (contained in $\Lambda_{L_1}(r)$) with the following structure:
\begin{itemize}
	\item $\Lambda_{L_1}(r) = \bigcup_{\La_{L_2}(y) \in \mcF^{(2)}_r} \La_{L_2}(y)$, and $\#\mcF^{(2)}_r \lesssim L_1/L_2 = L_0^{\varepsilon}$;
	\item For any distinct $\La_{L_2}(y), \La_{L_2}(y') \in \mcF^{(2)}_r$, we have $\Lambda_{L_2/10}(y')\cap \Lambda_{L_2}(y)=\varnothing$; 
	\item For every $x\in \Lambda_{L_1}(r)$, there exists $\La_{L_2}(y) \in \mcF^{(2)}_r$ such that $\Lambda_{L_2/10}(x)\cap \Lambda_{L_1}(r) \subset \Lambda_{L_2}(y)$. 
\end{itemize}
Indeed, the union $\mcF^{(2)}= \bigcup_{\Lambda_{L_1}(r) \in\mcF^{(1)}} \mcF^{(2)}_r$ gives a covering of the whole $\La$ by $L_2$-intervals satisfying the above properties. We repeat this construction for each $\La_{L_2}(r)\in \mcF^{(2)}$, now using $L_3$-intervals, and so on, until we obtain a covering by $L_M$-intervals. Thus we obtain 
\[
\mcF^{(3)}, \mcF^{(4)} ,\cdots,\mcF^{(M)}.
\]
These coverings give the \textbf{hierarchical structure} in our multi-scale proof.\\

Now for every $\La_{L_1}(r_1)$ in $\mcF^{(1)}$, we will call an $L_M$-interval $\La_{L_M}(r_M)\in \mcF^{(M)}$ a \textbf{``hereditary bad subinterval''} of $\La_{L_1}(r_1)$ (borrowing the terminology from \cite{DS20}) if there exists a nested sequence
\[
 \La_{L_1}(r_1) \supset \La_{L_2}(r_2) \supset \cdots \supset \La_{L_M}(r_M)
\]
such that for each $2\le k\le M$, the interval $\La_{L_k}(r_k)$ is bad. Let \textcolor{red}{$K_2=(K')^{M-1}$}, where $K'$ is a large positive integer to be chosen later. Denote the event
\begin{equation}\label{reverse event of hierarchical resonant structure}
A_{r_1}= \left\{
  \begin{aligned}
    &\exists \ 2\le s\le M \text{ and } \La_{L_1}(r_1)\supset \La_{L_{s-1}}(x_{s-1})\in \mcF^{(s-1)} \\
    &\text{such that it has more than } K' \text{ bad } L_s\text{-subintervals in } \mcF^{(s)}_{x_{s-1}}
  \end{aligned}
\right\}.
\end{equation}
By the pigeonhole principle, we have 
\begin{align}\label{less than K2 hereditary bad block}
\notag	\P & \left(\La_{L_1}(r_1) \text{ has more than } K_2 \text{ hereditary bad subintervals}\right) \leq \P(A_{r_1})\\
\notag	\leq & \sum_{2\leq s\leq M} \#\mcF^{(s-1)}\cdot (\# \mcF^{(s)}_{x_{s-1}})^{K'} \cdot L_{s}^{-\frac{1}{3}\cdot \frac{1}{2}K'} \\
\notag	\leq & \sum_{2\leq s\leq M}  \left(C\frac{L_1}{L_{s-1}}\right)\cdot \left(C\frac{L_{s-1}}{L_s}\right)^{K'} \cdot L_{s}^{-\frac{1}{3}\cdot \frac{1}{2}K'} \qquad (\text{$C$ is some numerical constant}) \\
\notag	\leq & L_1 \sum_{2\leq s\leq M}L_{s-1}^{-1+\left(\varepsilon +\mcO((\log L_{s-1})^{-1})-\frac{1}{6}\right)K'}  \leq  L_1 \sum_{2\leq s\leq M}L_{s-1}^{-1-\frac{1}{7}K'} \\
	\leq & L_1 L_{M}^{-K'/8} = L_1^{1-(1-\varepsilon)^{M}K'/8} \ll L_1^{-10}
 \end{align}
if we choose the numerical constant \textcolor{red}{$K'\ge 100(1-\varepsilon)^{-M}$}. Clearly, the event $A_{r_1}$ depends only on the randomness in $\La_{N_1}(r_1)$. Let 
\[
E_{r_1}=A_{r_1}^c \setminus \Omega_{L_1}(r_1)
\]
be the event that the interval $\La_{L_1}(r_1)$ is bad and that the estimate in \eqref{reverse event of hierarchical resonant structure} fails. Then
\begin{equation}\label{prob of E_r1 cup Omega_L1r1}
	\P(E_{r_1}\cup \Omega_{L_1}(r_1) )=\P(A_{r_1}^c\cup \Omega_{L_1}(r_1))\geq 1-\P(A_{r_1})\geq 1-L_1^{-10}.
\end{equation}
Set $E_{r_1}=\mcC_{r_1, 0}$ (corresponding to $\alpha=0$) and decompose $\Omega_{L_1}(r_1)$ into disjoint cylinders 
\[
\Omega_{L_1}(r_1)=\bigcup_{\alpha \geq 1} \mcC_{r,\alpha},
\]
where we regard the index set $\mathfrak{A}_{r_1}$ in \eqref{cylinder decomp of L_1 r_1} as a set of positive integers. By \eqref{prob of E_r1 cup Omega_L1r1}, we have 
\begin{equation}
    \P\left(\bigcap_{\La_{L_1}(r_1)\in \mcF^{(1)}} (E_{r_1}\cup \Omega_{L_1}(r_1))\right)\geq 1-\#\mcF^{(1)}\cdot L_1^{-10}\geq 1-L_0^{-9}.
\end{equation}
Recall that by \eqref{prob on event at least K_1 bad L_1 cube}, the event 
\[
\mathscr{P}=\left\{ \text{there are at most } K_1 \text{ bad } L_1\text{-intervals in } \mcF^{(1)} \right\}
\]
has probability larger than $1-L_1^{-K_1/7}$, and can be represented as the following disjoint union:
\[
\mathscr{P}=\bigcup_{\mcB\subset \mcF^{(1)} \atop \#\mcB \leq K_1} \mathscr{P}_{\mcB}
:=\bigcup_{\mcB\subset \mcF^{(1)} \atop \#\mcB \leq K_1}
\left( \bigcap_{\La_{L_1}(r_1)\notin\mcB} \Omega_{L_1}(r_1)
\cap \bigcap_{\La_{L_1}(r_1)\in\mcB} \Omega_{L_1}(r_1)^c \right).
\]
Therefore,
\begin{align}\label{event combine K1 bad cube and K2 hereditary bad cube}
\notag
\P\left(\mathscr{P} \cap \left( \bigcap_{\La_{L_1}(r_1)\in \mcF^{(1)}} (E_{r_1}\cup \Omega_{L_1}(r_1)) \right)\right)
&= \P\left(
\bigcup_{\mcB\subset \mcF^{(1)} \atop \#\mcB \leq K_1}
\mathscr{P}_{\mcB} \cap
\left( \bigcap_{\La_{L_1}(r_1)\in \mcF^{(1)}} (E_{r_1}\cup \Omega_{L_1}(r_1)) \right)
\right) \\
&\ge 1-L_1^{-K_1/7}-L_0^{-9}
\ge 1-L_0^{-8},
\end{align}
if we choose \textcolor{red}{$K_1=80$}.

Since $E_{r_1}\cup \Omega_{L_1}(r_1)=\cup_{\alpha\geq 0}\mcC_{r_1,\alpha}$ (disjoint union), we can further expand 
\[
\bigcap_{\La_{L_1}(r_1)\in \mcF^{(1)}} (E_{r_1}\cup \Omega_{L_1}(r_1)) =
\bigcup_{\bar{\alpha} } \left(\bigcap_{r_1}\mcC_{r_1,\alpha_{r_1}}\right)
\]
as a disjoint union, where $\bar{\alpha}=(\alpha_{r_1})_{\La_{L_1}(r_1)\in \mcF^{(1)}}$ is a multi-index with $\alpha_{r_1}\ge 0$ indexing the cylinders from each $E_{r_1}\cup \Omega_{L_1}(r_1)$. Further, the event \eqref{event combine K1 bad cube and K2 hereditary bad cube} can be written as
\[
\bigcup_{\mcB\subset \mcF^{(1)} \atop \#\mcB \leq K_1}
\bigcup_{\bar{\alpha} }
\mathscr{P}_{\mcB} \cap
\left(\bigcap_{r_1}\mcC_{r_1,\alpha_{r_1}}\right),
\]
which is clearly a disjoint union over the indices $\mcB$ and $\bar{\alpha}$. A simple observation is
\begin{equation*}
	\mathscr{P}_{\mcB} \cap \left(\bigcap_{r_1}\mcC_{r_1,\alpha_{r_1}}\right) =
	\begin{cases}
		\bigcap_{r_1}\mcC_{r_1,\alpha_{r_1}}, & {\rm if } \{\La_{L_1}(r_1):\alpha_{r_1}=0\}=\mcB;\\
		\varnothing, & {\rm else.}
	\end{cases}
\end{equation*}
Therefore
\[
\bigcup_{\mcB\subset \mcF^{(1)} \atop \#\mcB \leq K_1}
\bigcup_{\bar{\alpha} }
\mathscr{P}_{\mcB} \cap
\left(\bigcap_{r_1}\mcC_{r_1,\alpha_{r_1}}\right)
=
\bigcup _{\bar{\alpha}: \# \{r_1: \alpha_{r_1}=0 \}\leq K_1}
\bigcap_{r_1}\mcC_{r_1,\alpha_{r_1}}
\]
and \eqref{event at least K_1 bad L_1 cube} becomes
\begin{equation}\label{final disjoint union of the cylinders}
	\P\left( \bigcup _{\bar{\alpha}: \# \{r_1: \alpha_{r_1}=0 \}\leq K_1} \bigcap_{r_1}\mcC_{r_1,\alpha_{r_1}}   \right) \geq 1-L_0^{-8},
\end{equation}
which is again a disjoint union over the admissible $\bar{\alpha}$.\\

Now fix an admissible multi-index $\bar{\alpha}$ in \eqref{final disjoint union of the cylinders}. For each $r_1$ with $\alpha_{r_1}\ge 1$, let $S_{r_1,\alpha_{r_1}}$ and $\wtS_{r_1,\alpha_{r_1}}$ denote the set of free sites and its center, respectively, of the cylinder $\mcC_{r_1,\alpha_{r_1}}$. Define
\begin{equation}\label{mcT baralpha}
	\mcT_{\bar{\alpha}}:= \bigcap_{r_1}\mcC_{r_1,\alpha_{r_1}}.
\end{equation}

On the event $\mcT_{\bar{\alpha}}$, set $\mathfrak{B}_{\bar{\alpha}}=\{r_1:\alpha_{r_1}=0\}$ and define
\begin{equation}\label{wtS baralpha}
	\wtS_{\bar{\alpha}}= \bigcup_{\substack{r_1: \operatorname{dist}(r_1,\mathfrak{B}_{\bar{\alpha}})\ge 20 L_1}} \bigl( \La_{L_1/10}(r_1)\cap \wtS_{r_1,\alpha_{r_1}} \bigr).
\end{equation}
Furthermore, set
\begin{equation}\label{S baralpha}
	S_{\bar{\alpha}}=\bigcup_{m\in \wtS_{\bar{\alpha}}}[m-\Dhop(T),\,m+\Dhop(T)].
\end{equation}

\begin{claim}\label{generate set of free sites}
$S_{\bar{\alpha}}$ is a set of free sites in $\mcT_{\bar{\alpha}}$.
\end{claim}
\begin{proof}[Proof of Claim \ref{generate set of free sites}]
    Since $\mcT_{\bar{\alpha}}$ is an intersection of the cylinders $\mcC_{r_1,\alpha_{r_1}}$, we consider the following cases. When $r_1\in \mathfrak{B}_{\bar{\alpha}}$, we have $\alpha_{r_1}=0$; hence $\mcC_{r_1,0}=E_{r_1}$ is not a cylinder and therefore has no set of free sites. When $r_1\notin \mathfrak{B}_{\bar{\alpha}}$, the set $\mcC_{r_1,\alpha_{r_1}}$ is a cylinder and possesses a set of free sites $S_{r_1,\alpha_{r_1}}\subset \La_{L_1}(r_1)$.

    However, for $r_1\neq r'_1$, the intersection $\mcC_{r_1,\alpha_{r_1}}\cap \mcC_{r'_1,\alpha_{r'_1}}$ causes the free sites in the overlap $\La_{L_1}(r_1)\cap \La_{L_1}(r'_1)$ to cease to be free. Consequently, in $\mcT_{\bar{\alpha}}$, the remaining free sites lie in
    \begin{equation}\label{intersection cause the free sites change}
        \bigcup_{\substack{r_1: r_1\notin \mathfrak{B}_{\bar{\alpha}}}} S_{r_1,\alpha_{r_1}}
        \setminus
        \bigcup_{r_1\neq r'_1} \bigl( \La_{L_1}(r_1)\cap \La_{L_1}(r'_1) \bigr).
    \end{equation}

    By the construction of the covering $\mcF^{(1)}$, two distinct intervals $\La_{L_1}(r_1)$ and $\La_{L_1}(r'_1)$ can only intersect near their boundaries, while their shrunk versions $\La_{L_1/10}(r_1)$ and $\La_{L_1/10}(r'_1)$ are mutually disjoint. Hence $S_{\bar{\alpha}}$ is contained in \eqref{intersection cause the free sites change} and thus forms a valid set of free sites.
\end{proof}

Since the proof of Claim \ref{generate set of free sites} already shows that $\wtS_{\bar{\alpha}}$ is a disjoint union, for any $m$-interval $\La'\subset \La=\La_N$ with $m\geq N^{1-\varepsilon/4}=L_1^{\frac{1-\varepsilon/4}{1-\varepsilon}}$, we have 
\begin{align}\label{dense of wtS baralpha}
  \notag	\# (\La'\cap \wtS_{\bar{\alpha}}) &\geq \left(\# \{r_1: \La_{L_1/10}(r_1) \subset \La' \}- \#\{r_1: \operatorname{dist}(r_1,\mathfrak{B}_{\bar{\alpha}})< 20 L_1 \} \right) \cdot \# (\La_{L_1/10}(r_1)\cap \wtS_{r_1,\alpha_{r_1}}  ) \\
	                               &\gtrsim \left(\frac{m}{L_1}-100K_1\right) \cdot (L_1/10)^{1-\rho}			   \gtrsim m^{1-\frac{1-\varepsilon}{1-\varepsilon/4}\rho}\gg m^{1-\rho}, 
\end{align} 
where we have used $L_1/10\ge L_1^{1-\varepsilon/4}$ and the inductive hypothesis \eqref{ample of free sites} for each $\La_{L_1}(r_1)$ with $\operatorname{dist}(r_1,\mathfrak{B}_{\bar{\alpha}})\ge 20L_1$.
\begin{rmk}\label{rmk on dense of S baralpha}
	Indeed, the estimate \eqref{dense of wtS baralpha} can be refined to hold for all $m\ge N^{1-\varepsilon/a}$ with any numerical constant $a>1$, and we then deduce
	\[
	\# (\La'\cap \wtS_{\bar{\alpha}}) \gtrsim  m^{1-\frac{1-\varepsilon}{1-\varepsilon/a}\rho}\gg m^{1-\rho}
	\]
	by taking the initial scale $N_{\rm in}$ sufficiently large (depending on $a$).
\end{rmk}

The above discussion allows us to further decompose $\mcT_{\baralpha}$ into cylinders as
\[
\mcT_{\baralpha}= \bigcup_{\substack{
   \bar{\omega} \in \operatorname{proj}_{\Lambda \setminus S_{\baralpha} }(\mcT_{\baralpha})}} 
    \{(\omega_j)_{j\in \Lambda\setminus S_{\baralpha}}\} \times \{0,1\}^{S_{\baralpha}},
\]
and consequently the event in \eqref{final disjoint union of the cylinders} can be written as
\begin{align}
\notag	\bigcup _{\bar{\alpha}: \# \{r_1: \alpha_{r_1}=0 \}\leq K_1} \mcT_{\baralpha} &= \bigcup _{\bar{\alpha}: \# \{r_1: \alpha_{r_1}=0 \}\leq K_1} \bigcup_{\substack{
   \bar{\omega} \in \operatorname{proj}_{\Lambda \setminus S_{\baralpha} }(\mcT_{\baralpha})}}     \{(\omega_j)_{j\in \Lambda\setminus S_{\baralpha}}\} \times \{0,1\}^{S_{\baralpha}} \\
              &:= \bigcup_{\beta} \mcC_{\beta}.
\end{align}
Here, for short, we denote the admissible multi-index $(\baralpha,\baromega)$ by $\beta$, and write the corresponding cylinder as $\mcC_{\beta}:= \{(\omega_j)_{j\in \Lambda\setminus S_{\baralpha}}\} \times \{0,1\}^{S_{\baralpha}}$. The union over $\beta$ remains disjoint.\\

Now we condition on each fixed cylinder $\mcC_{\beta}$, where $\beta=(\baralpha,\baromega)$ is an admissible multi-index. Recall that the index $\baralpha$ encodes the following information:
\begin{itemize}
	\item For each $r_1\notin \mathfrak{B}_{\baralpha}$, the interval $\La_{L_1}(r_1)$ is good;
	\item For each $r_1\in\mathfrak{B}_{\baralpha}$, the interval $\La_{L_1}(r_1)$ is bad but has at most $K_2$ hereditary bad subintervals.
\end{itemize}
Here “good” and “bad” are understood with respect to the Green's function estimates for potentials of the form given by the fixed configuration $\baromega$ outside $S_{\baralpha}$ and arbitrary $t_j\in[0,1]$ for $j\in S_{\baralpha}$. With the above structure in hand, the only missing ingredient is a Wegner estimate on certain intervals. To ensure that the remaining free sites are still sufficiently abundant after removing a probability set coming from the Wegner estimate, we introduce the following \textbf{intermediate scale}
\[
     N^{1-\varepsilon/3}= L_0^{1-\varepsilon/3}\gg L_1.
\]
Clearly, by the argument in \cite[(5.48)--(5.52)]{BK05} or \cite[Lemma 8.1]{DS20}, we can obtain a scale 
\[
N^{1-\varepsilon/3} \leq \whL \lesssim_{K_1} N^{1-\varepsilon/3}
\]
such that there is a collection of $\whL$-intervals $Q_1,Q_2,\dots,Q_K$ with $K\le K_1$ in $\Lambda$ satisfying:
\begin{itemize}
	\item For every bad interval $\La_{L_1}(r_1)$ with $r_1\in \mathfrak{B}_{\baralpha}$, there exists $1\le s\le K$ such that $\La_{L_1}(r_1)\subset Q_s$ and 
	        \begin{equation}\label{L1 bad cube far from boundary Qs}
				    \dist(\La_{L_1}(r_1), \Lambda\setminus Q_s)\ge \frac{1}{8}\whL;
			\end{equation}
	    
	\item $\dist(Q_s,Q_{s'})\ge 10\whL$ for any $s\neq s'$.
\end{itemize}
We will prove the following claim in Subsection \ref{subsection: Wegner}.

\begin{claim}\label{L2 bound on Qs}
	For each $Q_s$, $1\le s\le K$, we can remove an event $\mcW_s$ of probability less than $\whL^{-\frac{1}{2}+\varepsilon}$, depending only on the randomness in $S_{\baralpha}\cap Q_s$, such that
	\begin{equation}\label{Qs inequality}
		\|G_{Q_s}(E)\| \le e^{\whL^{1-\varepsilon/3}}.
	\end{equation}
\end{claim}

Once Claim \ref{L2 bound on Qs} is established, the standard coupling lemma in MSA, Lemma \ref{coupling lemma}, will yield \eqref{L2 estimate for large scales} and \eqref{Off diagonal estimate for large scales} for $\La$ on the event 
\begin{equation}\label{mcC beta after wegner}
	\mcC_{\beta}\setminus \bigcup_{1\le s\le K}\mcW_s
\end{equation}
with probability (conditioned on $\mcC_{\beta}$) greater than $1-K_1 \whL^{-(\frac{1}{2}-)}$. 
\begin{rmk}
In the setting of Lemma \ref{coupling lemma} here, we take $\mathfrak{a}= (1-\varepsilon)^{-1}, \ \sigma=\varepsilon/4$, and the scales
\[\ell_0=L_0=N, \ \ell_1=\whL,\ \ell_0=L_1.\] 
The class of resonant intervals is 
\[\mathfrak{R}=\{\La_{L_1}(r_1):r_1\in \mathfrak{B}_{\baralpha}\},\]
and 
\[\mathfrak{R}'=\{Q_1,Q_2,\cdots,Q_K\}.\]
The covering of good blocks is 
\begin{equation}\label{L1 mfF}
	\mathfrak{F}=\{\La_{L_1}(r_1):r_1\notin \mathfrak{B}_{\baralpha}\}.
\end{equation}

Regarding the rate of the off-diagonal decay of the Green's function, Lemma \ref{coupling lemma} yields that 
\begin{align*}
	\gamma_{N} &\geq \gamma_{L_1} -\mcO_{+}\left( L_1^{-c}\right)= \gamma_{N^{1-\varepsilon}} -\mcO_{+}\left( N^{-c(1-\varepsilon)}\right) \\
	       & \overset{iterate}{\geq} \gamma_0 -\mcO_{+}\left( \sum_{\substack{s\geq 1 \\ \text{until the scale } \sim N_{\rm in}}} N^{-c(1-\varepsilon)^s}\right)= \gamma_0-\mcO_{+}\left( N_{\rm in}^{-c}\right).
\end{align*}
It is therefore clear that $\gamma_{N}\ge \gamma_0 /2$, since we have chosen $\gamma_0=(\log N_{\rm in})^{-8000d_T}\gg N_{\rm in}^{-c}$.
\end{rmk}
Moreover, if one perturbs the energy by $|\wtE-E|\le e^{-N^{1-\varepsilon/2}}$, then by \eqref{Qs inequality} we have 
\[
\operatorname{dist}(\spec(H_{Q_s}),\wtE)\ge \operatorname{dist}(\spec(H_{Q_s}),E)-|E-\wtE|\ge e^{-\whL^{1-\varepsilon/3}} -e^{-N^{1-\varepsilon/2}}\ge  \frac{1}{2} e^{-\whL^{1-\varepsilon/3}}.
\] 
Therefore,
\begin{equation}\label{Qs inequality, perturbation wegner}
	\|G_{Q_s}(\wtE)\| \le 2 e^{\whL^{1-\varepsilon/3}},
\end{equation}
which still satisfies the condition in \eqref{coupling lemma}. The $L_1$-intervals in \eqref{L1 mfF} remain good, since 
\[
|\wtE-E|\le e^{-N^{1-\varepsilon/2}} \ll e^{-L_1^{1-\varepsilon/2}},
\]
and we may apply the inductive hypothesis in the form of \eqref{L2 estimate for large scales, perturbation energy} and \eqref{Off diagonal estimate for large scales, perturbation energy} at scale $L_1$. Therefore, applying Lemma \ref{coupling lemma} again proves \eqref{L2 estimate for large scales, perturbation energy} and \eqref{Off diagonal estimate for large scales, perturbation energy} at scale $N$.\\

After constructing the event \eqref{mcC beta after wegner}, the remaining set of free sites is 
\begin{equation}\label{S remaining}
	S_{\rm remaining}=S_{\baralpha}\setminus \bigcup_{1\le s\le K} Q_s,
\end{equation}
with center set 
\begin{equation}\label{wtS remaining}
\wtS_{\rm remaining}=\wtS_{\baralpha}\setminus \bigcup_{1\le s\le K} Q_s.	
\end{equation}

Thus, the event \eqref{mcC beta after wegner} can be further decomposed into disjoint cylinders with set of free sites $S_{\rm remaining}$.

Finally, we define 
\[
\Omega_N=\bigcup_{\beta} \left(\mcC_{\beta}\setminus \bigcup_{1\le s\le K}\mcW_s\right),
\]
which can be decomposed into a disjoint union of cylinders of the form \eqref{form of cylinder} with free-site set $S_{\rm remaining}$. Moreover, by \eqref{dense of wtS baralpha}, we have 
\begin{align}\label{dense of wtS remaining}
  \notag	\# (\La'\cap \wtS_{\rm remaining}) &\ge \# (\La'\cap \wtS_{\bar{\alpha}}) - K_1 \# (Q_s\cap \wtS_{\baralpha})\\
	                               &\gtrsim m^{1-\frac{1-\varepsilon}{1-\varepsilon/4}\rho}-10 K_1 \whL \gtrsim m^{1-\frac{1-\varepsilon}{1-\varepsilon/4}\rho} \gg m^{1-\rho}, 
\end{align} 
for any $m$-interval $\La'\subset \La=\La_N$ with $m\ge N^{1-\varepsilon/4}$, provided that 
\[
m^{1-\frac{1-\varepsilon}{1-\varepsilon/4}\rho} \ge N^{1-\varepsilon/4 -(1-\varepsilon)\rho}\gg \whL =N^{1-\varepsilon/3},
\]
which is guaranteed by the condition \textcolor{red}{$\rho \le \frac{\varepsilon}{12}$}. This proves \eqref{ample of free sites} for our $\wtS_{\rm remaining}$. The estimates \eqref{L2 estimate for large scales}, \eqref{Off diagonal estimate for large scales}, \eqref{L2 estimate for large scales, perturbation energy} and \eqref{Off diagonal estimate for large scales, perturbation energy} hold on each cylinder of $\Omega_N$, and we have the probability estimate 
\begin{align*}
	\P(\Omega_N) &=\P \left( \bigcup_{\beta} \left (\mcC_{\beta}\setminus \cup_{1\leq s\leq K}\mcW_s \right)  \right)\\
	             &\ge  \P  \left(\bigcup_{\beta} \mcC_{\beta} \right)\cdot (1-K_1 \whL^{-\frac{1}{2}+\varepsilon }) \\
				 &\overset{\eqref{final disjoint union of the cylinders}}{\ge }1-L_0^{-8}-K_1 \whL^{-\frac{1}{2}+\varepsilon } \ge 1-L_0^{-1/3}.
\end{align*}
This completes the proof.
\end{proof}

\subsection{The Wegner estimate}\label{subsection: Wegner}
We next prove Claim \ref{L2 bound on Qs} via the following Wegner estimate for our model:
\begin{proof}[Proof of Claim \ref{L2 bound on Qs}]
Define the following refined scales, with $0<\eta\ll 1$:
\[
D_0=\whL,\ D_1=D_0^{1-\eta},\ D_2=D_1^{1-\eta},\ D_3=D_2^{1-\eta}=L_1,\ D_4=D_3^{1-\eta},\ D_5=D_4^{1-\eta}.
\]
Since $\whL\sim_{K_1} L_0^{1-\varepsilon/3}$ and $L_1=L_0^{1-\varepsilon}$, we have
\[
(1-\eta)^3 \approx \frac{1-\varepsilon}{1-\varepsilon/3}
\quad\Longleftrightarrow\quad
\eta = \left(\frac{2}{9}+\mcO(\varepsilon)\right)\varepsilon \approx \frac{2}{9}\varepsilon.
\]
For each $Q=Q_s$, let $S=S_{\baralpha}\cap Q$ be the set of free sites, with center set $\wtS=\wtS_{\baralpha}\cap Q$. Since the randomness on $Q\setminus S$ has already been fixed by the conditioning $\baromega|_{Q\setminus S}$ (recall that we previously conditioned the probability in cylinder $\mcC_{\beta}$), we may view
\[
H_Q(t)=H_Q(\baromega|_{Q\setminus S},\, t_j\,(j\in S)),\quad t_j\in[0,1],
\]
as a matrix-valued function of $t=(t_j)_{j\in S}\in [0,1]^S$. Denote by
\[
\lambda_1(t) \ge \lambda_2(t)\ge \cdots \ge \lambda_{\# Q}(t)
\]
the eigenvalues (counted with multiplicity) of $H_Q(t)$ in nonincreasing order. By the Kato--Rellich theorem, each $\lambda_i$, $1\le i\le \# Q$, is (piecewise) analytic in each coordinate of $t$. Moreover, for each $i$, let $\psi_i(t)$ be the normalized eigenvector associated with $\lambda_i(t)$.

By our construction, each bad $L_1$-interval $\La_{L_1}(r_1)\subset Q$ has at most $K_2$ hereditary bad subintervals. We denote by 
$Q'_1,Q'_2,\cdots, Q'_{K'}$
all the bad $L_1$-intervals in $Q$ (with $K'\leq K_1$). Further, let $\mathfrak{G}$ be a class of $L_M$-intervals, with at most $K' K_2$ elements, containing all hereditary bad subintervals that are contained in some $Q'_{k'}$. Then for every point $x$ outside 
\[
\mathcal{G}:= \bigcup _{\La_{L_M}(r_M)\in \mathfrak{G}} \La_{L_M}(r_M),
\]
there exists a scale $L_M \le m=L_k \le L_1$ and a good $m$-interval $B_x\in \mcF^{(k)}$ such that $\dist(x,Q\setminus B_x)\ge m/10$. In particular, for $x$ outside $\cup_{1\leq k'\leq K'}Q'_{k'}$, we may take $m=L_1$.

\begin{claim}\label{decay of eigenvector}
	If $t$ is such that $|\lambda_i(t)-E|\le e^{-D_5}$, then $\| \psi_i(t)\|_{\ell^2(\mcG)}\ge \frac{1}{2}$ and 
\begin{equation}\label{decay outside mcG}
	  |\psi_i(t)(x)|\le \exp\left\{-\frac{1}{4}\gamma_0 \bigl(\operatorname{dist}(x,\mcG)+\frac{L_M}{10}\bigr) \right \}, \qquad \forall x\in Q\setminus \mcG.
\end{equation}
\end{claim}
\begin{proof}[Proof of Claim \ref{decay of eigenvector}]
Since 
\[
e^{-D_5}=e^{-L_1^{(1-\eta)^2}}\ll e^{-L_1^{1-\varepsilon/2}} \le e^{-L_k^{1-\varepsilon/2}}, \qquad \forall 1\le k\le M,
\]
the goodness of $B_x$ is preserved under the perturbation of $E$ by $\lambda_i(t)$ for all $x\in Q\setminus \mcG$, by the inductive hypothesis.  
Therefore, by Poisson's formula, for $x\in Q\setminus \mcG$ we have  
\begin{align*}
	&|\psi_i(t)(x)|  = \left| -\bigl(G_{B_x}(\lambda_i(t)) \cdot R_{B_x} T R_{Q\setminus B_x}\cdot \psi_i(t) \bigr)(x) \right| \\
	           &\le \sum_{\substack{w\in B_x \\ w'\notin B_x}} |G_{B_x}(\lambda_i(t))(x,w)| \cdot |f(w-w')|\cdot \, |\psi_i(t)(w')| \\
			   &\lesssim \sum_{ w'\notin B_x} \left( \sum_{\substack{w\in B_x \\ |w-x|< m/200}} |G_{B_x}(\lambda_i(t))(x,w)|  +\sum_{\substack{w\in B_x \\ |w-x|\ge  m/200}} |G_{B_x}(\lambda_i(t))(x,w)| \right) e^{-c_T|w-w'|}  |\psi_i(t)(w')|\\
			   &\lesssim \sum_{ w'\notin B_x} \left( \sum_{\substack{w\in B_x \\ |w-x|< m/200}}\|G_{B_x}(\lambda_i(t)) \|  e^{-c_T|w-w'|} + \sum_{\substack{w\in B_x \\ |w-x|\ge  m/200}} e^{-\gamma_m |x-w|}  e^{-c_T|w-w'|} \right)  |\psi_i(t)(w')| .
\end{align*}
When $|w-x|<m/200$, since $\dist(x,Q\setminus B_x)\ge m/10$, we have $|x-w'|\ge m/10$ and $|w-w'|\ge \frac{1}{2}|x-w'|$. Therefore
\[
\|G_{B_x}(\lambda_i(t)) \|  e^{-c_T|w-w'|}  \le \exp \{ m^{1-\varepsilon/4}-\frac{c_T}{2}|x-w'|\}\le \exp \{-\frac{c_T}{4}|x-w'|\}.
\]
Hence we obtain the estimate 
\begin{align}\label{Poisson's formula}
\notag	|\psi_i(t)(x)| & \lesssim \sum_{\substack{w\in B_x \\ w'\notin B_x}} e^{-\gamma_0 |x-w'|} |\psi_i(t)(w')|  \\
	     \notag      &\lesssim  \# B_x \cdot \#Q \cdot \sup_{ w'\notin B_x} e^{-\gamma_0 |x-w'|} |\psi_i(t)(w')| \\
			   	           &\le \sup_{ w'\notin B_x}   e^{-\frac{1}{2}\gamma_0 |x-w'|} |\psi_i(t)(w')| .
\end{align}
In the above estimate we used $|x-w'|\ge m/10$, which implies $\# B_x \cdot \#Q \, e^{-\frac{1}{2}|x-w'|}\le m \whL \, e^{-\frac{1}{20}\gamma_0 m }\ll 1$. Iterating \eqref{Poisson's formula} yields 
\[
|\psi_i(t)(x)| \le  \exp\left\{ -\frac{1}{2}\gamma_0 |x-x_1|- \frac{1}{2}\gamma_0 |x_1-x_2| - \cdots - \frac{1}{2}\gamma_0 |x_{s-1}-x_s|\right\} |\psi_i(t)(x_s)|
\]	
for some sequence $x_0=x,x_1,x_2,\dots$ with $x_s\notin B_{x_{s-1}}$ and $|x_s-x_{s-1}|>m_{x_{s-1}}/10 \ge L_M/10$. We stop the iteration when either $x_s\in \mcG$ or the number of steps $s\ge \dist(x,\mcG)+L_M$. This finally yields 
\[
|\psi_i(t)(x)|\le \exp\left\{-\frac{1}{2}\gamma_0 \left(\dist(x,\mcG)\vee \frac{L_M}{10}\right) \right \} \le \exp\left\{-\frac{1}{4}\gamma_0 \left(\dist(x,\mcG)+\frac{L_M}{10}\right) \right \}.
\]
This proves \eqref{decay outside mcG}.

Moreover, we have 
\begin{align*}
	\| \psi_i(t)\|_{\ell^2(Q\setminus \mcG)}^2 &=\sum_{x\in Q\setminus \mcG} |\psi_i(t)(x)|^2 \\
                &\leq \sum_{x\in Q\setminus \mcG} \exp\left\{-\frac{1}{2}\gamma_0 \left(\dist(x,\mcG)+\frac{L_M}{10}\right) \right \} \leq \# Q \cdot e^{-\frac{1}{20}\gamma_0 L_M}\ll 1,
\end{align*}
and hence $\| \psi_i(t)\|_{\ell^2( \mcG)} = \bigl(1-\| \psi_i(t)\|_{\ell^2(Q\setminus \mcG)}^2\bigr)^{1/2} \ge 1/2$.
\end{proof}

$\quad$\\

Now recall \eqref{S baralpha}. For every $x\in S=S_{\baralpha}\cap Q$, we have 
\[
\operatorname{dist}(x,\mcG)\ge \operatorname{dist}\Bigl(x,\bigcup_{1\le k'\le K'}Q'_{K'}\Bigr)\ge 10L_1=10D_3.
\]
Hence, Claim \ref{decay of eigenvector} yields that 
\begin{equation}\label{previous bootstrap eigenvalue}
	|\lambda_i(t)-E| \le e^{-D_5}\Longrightarrow |\psi_i(t)(x)|\le e^{-2\gamma_0 D_3}, \qquad \forall x\in S.
\end{equation}
Together with a bootstrap argument of Bourgain (see \cite[(3.57)-(3.61)]{LSZ25} or \cite[Claim 5.11]{DS20} for details), \eqref{previous bootstrap eigenvalue} implies 
\begin{equation}\label{bootstrap eigenvalue}
	\min_{t\in [0,1]^S}|\lambda_i(t)-E| \le e^{-D_4}\Longrightarrow  \max_{t\in [0,1]^S} |\lambda_i(t)-E| \le e^{-D_4} + \# S\cdot e^{-2\gamma_0 D_3}\le 2e^{-D_4}.
\end{equation}
Since the Bernoulli potential $V(x)$ for $x\in S$ takes values in $\{0,1\}$, we define the index set 
\begin{equation}\label{mcK index eigenvalue}
\mcK:=	\left\{ i: \min_{V_S \in \{ 0,1\}^S}|\lambda_i(V_s)-E| \le e^{-D_4}\right\}.
\end{equation}
For each $i\in\mcK$, \eqref{bootstrap eigenvalue} ensures that
\[
\max_{t\in [0,1]^S}|\lambda_i(t)-E| \le 2 e^{-D_4}\ll e^{-D_5},
\]
and Claim \ref{decay of eigenvector} then implies that
\begin{equation}\label{eigenvectors concentrated in mcG}
	 \inf_{t\in [0,1]^S}\| \psi_i(t)\|_{\ell^2(\mcG)}\geq 1/2, \ \forall i\in \mcK.
\end{equation}
With \eqref{eigenvectors concentrated in mcG} in hand, the orthogonality of $\{\psi_i\}_{1\leq i\leq \#Q}$, together with a standard Hilbert-Schmidt argument, yields
\begin{equation}\label{cardinality of mcK}
	\# \mcK\leq 2 \sqrt{\# \mcG} \lesssim L_M^{\frac{1}{2}}=L_0^{\kappa/2}, \ \kappa=(1-\varepsilon)^M \leq \varepsilon/10.
\end{equation}

\begin{claim}\label{Wegner set inclusion}
    For the randomness $\omega_S$ on $S$ (conditioned on $\mcC_\beta$), we have 
	\begin{equation}\label{DS20 event}
			\{\omega_S: \|G_{Q}(E)\| > e^{D_1} \} =\{\omega_S: \dist(\spec(H_Q),E) < e^{-D_1} \}\subset \bigcup_{\substack{j_1,j_2\in \mcK \\ 0\leq \ell \leq 10\#\mcK }}\mcE_{j_1,j_2,\ell}.
	\end{equation}
Here, for $1\leq j_1, j_2 \leq \# Q$ and an integer $0\leq \ell \leq 10\#\mcK  \lesssim L_0^{\varepsilon/20}$, we denote by $\mcE_{j_1,j_2,\ell}$ the event (concerning $\omega_S$) such that
\[
|\lambda_{j_1}-E|\vee|\lambda_{j_2}-E|<s_{\ell} \quad\text{and}\quad |\lambda_{j_1-1}-E|\wedge|\lambda_{j_2+1}-E|\geq s_{\ell},
\]
where 
\[
s_{\ell}=\exp\{-D_1+(2D_2 -2D_4+C_{\rm ratio})\ell\}
\]
with a numerical constant $C_{\rm ratio}\gg1$ to be specified later. 
\end{claim}

\begin{proof}[Proof of Claim \ref{Wegner set inclusion}]
	The claim is just a restatement of \cite[Claim 5.11]{DS20}, and we just sketch its proof. Suppose $\zeta \in \{\omega_S: \|G_{Q}(E)\| > e^{D_1} \} $. Then there exists an eigenvalue $\lambda_{j_0}(\zeta)$ satisfying 
\[|\lambda_{j_0}(\zeta)-E|<\exp\{-D_1\}\leq \exp\{-D_4\}.\]
Therefore, by our definition of $\mcK$, we have $j\in \mcK$. Recall that we let $0\leq \ell\leq 10 \#\mcK$ and $s_0=\exp\{-D_1\}$. Thus the interval $(E-s_0,E+s_0)$ contains $\lambda_{j_0}(\zeta)$. Moreover, by pigeonhole principle, there must be some $0\leq \ell' \leq 10 \#\mcK$ such that 
\[\big( (E-s_{\ell'+1},E-s_{\ell'}]\cup[E+s_{\ell'},E+s_{\ell'+1}) \big)\cap \{\lambda_j(\zeta):\ j-1 \ {\rm or} \ j \ {\rm or}\ j+1\in\mcK\}=\emptyset .\]
Now from  $\lambda_{j_0}(\zeta)\in (E-s_0,E+s_0)$, it follows  that  
\[\{j\in \mcK:\ E-s_{\ell'}<\lambda_j(\zeta)<E+s_{\ell'} \}\neq \emptyset.\]
 Therefore, we can define
\begin{align*}
   &j_1=\min\{j\in \mcK:\ E-s_{\ell'}<\lambda_j(\zeta)<E+s_{\ell'} \},\\
   & j_2=\max\{j\in \mcK:\ E-s_{\ell'}<\lambda_j(\zeta)<E+s_{\ell'} \},
\end{align*}
and then  $\mcE_{j_1,j_2,\ell'}$ happens. Hence $\zeta\in \mcE_{j_1,j_2,\ell'}$, which completes the proof.
\end{proof}

For the event $\mcE_{j_1,j_2,\ell}$, we have the following probability estimate:
\begin{claim}\label{Sperner prob claim}
	$\P _{\omega_S} (\mcE_{j_1,j_2,\ell}) \leq D_0^{-\frac{1}{2}+2\varepsilon/3}.$ 
\end{claim}
\begin{proof}[Proof of Claim \ref{Sperner prob claim}]
Now we have a set of free sites $S=S_{\baralpha}\cap Q$. In order to apply the quantitative unique continuation theorem, we need to further restrict $S$. Since \eqref{L1 bad cube far from boundary Qs} ensures that the distance between $\cup_{1\le k'\le K'}Q'_{k'}$ and the boundary of $Q$ is larger than $D_0/8\gg \sqrt{D_2D_3}$, the same argument used in constructing $Q_1,\dots,Q_K$ allows us to find a scale 
\[
\whD\sim \sqrt{D_2D_3}
\]
such that there is a collection of $\whD$-intervals $Q''_1,Q''_2,\dots,Q''_{K''}$ with $K''\le K'$ in $Q$ satisfying:
\begin{itemize}
	\item For every $Q'_{k'}$, there exists $1\le k''\le K''$ such that $Q_{k'}\subset Q_{k''}$ and 
	        \begin{equation}\label{L1 bad cube far from boundary Q''s}
				    \dist(Q'_{k'}, Q \setminus Q''_{k''})\ge \frac{1}{8}\whD;
			\end{equation}
	    
	\item $\dist(Q''_{k''},Q''_{s''})\ge 10\whD$ for any $k''\neq s''$.
\end{itemize}

Now recall that we take the energy $E\in [2^{-6000d_T}\delta,\delta]$. Hence, for $i\in \mcK$, we have 
\begin{equation}\label{region of lambda i}
	2\delta\geq E+e^{-D_4}\geq \lambda_i(V_S)\geq E-2e^{-D_4}\geq 2^{-6000d_T-1}\delta, \quad \forall\  V_s\in \{0,1\}^S. 
\end{equation}
Since $\psi_i(V_S)$ satisfies the equation 
\begin{equation}\label{equation eigen}
	(T_Q+ \lambda V_Q-\lambda_i(V_S)) \psi_i(V_S)=0, 
\end{equation}
we check the lower bound conditions \eqref{lower bound on potential, R0} and \eqref{lower bound on potential, R-}. Since $V$ only takes values $1$ or $0$:
\begin{itemize}
	\item When $T\in \mathscr{R}_0$, recalling that $p_{d_1},q_{d_1}\neq 0$ are (possibly complex) parameters depending on $T$, we have 
	    \begin{align}\label{low bound quanti in R0}
			\inf_{n\in Q} \left|\lambda V(n)-\lambda_i(V_S)+\frac{p_{d_1}}{q_{d_1}}\right| &= \left|\lambda-\lambda_i(V_S)+\frac{p_{d_1}}{q_{d_1}}\right| \wedge \left|-\lambda_i(V_S)+\frac{p_{d_1}}{q_{d_1}}\right|.
 		\end{align}
		By taking $\delta$ sufficiently small, if $\lambda+p_{d_1}/q_{d_1}\neq 0$, then \eqref{region of lambda i} yields a lower bound in \eqref{low bound quanti in R0} bounded away from zero by some constant; if $\lambda+p_{d_1}/q_{d_1}=0$, then \eqref{region of lambda i} yields a lower bound in \eqref{low bound quanti in R0} larger than $2^{-6000d_T-1}\delta$.
	\item When $T\in \mathscr{R}_-$, we have 
	    \begin{align}\label{low bound quanti in R-}
			\inf_{n\in Q} \left|\lambda V(n)-\lambda_i(V_S)\right| &= \left|\lambda-\lambda_i(V_S)\right| \wedge \left|-\lambda_i(V_S)\right|.
 		\end{align}
		By taking $\delta$ sufficiently small, \eqref{region of lambda i} yields a lower bound in \eqref{low bound quanti in R-} larger than $2^{-6000d_T-1}\delta$.
\end{itemize}

Therefore, the conditions \eqref{lower bound on potential, R0} and \eqref{lower bound on potential, R-} hold for equation \eqref{equation eigen} with $b=2^{-6000d_T-1}\delta$. 
Applying Theorem \ref{QUC for rational class, refined to free sites} to each $Q''_{k''}\subset Q$ with the set $S_{\baralpha} \cap Q''_{k''}$, we conclude that 
\begin{equation*}
	                       \#\Bigl\{ x\in \mcS_{\baralpha}\cap Q''_{k''} : |\psi_i(V_S) (x)| > (b^{-1}C)^{-\whD} \| \psi_i(V_S)  \|_{\ell^\infty(Q''_{k''})} \Bigr\} \;\ge\; \# (\wtS_{\baralpha}\cap Q''_{k''}).
\end{equation*}
A simple computation shows that 
\[
\whD\sim \sqrt{D_2D_3}= \whL ^{\frac{1}{2}(1-\eta)^2+\frac{1}{2}(1-\eta)^3} \gg N^{1-\varepsilon/3}.
\]
Therefore, by taking $a=3$ in Remark \ref{rmk on dense of S baralpha}, we obtain 
\[ \# (\wtS_{\baralpha}\cap Q''_{k''}) \geq \whD^{1-\rho},\]
and hence 
\begin{equation}\label{transversality set in Q''_k''}
	                       \#\Bigl\{ x\in \mcS_{\baralpha}\cap Q''_{k''} : |\psi_i(V_S) (x)| > (b^{-1}C)^{-\whD} \|  \psi_i(V_S) \|_{\ell^\infty(Q''_{k''})} \Bigr\} \;\ge\; \whD^{1-\rho}
\end{equation}
holds for each $Q''_{k''}$. 
On the other hand, \eqref{eigenvectors concentrated in mcG} ensures that 
\[ \|\psi_i(V_S) \|_{\ell^{\infty}(\mcG)} \geq \|\psi_i(V_S) \|_{\ell^{2}(\mcG)} / \sqrt{\# \mcG} \gtrsim L_M^{-\frac{1}{2}}, \]
and together with \eqref{transversality set in Q''_k''} we conclude that there is at least one choice of $k''$ such that 
\begin{equation}\label{transversality set cardinality}
	                       \#\Bigl\{ x\in \mcS_{\baralpha}\cap Q''_{k''} : |\psi_i(V_S) (x)| \gtrsim L_M^{-\frac{1}{2}} (b^{-1}C)^{-\whD}  \Bigr\} \;\ge\; \whD^{1-\rho}.	
\end{equation}
Recall that $b$ (which comes from the lower bound condition in the QUC) is of order $b\sim \delta\sim (\log N_{\rm in})^{-6000d_T}$. The transversality estimate in \eqref{transversality set cardinality} can be further estimated as 
\begin{align*}
   	L_M^{-\frac{1}{2}} (b^{-1}C)^{-\whD}\geq \exp \left\{ -\frac{\kappa}{2}\log L_0 -\mcO_+(\log\log N_{\rm in}) \sqrt{\frac{D_3}{D_2}}\cdot D_2 \right\}  \gg e^{-D_2},
\end{align*}
since
\[  (\log\log N_{\rm in}) \sqrt{\frac{D_3}{D_2}} =(\log\log N_{\rm in}) \whL^{-\eta/2} \leq (\log\log N_{\rm in}) N_{\rm in} ^{-\eta/2} \ll 1.\]
Therefore, \eqref{transversality set cardinality} becomes
\begin{equation}\label{transversality set cardinality, refined}
	                       \#\Bigl\{ x\in \mcS_{\baralpha}\cap Q''_{k''} : |\psi_i(V_S) (x)| \geq e^{-D_2}  \Bigr\} \;\ge\; \whD^{1-\rho}.	
\end{equation}

Since $Q''_{k''}\subset Q$ and hence $\mcS_{\baralpha}\cap Q''_{k''}\subset S$, we conclude that
\begin{lem}\label{existence lemma of S uc}
For each $i\in \mcK$ and every realization of the potential $V_S\in \{0,1\}^S$ on the set of free sites $S$, there exists a subset 
\[S_{\rm uc}= S_{\rm uc}(i,V_s) \subset S\]
 (which may depend on $i$ and $V_S$) such that 
\begin{equation}
	|\psi_i(V_S) (x)| \geq e^{-D_2} \quad \text{on } S_{\rm uc}, \quad \# S_{\rm uc}\geq \whD^{1-\rho} \gtrsim (D_2 D_3)^{\frac{1-\rho}{2}}.
\end{equation}
\end{lem}

With Lemma \ref{existence lemma of S uc} in hand, we go to the estimate on the probability of $\mcE_{j_1,j_2,\ell}$ with $j_1,j_2\in \mcK$. 
For $i=0,1$, let $\mcE_{j_1,j_2,\ell,i}$ denote the event that 
\begin{equation}\label{pigeonhole event}
	\mcE_{j_1,j_2,\ell} \ {\rm and} \ \#\left( S_{\rm uc}(j_1,\omega_S)  \cap \{n\in S:\omega(n) = i\}  \right)\geq \frac{1}{2} \whD^{1-\rho} .
\end{equation}
By the pigeonhole principle and Lemma \ref{existence lemma of S uc}, we have 
\begin{equation}\label{i=0,1 event inclusion}
  \mcE_{j_1,j_2,\ell} \subset \mcE_{j_1,j_2,\ell,0}\cup  \mcE_{j_1,j_2,\ell,1}.
\end{equation}

Now if $\omega_S \in \mcE_{j_1,j_2,\ell,i}$ and $ x\in S_{\rm uc}(j_1,\omega_S)  \cap \{n\in S:\omega(n) = i\}\subset S$ is a site such that
\[ \omega (x)=i \quad {\rm and}\quad |\psi_{j_1}(\omega_S)(x)|\geq e^{-D_2},\]
then we change the value of $V$ at $x$, i.e.,  take 
\begin{equation*}
  \omega_S^{(x)}(y)=\begin{cases}
    \omega(y), \ {\rm if} \ y\in S \ {\rm and}\ y \neq x,\\
    1-\omega(x), \ {\rm if} \ y\in S \ {\rm and}\ y=x.
  \end{cases}
\end{equation*}
We show the following conclusion:
\begin{lem}\label{lemma Sperner structure}
	$\omega_S^{(x)} \notin \mcE_{j_1,j_2,\ell,i}$.
\end{lem}
\begin{proof}[Proof of Lemma \ref{lemma Sperner structure}]

We only consider the case $i=0$, as the proof of the case $i=1$ is analogous. Define the shifted operator
\[\widetilde{H}_{Q}(\omega_S)=H_{Q}(\omega_S)-E+s_\ell,\]
of which all eigenvalues are $\widetilde{\lambda}_j=\lambda_j(\omega_S)-E+s_{\ell}$  (For $i=1$, the corresponding operator is $-(H_{Q}(\omega_S)-E-s_{\ell})$).  Set
\begin{equation}\label{r_1,r_2}
  r_1=2s_\ell,\ r_2=s_{\ell}+s_{\ell+1}.
\end{equation}
Then since $\omega_S\in \mcE_{j_1,j_2,\ell,i}$,  the following ordering holds true:
\[0<\widetilde{\lambda}_{j_1}\leq \wtlambda_{j_2}<r_1<r_2<\wtlambda_{j_2+1}.\]
At the site $x$, we also have 
\begin{equation}\label{r_3}
	|\psi_{j_1}(\omega_S)(x)|^2 \geq e^{-2D_2}:=r_3.
\end{equation}
We also set 
\begin{equation}\label{r_5}
  r_5= s_{\ell}+e^{-D_5}.
\end{equation}
Consider now  a $j$ such that $r_2<\wtlambda_j<r_5$. Then $|\lambda_j(\omega_S)-E|\leq r_5-s_{\ell}<e^{-D_5}$. 
Therefore, Claim \ref{decay of eigenvector} applies to $\lambda_j(\omega_S)$ and we thus obtain (since the construction \eqref{S baralpha} ensures that $\dist(x,\mcG)>10 L_1=10D_3$)
\begin{equation}\label{r_4}
      \sum_{j:r_2<\wtlambda_j<r_5}|\psi_j(x)|^2 \leq \# Q \cdot \exp\{-2\gamma_0 D_3\}\leq \exp\{-\gamma_0 D_3\}\leq e^{-2D_4}:= r_4.
\end{equation}

Now, after defining $r_1,r_2,r_3,r_4$ and $r_5$, one can check that 
\begin{itemize}
	\item $s_l \leq \exp\{-D_1+(2D_2 -2D_4+C_{\rm ratio}) L_0^{\frac{\varepsilon}{20}}\}\leq e^{-D_1/2}$, and therefore 
	             \[r_1<r_2<r_3<r_4<r_5;\]
	\item \begin{align*}
		\frac{r_2r_3}{r_4}&=r_1 \frac{1+\exp\{2D_2 -2D_4+C_{\rm ratio}\}}{2}\cdot \exp\{-2D_2+2D_4\} \geq \frac{1}{2}e^{C_{\rm ratio}}r_1;
	\end{align*}
	\item \begin{align*}
		r_3 r_5 \geq e^{-D_5-2D_2}\gg 2e^{-D_1/2}\geq r_1.
	\end{align*}
\end{itemize}
By choosing $C_{\rm ratio}$ sufficiently large, we thus ensure $r_1,r_2,r_3,r_4$ and $r_5$ satisfying the parameter condition of the rank-one perturbation lemma, Lemma \ref{rank-one perturbation} from \cite{DS20}. Using that lemma, we conclude that 
\[
\operatorname{trace} \mathbf{1}_{[r_1, \infty)}(\widetilde{H}_{Q}(\omega_S)) < \operatorname{trace} \mathbf{1}_{[r_1, \infty)}(\widetilde{H}_{Q}(\omega_S^{(x)})).
\]
This inequality is precisely equivalent to $\lambda_{j_2}(\omega_S^{(x)}) > s_{\ell+1}$. Hence, $\omega_S^{(x)} \notin \mcE_{j_1,j_2,\ell}$, and we obtain
\begin{equation}
	\omega_S^{(x)} \notin \mcE_{j_1,j_2,\ell,0} \subset \mcE_{j_1,j_2,\ell}.
\end{equation}

\end{proof}

Lemma \ref{lemma Sperner structure} means that $\mcE_{j_1,j_2,\ell,i}$ is a $\rho$-Sperner family of $ S$ (for the definition of $\rho$-Sperner, see \cite[Definition 4.1]{DS20}), where by \eqref{pigeonhole event},
\begin{equation}\label{Sperner rho final}
  \rho= \frac{\frac{1}{2} \whD^{1-\rho}}{\# S }.
\end{equation}\label{Sperner estimate for i=0}
Thus, recall that $S=S_{\baralpha}\cap Q$ and use \cite[Theorem 4.2]{DS20}. We get 
\begin{align}
  \P_{\omega_S}(\mcE_{j_1,j_2,\ell,i} ) \leq \left(\#S\right)^{-\frac{1}{2}} \cdot \rho^{-1}=\frac{2\sqrt{\# S}}{\whD^{1-\rho}} \lesssim D_0^{\frac{1}{2}}D_2^{-\frac{1-\rho}{2}}D_3^{-\frac{1-\rho}{2}}\ll D_0^{-\frac{1}{2}+2\varepsilon/3}
\end{align}
since we choose $\rho\leq \varepsilon/12$. Finally, \eqref{i=0,1 event inclusion} yields 
\begin{align}
  \P_{\omega_S}(\mcE_{j_1,j_2,\ell} ) \leq \P_{\omega_S}(\mcE_{j_1,j_2,\ell,0} )+\P_{\omega_S}(\mcE_{j_1,j_2,\ell,1} ) \leq D_0^{-\frac{1}{2}+2\varepsilon/3},
\end{align}
and we finish the proof of Claim \ref{Sperner prob claim}.
\end{proof}

Finally, combining Claim \ref{Wegner set inclusion} and Claim \ref{Sperner prob claim}, we conclude that 
\begin{align*}
	\P_{\omega_S} (\{\omega_S: \|G_{Q}(E)\| > e^{D_1} \}) & \leq \P_{\omega_S} \left( \bigcup_{\substack{j_1,j_2\in \mcK \\ 0\leq \ell \leq 10\#\mcK }}\mcE_{j_1,j_2,\ell}\right) \\
	                  & \leq\sum_{\substack{j_1,j_2\in \mcK \\ 0\leq \ell \leq 10\#\mcK }} \P_{\omega_S} \left( \mcE_{j_1,j_2,\ell}\right) \\
					  & \lesssim (\#\mcK)^3 D_0^{-\frac{1}{2}+2\varepsilon/3} \\
					  & \overset{\eqref{cardinality of mcK}}{\lesssim} L_0^{\varepsilon/20} D_0^{-\frac{1}{2}+2\varepsilon/3} \ll D_0^{-\frac{1}{2}+\varepsilon}=\whL^{-\frac12+\varepsilon}.
\end{align*}
The proof of Claim \ref{L2 bound on Qs} is then immediately complete, once we note that $e^{D_1}=e^{\whL^{1-\eta}}\ll e^{\whL^{1-\varepsilon/3}}$.

\end{proof}

\section{Elimination of the energy and proof of the localization}\label{localization section}
This section is devoted to proving our main result, Theorem \ref{Localization near the edge band}, using the LDT Theorem \ref{LDT for the large scales}. The obstacle is that Theorem \ref{LDT for the large scales} is stated for a fixed energy $E\in [2^{-6000d_T}\delta,\delta]$, whereas the eigenvalues themselves depend on the randomness (and hence fluctuate). Therefore, in order to obtain Anderson localization, one must eliminate the energy dependence on the randomness, i.e., approximate all eigenvalues by some fixed energies.

Following the above reasoning, the almost sure Anderson localization in the region $[2^{-6000d_T}\delta,\delta]$ follows from Theorem \ref{LDT for the large scales} by applying the Peierls argument developed in \cite[Section 7]{BK05}. (Indeed, such an argument applies to arbitrary dimensions $\Z^d$.) An axiomatic version can also be found in \cite[Section 6 and 7]{GK13}.

Moreover, by the trick described in Remark \ref{rmk on LDT for large scales} (2), the localization region can be enlarged from $[2^{-6000d_T}\delta,\delta]$ to the entire interval $[0,\delta]$, and hence Theorem \ref{Localization near the edge band} is established.

Since the Peierls argument of Bourgain–Kenig is now a relatively standard procedure, we will not elaborate on it here. However, we point out that the only subtle modification required for the argument to apply to our model is that our operator is long-range (the argument itself is dimension-independent; here we take $H$ on $\Z^d$), and hence formula \cite[(7.10)]{BK05} must be reproved by a different argument. We address this below.

\begin{lem}\label{Peierls argument in long range}
Let \eqref{longrange schrodinger operator} be defined on $\Z^d$, with $T$ satisfying \textup{\textbf{(A1)}} and \textup{\textbf{(A2)}}. Let $L\gg 1$. Assume that for all $E\in [2^{-6000d_T}\delta,\delta]$, there exists a set $\mathscr{C}$ such that $\Lambda_L\subset \mathscr{C}\subset \Lambda_{\frac{3}{2}L}$ and for every $x\in \partial^{-}\mathscr{C}$, there is a good $L^{1/2}$-cube $\Lambda'$ for which $G_{\Lambda'}(E)$ satisfies \eqref{L2 estimate for large scales} and \eqref{Off diagonal estimate for large scales}, and 
\begin{equation}\label{La' intersect both C and out C}
	\Lambda_{L^{1/2}/10}(x) \subset \Lambda'. 
\end{equation}
Then for any eigenvalue $\mathscr{E}$ of $H$ lying in $[2^{-6000d_T}\delta,\delta]$ with (normalized) eigenfunction $\psi$, we have 
\begin{equation}\label{approx eigenvalue of H in 2L scale}
          \| (H_{\Lambda_{2L}}-\mathscr{E}) R_{\mathscr{C}}\psi \| \leq e^{-\frac{1}{70}\gamma_0 L^{1/2}}	,
\end{equation}
where $\mathscr{C}$ is the set corresponding to $\mathscr{E}$ as in the assumption above.
\end{lem}

\begin{proof}[Proof of Lemma \ref{Peierls argument in long range}]
We have the eigen equation
\begin{equation}\label{mscrE eigen equation}
	(H-\mathscr{E})\psi =0.
\end{equation}
For any $y\in \mathscr{C}$ satisfying $\dist(y,\partial^+ \mathscr{C})\leq \frac{1}{20}L^{1/2}$, our assumption \eqref{La' intersect both C and out C} ensures that there is a good $L^{1/2}$-cube $\La'$ such that 
\begin{equation}\label{slightly shrink inclusion}
	\La_{L^{1/2}/20}(y)\subset \La'.
\end{equation}
Applying Poisson's formula to \eqref{mscrE eigen equation}, we obtain 
\begin{align}\label{decay near the boundary}
	\notag |\psi(y)| &\leq \left| -\bigl(G_{\La'}(\mathscr{E}) \cdot R_{\La'} T R_{(\La')^c}\cdot \psi \bigr)(y) \right| \\
	   \notag        &\le \sum_{\substack{w\in \La' \\ w'\notin \La'}} |G_{\La'}(\mathscr{E})(y,w)| \cdot |f(w-w')|\cdot \, |\psi(w')| \\
		\notag	   &\lesssim \sum_{ w'\notin \La'} \left( \sum_{\substack{w\in \La' \\ |w-y|< L^{1/2}/200}} |G_{\La'}(\mathscr{E})(y,w)|  +\sum_{\substack{w\in \La' \\ |w-y|\ge  L^{1/2}/200}} |G_{\La'}(\mathscr{E})(y,w)| \right) e^{-c_T|w-w'|} \\
		\notag	   &\lesssim \sum_{ w'\notin \La'} \left( \sum_{\substack{w\in \La' \\ |w-y|< L^{1/2}/200}}e^{L^{\frac{1-\varepsilon/4}{2}}}  e^{-c_T|w-w'|} + \sum_{\substack{w\in \La' \\ |w-y|\ge  L^{1/2}/200}} e^{-\frac{1}{2}\gamma_0 |y-w|}  e^{-c_T|w-w'|} \right)   \\
		\notag	   &\lesssim \sum_{\substack{w\in \La' \\ w'\notin \La'}} e^{-\frac{1}{2}\gamma_0 |y-w'|} \lesssim \# \La' \cdot \sum_{w':|w'-y|\geq L^{1/2}/20} e^{-\frac{1}{2}\gamma_0|y-w'| } \\
			   &\leq e^{-\frac{1}{50}\gamma_0 L^{\frac{1}{2}}}  \qquad \qquad \text{for every } y\in \mathscr{C} \text{ with } \dist(y,\partial^+ \mathscr{C})\leq \frac{1}{20}L^{1/2}.
\end{align}
In the above estimate, we used \eqref{slightly shrink inclusion}, which implies $|w'-y|\geq L^{1/2}/20$. Now for $x\in \La_L$, we distinguish the following cases:
\begin{itemize}
	\item If $x\in \mathscr{C}$ and $\dist(x,\partial^+ \mathscr{C})>\frac{1}{40}L^{1/2}$, then 
	       \begin{align}\label{shift to outside mscrC}
	         (H_{\Lambda_{2L}}-\mathscr{E}) R_{\mathscr{C}}\psi (x) & = R_{\La_{2L}}  (H-\mathscr{E}) R_{\mathscr{C}}\psi (x)  \overset{\eqref{mscrE eigen equation}}{=} -(H_{\Lambda_{2L}}-\mathscr{E}) R_{\Z^d\setminus \mathscr{C}}\psi (x).
            \end{align}
			Therefore, 
			\begin{align}\label{region 1 in La 2L}
				\notag | (H_{\Lambda_{2L}}-\mathscr{E}) R_{\mathscr{C}}\psi (x)| & \leq \sum_{y\notin \mathscr{C}} |f(x-y)|\cdot |\psi(y)| \lesssim \sum_{y:|y-x|\geq L^{1/2}/40} e^{-c_T|x-y|} \\
				                           & \leq e^{-\frac{1}{50}c_T L^{1/2}}.
			\end{align}

	\item If $x\in \mathscr{C}$ and $\dist(x,\partial^+ \mathscr{C})\leq \frac{1}{40}L^{1/2}$, then applying \eqref{decay near the boundary} gives
			\begin{align}\label{region 2 in La 2L}
				\notag | (H_{\Lambda_{2L}}-\mathscr{E}) R_{\mathscr{C}}\psi (x)| & \leq \lambda |\psi(x)| +\sum_{y\in \mathscr{C}} |f(x-y)|\cdot |\psi(y)| \\
           \notag				  &\lesssim  e^{-\frac{1}{50}\gamma_0 L^{\frac{1}{2}}} +\sum_{y\in \mathscr{C}} e^{-c_T|x-y|}\cdot |\psi(y)| \\
		              \notag				  &\lesssim  e^{-\frac{1}{50}\gamma_0 L^{\frac{1}{2}}} +\sum_{\substack{y\in \mathscr{C} \\ \dist(y,\partial^+ \mathscr{C})\leq \frac{1}{20}L^{1/2}}} e^{-\frac{1}{50}\gamma_0 L^{\frac{1}{2}}} +\sum_{\substack{y\in \mathscr{C} \\ \dist(y,\partial^+ \mathscr{C})> \frac{1}{20}L^{1/2}}} e^{-c_T|x-y|} \\
				         \notag                  & \lesssim e^{-\frac{1}{50}\gamma_0 L^{\frac{1}{2}}} + \#\mathscr{C}\cdot e^{-\frac{1}{50}\gamma_0 L^{\frac{1}{2}}} +\# \mathscr{C}\cdot e^{-c_T (\frac{L^{1/2}}{20}-\frac{L^{1/2}}{40})} \\
										   & \leq e^{-\frac{1}{60}\gamma_0 L^{\frac{1}{2}}}.
			\end{align}		
	
	\item If $x\in\La_{2L}\setminus \mathscr{C}$, then 
	          \[| (H_{\Lambda_{2L}}-\mathscr{E}) R_{\mathscr{C}}\psi (x)|  \leq  \sum_{y\in \mathscr{C}} |f(x-y)|\cdot |\psi(y)|. \]
			  Therefore, the same argument as in \eqref{region 2 in La 2L} yields
			\begin{align}\label{region 3 in La 2L}
				 | (H_{\Lambda_{2L}}-\mathscr{E}) R_{\mathscr{C}}\psi (x)|  \leq e^{-\frac{1}{60}\gamma_0 L^{\frac{1}{2}}}.
			\end{align}	
\end{itemize} 
Combining \eqref{region 1 in La 2L}, \eqref{region 2 in La 2L} and \eqref{region 3 in La 2L}, we obtain 
\[
\|(H_{\Lambda_{2L}}-\mathscr{E}) R_{\mathscr{C}}\psi \|_{\ell^{\infty}(\La_{2L})}\leq e^{-\frac{1}{60}\gamma_0 L^{\frac{1}{2}}},
\]
and therefore 
\begin{equation*}
	\|(H_{\Lambda_{2L}}-\mathscr{E}) R_{\mathscr{C}}\psi \|_{\ell^{2}(\La_{2L})} \leq \sqrt{\# \La_{2L}} \cdot  \|(H_{\Lambda_{2L}}-\mathscr{E}) R_{\mathscr{C}}\psi \|_{\ell^{\infty}(\La_{2L})} \leq e^{-\frac{1}{70}\gamma_0 L^{\frac{1}{2}}} .
\end{equation*}
Thus, we recover the estimate \cite[(7.10)]{BK05} in our exponential decay long-range hopping setting.
\end{proof}

\appendix

\section{Proof of Theorem \ref{QUC for super-exponential decay}}\label{appendix: proof of QUC for super-exponetial decay}
\begin{proof}[Proof of Theorem \ref{QUC for super-exponential decay}.]
The main strategy of the proof is to apply a Carleman estimate to a cut-off solution. Assume that the Laurent symbol of $T$ is given by
\[
F(z)=\sum_{k\in \mathbb{Z}} f(k)z^{k},\quad |f(k)|\lesssim \exp\{ -|k|^{\alpha} \},\ \alpha>1.
\]
Decompose $F(z)=F_+(z)+F_-(1/z)$ by
\[
F_+(z)=\sum_{k\geq 0} f(k)z^k, \quad F_-(z)=\sum_{k>0}f(-k)z^k.
\]
Obviously, the super-exponential decay of $|f(k)|$ ensures that $F_+(z)$ and $F_-(z)$ are entire functions, and therefore $F(z)$ is holomorphic in $\mathbb{C}_{\times}$. Consider the following exponential-conjugate operator for $\tau>0$,
\begin{equation}\label{conjungate T tau}
T_{\tau}=e^{\tau n} T e^{-\tau n},
\end{equation}
with
\[
T_{\tau}u(m)=\sum_{n\in\mathbb{Z}} e^{\tau m} f(m-n) e^{-\tau n} u(n)=\sum_{k\in\mathbb{Z}} \bigl(e^{\tau k}f(k)\bigr)\cdot u(m-k).
\]
Therefore, $T_{\tau}$ is also a convolution operator, and its Laurent symbol is
\begin{equation}\label{Laurent symbol of T tau}
F_{\tau}(z)=\sum_{k\in\mathbb{Z}} f(k)e^{\tau k} z^k = F(e^{\tau}z).
\end{equation}
The corresponding Fourier symbol of $T_{\tau}$ is 
\begin{equation}
	\hat{f_{\tau}}(\theta)=F_{\tau}(e^{2\pi i\theta}),
\end{equation}
and $T_{\tau}$ is exactly the multiplier $\hat{f_{\tau}}(\theta)$ on $\ell^2(\mathbb{T})$ up to the Fourier transform. Denote the minimal modulus by
\begin{equation}\label{minimal modulus}
	m(\tau) := \inf_{\theta\in \T} \left|\hat{f_{\tau}}(\theta) \right|=\inf_{|z|=e^{\tau}}\left| F(z) \right|.
\end{equation}
Since $\sup_{|z|=r}\left| F_-(1/z)\right|\rightarrow |F_-(0)|=0$ as $r\rightarrow +\infty$, we have 
\begin{equation}\label{m and F+}
	\limsup_{\tau\rightarrow\infty} m(\tau)=\limsup_{\tau\rightarrow \infty} \inf_{|z|=e^{\tau}}\left|F_+(z)\right|.
\end{equation}

Now recall that the order of an entire function 
\[
E(z)=\sum_{k\geq 0} e_k z^k
\]
is defined by 
\[
\rho(E)=\limsup _{n\rightarrow \infty} \frac{n\log n}{\log(1/|e_n|)}.
\]
The super-exponential decay of $|f(n)|$ yields 
\begin{align}\label{F+ is order zero}
	\rho(F_+)=\limsup _{n\rightarrow \infty} \frac{n\log n}{\log(1/|f(n)|)}\leq \limsup _{n\rightarrow \infty} \frac{n\log n}{\log(C\cdot e^{|n|^{\alpha}})}= 0,
\end{align}
since $\alpha>1$. We can apply the following Wiman's theorem on the minimal modulus (cf. \cite{Wi05} or \cite[Section 8.73]{Tit64}):
\begin{thm}[Wiman's Theorem]\label{Wiman}
	If an entire function $E(z)$ has order $\rho(E)<1/2$, then 
	 \[\limsup_{r\rightarrow\infty} \inf_{|z|=r} |E(z)|=+\infty.\]
\end{thm}
 
Applying Theorem \ref{Wiman} together with \eqref{m and F+} and \eqref{F+ is order zero}, we obtain $\limsup_{\tau\rightarrow\infty} m(\tau)=\infty$. Hence we can find a large $\tau_0=\tau_0(T,B)$ such that $m(\tau_0)>2B$, in which case 
\[\|T_{\tau_0}^{-1} \| = \frac{1}{m(\tau_0)}<\frac{1}{2B},\]
so $T_{\tau_0}$ is invertible. Therefore, for any compactly supported $v\in \ell^2_0(\Z)$, we have the following Carleman estimate:
\begin{equation}\label{Carleman estimate}
	\| e^{\tau_0 n} v\|_2= \|T^{-1}_{\tau_0} e^{\tau_0 n} T v\|_2\leq \frac{1}{m(\tau_0)} \|e^{\tau_0 n} T v\|_2.
\end{equation}

Now let $N$ be a large positive integer and let $\chi_N = \chi_{[-N,N]}$ be the cut-off function on $\Z$. Assume $u$ is a solution of 
\[
(T+V)u = 0 \quad \text{on } \Z, \qquad u(0)=1, \quad \|u\|_{\infty} \leq \wtB,
\]
and set $v = \chi_N u$. A simple computation shows that
\begin{equation}\label{cut off equation}
	Tv =-Vv+[T,\chi_N] u.
\end{equation}
Substituting \eqref{cut off equation} into \eqref{Carleman estimate} gives
\begin{align*}
	\| e^{\tau_0 n} v \|_2 &\leq \frac{1}{m(\tau_0)}\| e^{\tau_0 n}\left( -V v+[T,\chi_N]u\right) \|_2 \\
	      &\leq \frac{1}{m(\tau_0)} \|V \|_{\infty} \cdot \|e^{\tau_0 n} v\|_2 + \frac{1}{m(\tau_0)} \|e^{\tau_0 n} [T,\chi_N] u \|_2,
\end{align*}
which is equivalent to 
\[\| e^{\tau_0 n} v \|_2 \leq\frac{1}{m(\tau_0)-\| V\|_{\infty}} \|e^{\tau_0 n} [T,\chi_N] u \|_2. \]
By our choice of $\tau_0$ and since $\| V\|_{\infty}<B$, we obtain 
\[\|e^{\tau_0 n} [T,\chi_N] u \|_2 \geq B \| e^{\tau_0 n} v \|_2. \]
Since $ \| e^{\tau_0 n} v \|_2 \geq |v(0)|^2 = |u(0)|^2=1$, we conclude that 
\begin{equation}\label{weight larger than B}
	\|e^{\tau_0 n} [T,\chi_N] u \|_2 \geq B .
\end{equation}

In \eqref{weight larger than B}, the entries of the commutator are 
\begin{equation}\label{entries for the commutator}
	 [T,\chi_N](m,n)= f(m-n)\bigl( \chi_N(n)-\chi_N(m) \bigr),
\end{equation}
which are nonzero only when $m\in [-N,N], n\notin [-N,N]$ or $m\notin [-N,N], n\in [-N,N]$. We choose some positive integer $R$ and decompose $\mathbb{Z}$ into the following five regions: 
\begin{equation*}
\bigl( (-\infty,-N+R]\cap \mathbb{Z} \bigr) \cup \bigl( (-N+R,N-R]\cap \mathbb{Z} \bigr) \cup \bigl( (N-R,N]\cap \mathbb{Z} \bigr) \cup \bigl( (N,N+R]\cap \mathbb{Z} \bigr) \cup \bigl( (N+R,\infty)\cap \mathbb{Z} \bigr).
\end{equation*}

The following discussion is valid if we take $N \gg_{T,B,\wtB} 1$ and $R\sim_{T,B} N^{\frac{1}{\alpha}}$.
\begin{itemize}
	\item If $n\leq -N+R<-\frac{N}{2}$, then
	       \begin{align}\label{region 1}
			  \notag  \left\|\chi_{\{n\leq -N+R\}} \cdot e^{\tau_0 n} [T,\chi_N] u \right\|_2 & =\left( \sum_{n\leq -N+R} \left(e^{\tau_0 n}\sum_{k\in \Z} [T,\chi_N](n,k) u(k)\right)^2 \right)^{\frac12} \\
           \notag				            &\leq \| u\|_{\infty} \cdot \left(\sum_{n<-\frac{N}{2}} e^{2 \tau_0 n} \left(\sum_{k\in \Z} \big|[T,\chi_N](n,k)\big| \right)^2  \right)^{\frac{1}{2}} \\
							&\leq   \wtB \cdot  \|f\|_1 \cdot \left(\sum_{n<-\frac{N}{2}} e^{2 \tau_0 n}   \right)^{\frac{1}{2}} \leq \frac{1}{100}B .
		   \end{align}
          In the above we used $\big|[T,\chi_N](m,n) \big|\leq |f(m-n)|$ for all $m,n\in \Z$.
    \item If $-N+R< n \leq N-R$, then 
            \begin{align}\label{region 2}
			  \notag  \left\|\chi_{\{-N+R<n\leq N-R\}} \cdot e^{\tau_0 n} [T,\chi_N] u \right\|_2 & =\left( \sum_{-N+R< n\leq N-R} \left(e^{\tau_0 n}\sum_{k\in \Z} [T,\chi_N](n,k) u(k)\right)^2 \right)^{\frac12} \\
			  \notag                       & =\left( \sum_{-N+R< n\leq N-R} \left(e^{\tau_0 n}\sum_{k\notin [-N,N]} f(n-k) u(k)\right)^2 \right)^{\frac12} \\
               \notag				            &\lesssim_T \| u\|_{\infty} \cdot  e^{\tau_0 N} \left(\sum_{-N+R< n\leq N-R}  \left(\sum_{k\notin [-N,N]} e^{-|n-k|^{\alpha}} \right)^2  \right)^{\frac{1}{2}} \\
		        \notag				            &\lesssim_T \wtB \cdot \sqrt{N} e^{\tau_0 N}  \cdot \sum_{|k|\geq R} e^{-|k|^{\alpha}} \\
		       			            &\lesssim_T \wtB \cdot \sqrt{N} e^{\tau_0 N-R^{\alpha}} \leq \frac{1}{100}B. 
		   \end{align}
	\item If $N-R<n\leq N$, then 
	           \begin{align*}
			  \left\|\chi_{\{N-R<n\leq N\}} \cdot e^{\tau_0 n} [T,\chi_N] u \right\|_2 & =\left( \sum_{N-R<n\leq N} \left(e^{\tau_0 n}\sum_{k\in \Z} [T,\chi_N](n,k) u(k)\right)^2 \right)^{\frac12} \\
			                       & =\left( \sum_{N-R<n\leq N} \left(e^{\tau_0 n}\sum_{k\notin [-N,N]} f(n-k) u(k)\right)^2 \right)^{\frac12} \\
			                        &\leq \left( \sum_{N-R<n\leq N} \left(e^{\tau_0 n}\sum_{N<k\leq N+R} f(n-k) u(k)\right)^2 \right)^{\frac12} \\
			                          & \qquad \qquad + \left( \sum_{N-R<n\leq N} \left(e^{\tau_0 n}\sum_{k<-N \ {\rm or} \ k>N+R} f(n-k) u(k)\right)^2 \right)^{\frac12} \\  
		   \end{align*}
		   The first term on the right-hand side of the above estimate can be bounded by 
		   \begin{align*}
			   e^{\tau_0 N}  \cdot  & \sup_{N<k\leq N+R} |u(k)| \cdot \left( \sum_{N-R<n\leq N} \left(\sum_{N<k\leq N+R} |f(n-k) |\right)^2 \right)^{\frac12}\\
                            &\leq \sqrt{R} e^{\tau_0 N}  \cdot \| f\|_1\cdot   \sup_{N<k\leq N+R} |u(k)| \ \leq \frac{1}{2}e^{2\tau_0 N}\cdot  \sup_{N<k\leq N+R} |u(k)|,
		   \end{align*}
		   and the second term can be bounded by 
		    \begin{align*}
			   e^{\tau_0 N}   & \| u\|_{\infty} \cdot \left( \sum_{N-R<n\leq N} \left(\sum_{|k|\geq R} |f(k) |\right)^2 \right)^{\frac12} \lesssim_T \wtB \cdot\sqrt{R} e^{\tau_0 N-R^{\alpha}}   \leq \frac{1}{100}B.
		   \end{align*}
		   Therefore, we obtain 
		   \begin{equation}\label{region 3}
			\left\|\chi_{\{N-R<n\leq N\}} \cdot e^{\tau_0 n} [T,\chi_N] u \right\|_2 \leq \frac{1}{2}e^{2\tau_0 N}\cdot  \sup_{N<k\leq N+R} |u(k)| +\frac{1}{100}B.
		   \end{equation}
    \item If $N<n\leq N+R$, then a similar estimate to \eqref{region 3} also yields 
            \begin{equation}\label{region 4}
			\left\|\chi_{\{N<n\leq N+R\}} \cdot e^{\tau_0 n} [T,\chi_N] u \right\|_2 \leq \frac{1}{2}e^{2\tau_0 N}\cdot  \sup_{N-R\leq k\leq  N} |u(k)| +\frac{1}{100}B.
		   \end{equation}
	\item If $n>N+R$, then 
           \begin{align*}
			    \left\|\chi_{\{n>N+R\}} \cdot e^{\tau_0 n} [T,\chi_N] u \right\|_2 & =\left( \sum_{n>N+R} \left(e^{\tau_0 n}\sum_{k\in \Z} [T,\chi_N](n,k) u(k)\right)^2 \right)^{\frac12} \\
			                         & =\left( \sum_{n>N+R} \left(e^{\tau_0 n}\sum_{k\in [-N,N]} f(n-k) u(k)\right)^2 \right)^{\frac12} \\
              			            &\lesssim_T \| u\|_{\infty} \cdot  \left(\sum_{n>N+R} e^{2\tau_0 n} \left(\sum_{k\in [-N,N]} e^{-|n-k|^{\alpha}} \right)^2  \right)^{\frac{1}{2}} \\
		                			            &\lesssim_T \wtB N \cdot  \left(\sum_{n>N+R} e^{2\tau_0 n -2|n-N|^{\alpha}}   \right)^{\frac{1}{2}} 
		   \end{align*}
		   Since we take $R\sim_{T,B} N^{\frac{1}{\alpha}}$, we can ensure that $|n-N|^{\alpha} > 2\tau_0 n$ holds for any $n>N+R$. Therefore, 
		    \begin{align}\label{region 5}
			  \notag  \left\|\chi_{\{n>N+R\}} \cdot e^{\tau_0 n} [T,\chi_N] u \right\|_2  &\lesssim_T \wtB N \cdot  \left(\sum_{n>N+R} e^{-2\tau_0 n }   \right)^{\frac{1}{2}} \\
				                                  &\lesssim_T \wtB N e^{-\tau_0 N} \leq \frac{1}{100}B.
		   \end{align}
\end{itemize}

Combining \eqref{weight larger than B}, \eqref{region 1}, \eqref{region 2}, \eqref{region 3}, \eqref{region 4} and \eqref{region 5} together yields 
\begin{equation*}
	B\leq \frac{3}{100}B +e^{2\tau_0 N } \cdot \sup_{N-R\leq k\leq  N+R} |u(k)|
\end{equation*}
for all $N \gg_{T,B,\wtB} 1$ and $R\sim_{T,B} N^{\frac{1}{\alpha}}$. This is equivalent to saying that, there exist constants $C_1=C_1(T,B),\ C=C(T,B)$ and $N_0=N_0(T,B,\wtB)\gg 1$ such that 
\begin{equation}\label{right cone prop for super-exp decay}
	\sup_{N\leq k < N+C_1 N^{\frac{1}{\alpha}}} |u(k)|\geq C^{-N} |u(0)|=C^{-N}, \ \forall N\geq N_0.
\end{equation}
Let $N_{s+1}= \lceil N_{s}+C_1 N_s^{\frac{1}{\alpha}}\rceil, \ s=0,1,2,\dots$. Then \eqref{right cone prop for super-exp decay} ensures that in each interval $[N_s,N_{s+1})$ there is a point $x_s$ such that $|u(x_s)|\geq C^{-N_s}$. Therefore, for any $L\geq L_0\gg N_0$, let $K$ be the unique index such that
\[N_K\leq L <N_{K+1},\]
and we obtain  
\begin{equation}\label{right direction UC for super-exp decay}
			\#\Bigl\{ x\in[0,L] : |u(x)| > C^{-L} \Bigr\} \;\ge\; K.
\end{equation}

Finally, to estimate $K$, let $g_{\alpha}(t)=t^{1-\frac{1}{\alpha}},\ t>0$. By the mean value theorem, there exists $\vartheta$ with $N_s<\vartheta<N_{s+1}<N_s+2C_1 N_s^{\frac{1}{\alpha}}$ such that 
\begin{align*}
g_{\alpha}(N_{s+1}) - g_{\alpha}(N_s) &= (N_{s+1}-N_s)\cdot g'_{\alpha}(\vartheta) \\
&\leq 2 C_1 N_s^{\frac{1}{\alpha}} \cdot \left(1-\frac{1}{\alpha}\right) \frac{1}{\vartheta^{\frac{1}{\alpha}}} \\
&\leq 2C_1 N_s^{\frac{1}{\alpha}} \cdot \left(1-\frac{1}{\alpha}\right) \frac{1}{N_s^{\frac{1}{\alpha}}} = 2\left(1-\frac{1}{\alpha}\right)C_1 .
\end{align*}
Set $C_2 = C_2(T,B) = 2\left(1-\frac{1}{\alpha}\right)C_1$. Then we have 
\begin{equation}\label{Ns increasing rate}
g_{\alpha}(N_{s+1}) - g_{\alpha}(N_s) \leq C_2.
\end{equation}
From \eqref{Ns increasing rate} we obtain 
\begin{equation}\label{lower bound on K}
K+1 \geq \frac{1}{C_2}\bigl(g_{\alpha}(N_{K+1}) - g_{\alpha}(N_0)\bigr) \geq \frac{1}{2C_2} g_{\alpha}(L) + 1.
\end{equation}
Let $\epsilon = \epsilon(T,B) = \frac{1}{2C_2}$. Combining \eqref{right direction UC for super-exp decay} and \eqref{lower bound on K} yields \eqref{transversality for super-exponential decay}.

\end{proof}

\section{Proof of Theorem \ref{R net theorem for general hopping with unique minimum}}\label{appendix: proof of R net argument}
\begin{proof}[Proof of Theorem \ref{R net theorem for general hopping with unique minimum}]
Denote by 
\[
g(\theta)=\sum_{i=1}^{d}(2-2\cos (2\pi \theta_i))
\]
the Fourier symbol of the free Laplacian on $\Z^d$. Under our assumptions on $\hatf(\theta)$, one can prove that there exists $\kappa=\kappa(T,d)>0$ such that 
\begin{equation}\label{compare hatf with Laplacian}
	\hatf(\theta) \geq \kappa \cdot g(\theta), \qquad \forall \theta \in \T^d.
\end{equation}
(Obviously this fails when $\hatf$ has multiple minima.) Inequality \eqref{compare hatf with Laplacian} implies that 
\[
T \geq \kappa \cdot \Delta_{\rm free},
\]
and thus, by the positivity of $V$ in \eqref{Bernoulli potential},
\begin{equation}\label{compare H with Laplacian}
T+\lambda V \geq \kappa \Delta_{\rm free} + \lambda V \geq (\kappa\wedge 1) \cdot \left( \Delta_{\rm free} +\lambda V \right).
\end{equation}
Assume that $\{V=1\}$ is an $R$-net in $\Lambda_{N_0}$. Then, by \eqref{compare H with Laplacian}, the min-max principle, and Theorem \ref{R net theorem from DS20}, we obtain 
\begin{equation}\label{principal eigenvalue for general hopping}
	E_{\rm prin}(H_{N_0}) \geq (\kappa \wedge 1 )\cdot E_{\rm prin} \left(  (\Delta_{\rm free}+\lambda V)_{N_0}\right) \geq  (\kappa \wedge 1 ) L(\lambda,R,d),
\end{equation}
with $L(\lambda,R,d)$ defined as in \eqref{principle eigenvalue}. This proves the lower bound on the principal eigenvalue.\\

To obtain the decay of the Green's function, we apply the Combes--Thomas estimate. From \eqref{principal eigenvalue for general hopping}, for every $E\in [0,(\kappa \wedge 1 ) L(\lambda,R,d) /2]$ near the spectral edge, we have 
\begin{equation}\label{formula B.4}
	\| G_{N_0}(E)\|=\| (H_{N_0}-E)^{-1} \| \leq \frac{1}{|E-E_{\rm prin}(H_{N_0})|} \lesssim_{T,d} L(\lambda,R,d)^{-1}.
\end{equation}
Now for any fixed $y\in \La_{N_0}$, we take the weight function $\Phi(n)=\alpha |n-y|$ with $\alpha$ to be determined. Let 
\[
M_{\Phi}u(n)=e^{\Phi(n)}u(n), \qquad H_{\Phi}:=M_{\Phi} H_{N_0} M_{\Phi}^{-1}=M_{\Phi} T_{N_0} M_{\Phi}^{-1} +\lambda V_{N_0}.
\]
A simple computation shows 
\[
(H_{\Phi}-H_{N_0})(m,n)=f(m-n)\bigl( e^{\Phi(m)-\Phi(n)}-1\bigr).
\]
Since 
\[
|\Phi(m)-\Phi(n)| =\alpha \big| |m-y|-|n-y| \big| \leq \alpha |m-n|
\]
and $T$ satisfies \eqref{exp decay hopping}, if we take $\alpha<c_T/2$, the Schur test yields 
\begin{equation}\label{Schur test on difference of H and rotated H}
	\| H_{\Phi}-H_{N_0}  \| \leq \sum_{k\in \Z^d} |f(k)|\cdot (e^{\alpha|k|}-1) \lesssim_{T,d} \alpha. 
\end{equation}
Thus, by taking $\alpha= c L(\lambda,T,d)$ with some sufficiently small constant $0<c=c(T,d)$, \eqref{Schur test on difference of H and rotated H} together with \eqref{formula B.4} ensures that 
\begin{equation}\label{formula B6, Neumann}
	\| H_{\Phi}-H_{N_0}  \| \cdot  \| G_{N_0}(E)\| \leq \frac{1}{2}.
\end{equation}
Therefore, \eqref{formula B6, Neumann} and the Neumann series expansion 
\[
(H_{\Phi}-E)^{-1}= G_{N_0}(E)\cdot \sum_{s\geq 0}  \left[  (H_{N_0}-H_{\Phi}) G_{N_0}(E)  \right]^s
\]
yield 
\[
\| (H_{\Phi}-E)^{-1} \|\leq 2 \| G_{N_0}(E)\| \lesssim_{T,d} L(\lambda,T,d)^{-1}.
\]
Finally, since $M_{\Phi}^{-1}(H_{\Phi}-E)^{-1} M_{\Phi}= G_{N_0}(E)$, we have for any fixed $y$ and any $x\in \La_{N_0}$,
\begin{align*}
	|G_{N_0}(E)(x,y)|& =e^{-\Phi(x)+\Phi(y)} |(H_{\Phi}-E)^{-1} (x,y)| \\
	                 &\leq e^{-\alpha |x-y|} \cdot \|(H_{\Phi}-E)^{-1} \|  \\
					 &\lesssim_{T,d} L(\lambda,R,d)^{-1} \exp \left\{ -c\,L(\lambda,R,d) |x-y| \right\}. 
\end{align*}
This proves \eqref{principle Green's function}.

\end{proof}

\section{A quantitative uncertainty principle}\label{appendix: uncertainty principle}
The following quantitative uncertainty principle on finite discrete abelian groups was first proved in one dimension in \cite{Klopp98}, and later generalized to higher dimensions in \cite{Klopp02}.
Let $\Z_{2N+1}:=\Z/(2N+1)\Z$, and the discrete Fourier transform on the finite abelian group $\Z_{2N+1}^d$ is
\begin{align*}
  \mcF_N: \ \ell^2(\Z_{2N+1}^d)&\to \ell^2(\Z_{2N+1}^d), \\
     a= (a_n)_{n\in \Z_{2N+1}^d }& \longmapsto \mcF_N a  := \hat{a},
\end{align*}
where 
\begin{equation}\label{discrete Fourier transform}
  \hat{a}_l=(\mcF_N a)_l= \sum_{n\in \Z_{2N+1}^d } a_n\cdot  \frac{1}{(2N+1)^{d/2}}e^{-2\pi i l \cdot \frac{n}{2N+1}}. 
\end{equation}
The quantitative uncertainty principle indicates that, if $a$ is supported in a $K$-size block in $\Z^d_{2N+1}$, then $\hat{a}$ can be nearly constant in a $\frac{N}{K}$-size block.  
More precisely, we have 
\begin{lem}\label{discrete uncertainty principle}
\textup{(\cite[Lemma 6.2]{Klopp02})} Assume $N, L, K, K', L'$ are positive integers such that
\begin{itemize}
    \item $2N + 1 = (2K + 1)(2L + 1) = (2K' + 1)(2L' + 1)$; 
    \item $K < K'$ and $L' < L$.
\end{itemize}
Let $a = (a_n)_{n \in \mathbb{Z}_{2N+1}^d} \in \ell^2(\mathbb{Z}_{2N+1}^d)$  satisfy  $\supp (a) \subset \La_{K}$. Then there exists some $b \in \ell^2(\mathbb{Z}_{2N+1}^d)$ such that
\begin{enumerate}    
	\item $\|a\|_{\ell^2(\mathbb{Z}_{2N+1}^d)} = \|b \|_{\ell^2(\mathbb{Z}_{2N+1}^d)}$;
    \item $\|a - b\|_{\ell^2(\mathbb{Z}_{2N+1}^d)} \leq C_{K, K'} \|a\|_{\ell^2(\mathbb{Z}_{2N+1}^d)},$  where $0<C_{K, K'} \overset{K/K' \to 0}{\sim} K/K'$;
    \item For $l' \in \mathbb{Z}_{2L'+1}^d$ and $k' \in \mathbb{Z}_{2K'+1}^d$, we have $\hat{b}_{l'+k'(2L'+1)} = \hat{b}_{k'(2L'+1)}$.
\end{enumerate}
\end{lem}

\section{A coupling lemma}\label{appendix: coupling lemma}

The following coupling lemma is standard in multi-scale analysis. It was first proved in \cite[Lemma 2.4]{BK05} (and the remark below that lemma), and it can be generalized to the case of long-range hopping with exponential decay on $\Z^d$ following the proof in \cite[Lemma 3.1, Remark 3.6]{LSZ25}.

\begin{lem}\label{coupling lemma}
Fix $\ell_0 \sim \ell_2 ^{\mathfrak{a} }$, $1\ll \ell_2\ll \ell_1 \leq \frac{1}{2} \ell_0$ with $\mathfrak{a}>1$. The following result holds in $\Z^d$ (and hence in particular in the one-dimensional case). Let $\Lambda$ be a $\ell_0$-size block and let $E\in \R$ be a fixed energy. Assume $\mathfrak{R} \subset \mathfrak{R}' \subset \Lambda\subset \Z^d$ satisfies the following:
\begin{itemize}
  \item $\mathfrak{R}$ is a union of at most $K$ many $\ell_2$-size blocks $\Lambda_2'$;
  \item $\mathfrak{R}'$ is a union of $\ell_1$-size blocks $\Lambda_1'$, such that for every $\Lambda_2'\in \mathfrak{R}$, there exists $\Lambda_1'\in \mathfrak{R}'$ with 
                          \[\Lambda'_2\subset \Lambda'_1, \ \dist(\Lambda'_2,\Lambda\setminus \Lambda'_1) \ge \frac{1}{8}\ell_1. \] 
   Moreover, distinct elements of $\mathfrak{R}'$ are separated by distance at least $\gtrsim \ell_1$.
\end{itemize}
 Assume there is a family $\mathfrak F=\{\Lambda':\ \Lambda'\subset \Lambda\}$ of $\ell_2$-size good blocks
covering $\Lambda\setminus \mathfrak{R}'$, such that for each $n\in \Lambda\setminus \mathfrak{R}$, there exists $\Lambda'\in \mathfrak F$ satisfying 
      \begin{equation}
        Q_{\ell_2/10}(n)\cap \Lambda \subset \Lambda'.
      \end{equation}
Here ``good'' means that
\begin{align}
  \| G_{\Lambda'}(E)\| &<\exp \{  \ell_2^{1-\sigma}\}, \\
  |G_{\Lambda'}(x,y;E)| &<\exp \{-\gamma_{\ell_2}  |x-y|\} \quad \text{for all } |x-y|\ge \frac{\ell_2}{200}. 
\end{align}
Assume further that 
  \begin{equation}\label{coupling lemma wegner}
           \| G_{\Lambda_1'} (E)\| \le   \exp\{(\ell_1)^{1-\sigma}\} \quad \text{for all } \Lambda_{1}' \in \mathfrak{R}'. 
  \end{equation}
Then
\begin{align}
\label{annuls L2 norm}  \| G_{\Lambda}(E)\| &<\exp \{  \ell_0^{1-\sigma}\}, \\
\label{annuls off-diagonal decay}  |G_{\Lambda}(x,y;E)| &<\exp \{- \gamma_{\ell_0} |x-y|\} \quad \text{for all } |x-y|\ge \frac{\ell_0}{200}
\end{align}
for some $\gamma_{\ell_0}\ge \gamma_{\ell_2}-\mcO_{+}(\ell_2^{-c})$ with $c= (\mathfrak{a}-1)\wedge \sigma$.
\end{lem}

\section{A rank-one perturbation lemma}
The following is a key lemma in \cite{DS20} for handling the movement of eigenvalues.

\begin{lem}\label{rank-one perturbation}
\textup{(\cite[Lemma 5.1]{DS20})} Suppose that  the real symmetric matrix \( A \in  \mathbb{R}^{n\times n}  \) has eigenvalues  
\[
\lambda_1 \geq \lambda_2 \geq \cdots \geq \lambda_n \in \mathbb{R}
\]  
with orthonormal eigenbasis \( v_1, v_2, \cdots, v_n \in \mathbb{R}^n \). Then for every $\beta>0$, there is some   $0<c\ll1$ (depending only  on $\beta$) such that, if  
\begin{enumerate}
    \item \(0 < r_1 < r_2 < r_3 < r_4 < r_5 < 1\),  
    \item \(r_1 \leq c \min\{r_3 r_5, r_2 r_3 / r_4\}\),
    \item \(0 < \lambda_j \leq \lambda_i < r_1 < r_2 < \lambda_{i-1}\),  
    \item \(v_{j}^2(x) \geq r_3\),  
    \item \(\displaystyle\sum_{r_2 < \lambda_s < r_5} v_{s}^2(x) \leq r_4\),  
\end{enumerate}
then 
\[
\operatorname{trace} \mathbf{1}_{[r_1, \infty)}(A) < \operatorname{trace} \mathbf{1}_{[r_1, \infty)}(A + \beta e_x \otimes e_x),
\]
where \( e_x \in \mathbb{R}^n \) is the \(x\)-th standard basis element.

\end{lem}

\section*{Acknowledgement}
This work  is  supported by  the National Key R\&D Program
of China under Grant 2023YFA1008801.
 Y. Shi is supported by the NSFC(12522110) and  Z. Zhang is  supported by the NSFC (12288101).   

\section*{Data Availability}
		The manuscript has no associated data.
\section*{Declarations}
		{\bf Conflicts of interest} \ The authors  state  that there is no conflict of interest.

\newcommand{\etalchar}[1]{$^{#1}$}

\end{document}